\documentclass[11pt]{amsart}
\usepackage{amsmath, amssymb, amscd, mathrsfs, slashed, url}

\usepackage{enumitem}
\usepackage{setspace}
\usepackage{color}
\usepackage{extsizes}
\usepackage{xcolor}
\usepackage[all,cmtip]{xy}
\usepackage{multicol}
\usepackage{indentfirst}
\usepackage{latexsym}
\usepackage{bm}
\usepackage{graphicx}
\usepackage{subfigure,esint}
\usepackage{float}
\usepackage{verbatim}
\usepackage{comment}
\usepackage[top=1in, bottom=1in, left=1.25in, right=1.25in]{geometry}
\usepackage{epsfig,dsfont,amsthm,amsfonts,amsbsy}
\usepackage[colorlinks]{hyperref}
\usepackage[toc,page]{appendix}
\usepackage{tikz-cd}
\usepackage{amsthm,thmtools,xcolor}
\usepackage{bbm}

\usepackage{adjustbox}
\usepackage{enumitem}
\setlist{nosep}
\setlist[itemize]{leftmargin=*}
\setlist[enumerate]{leftmargin=*,align=left,nolistsep}
\newtheorem{theorem}{Theorem}[section]

\newtheorem{corollary}[theorem]{Corollary}

\newtheorem{definition}[theorem]{Definition}
\newtheorem{question}[theorem]{Question}

\newtheorem{lemma}[theorem]{Lemma}

\newtheorem{proposition}[theorem]{Proposition}

\theoremstyle{definition}
\newtheorem{example}[theorem]{Example}
\newtheorem*{remark}{Remark}

\newcommand{\vp}{\mathfrak{p}}

\newcommand{\vn}{\mathfrak{n}}

\newcommand{\vq}{\mathfrak{q}}
\newcommand{\sph}{\mathbb{S}^2}
\newcommand{\lb}{\mathcal{I}_{\vp}}
\newcommand{\lbt}{\mathcal{I}_{\vp(t)}}
\newcommand{\lbc}{\mathcal{I}_{\vp}^{\mathbb{C}}}
\newcommand{\lbct}{\mathcal{I}_{\vp(t)}^{\mathbb{C}}}
\newcommand{\Rr}{\mathbb{R}}

\newcommand{\re}{\mathfrak{R}}

\newcommand{\dis}{\mathrm{dist}}
\newcommand{\ord}{\mathrm{ord}}
\newcommand{\sob}{\mathcal{L}_1^2(\lb)}
\newcommand{\sobc}{\mathcal{L}_1^2(\lbc)}

\newcommand{\SO}{\mathrm{SO}(3)}

\newcommand{\osb}{H_{-}^1}
\newcommand{\rst}{\Sigma_t}

\newcommand{\Cn}{\mathcal{C}_{2n}}
\newcommand{\ccn}{\overline{\mathcal{C}}_{2n}}
\newcommand{\mul}{\mathrm{mul}}
\newcommand{\ei}{\mathbf{e}}

\newcommand{\lan}{\langle }
\newcommand{\ran}{\rangle}

\newcommand{\Tr}{\mathrm{Tr}}

\newcommand{\pa}{\partial}

\newcommand{\Ric}{\mathrm{Ric}}
\newcommand{\ME}{\mathcal{E}}

\newcommand{\na}{\nabla}

\newcommand{\RR}{\mathbb R}
\newcommand{\CC}{\mathbb C}

\newcommand{\lam}{\lambda}

\newcommand{\MI}{\mathcal{I}}
\newcommand{\ZT}{\mathbb{Z}_2}

\newcommand{\dist}{\mathrm{dist}}

\newcommand{\ML}{\mathcal{L}}

\newcommand{\LS}{\lesssim}

\newcommand{\mbR}{\mathbb{R}}

\newcommand{\mbC}{\mathbb{C}}
\usepackage{calligra}
\DeclareMathAlphabet{\mathcalligra}{T1}{calligra}{m}{n}

\declaretheoremstyle[
headfont=\color{blue}\normalfont\bfseries,
bodyfont=\color{blue}\normalfont\itshape,
]{colored}

\usepackage[utf8]{inputenc}
\usepackage[english]{babel}
\usepackage{fancyhdr}
\usepackage{accents}

\title{On the existence and rigidity of critical $\mathbb{Z}_2$ eigenvalues}

\begin{document}

\author{Jiahuang Chen, Siqi He} 

\address{AMSS}
\email{chenjiahuang23@mails.ucas.ac.cn,sqhe@amss.ac.cn }	

\maketitle
\begin{abstract}
In this article, we study the eigenvalues and eigenfunction problems for the Laplace operator on multivalued functions, defined on the complement of the 2n points on the round sphere. These eigenvalues and eigensections could also be viewed as functions on the configuration spaces of points, introduced and systematically studied by Taubes-Wu. Critical eigenfunctions, which serve as local singularity models for gauge theoretical problems, are of particular interest. 

Our study focuses on the existence and rigidity problems pertaining to these critical eigenfunctions. We prove that for generic configurations, the critical eigenfunctions do not exist. Furthermore, for each $n>1$, we construct infinitely many configurations that admit critical eigensections. Additionally, we show that the Taubes-Wu tetrahedral eigensections are deformation rigid and non-degenerate. Our main tools are algebraic identities developed by Taubes-Wu and finite group representation theory.
\end{abstract}
\section{Introduction}
Let $\Cn$ be the space of unordered $2n$ distinct points of $\sph$, for each $\vp\in \Cn$, there exists a unique flat $\mbR$ bundle $\MI_{\vp}\to \sph\setminus \vp$, which has monodromy $-1$ on any embedded circle in $\sph\setminus \vp$ linking any given point from $\vp$. A smooth section $f\in C^{\infty}(\MI_{\vp})$ could also be considered as a two valued function over $\sph\setminus \vp$. On the other hand, when considering the $2$-fold branched covering, namely $\pi:\Sigma_\vp\to \sph$, determined by the same monodromy, a section $f$ can be lifted to an odd function $\tilde{f}$ defined on the surface $\Sigma_\vp$. 

As $\MI_{\vp}$ is a flat bundle, using the flat structure, the Laplace operator $\Delta_{\sph}$ for the round metric on $\sph$ naturally extends to $\MI_{\vp}$. And a section $f\in C^{\infty}(\MI_{\vp})$ is called a $\ZT$ eigensection if $-\Delta_{\sph} f=\lam f$ and $\int_{\sph\setminus \vp}|df|^2$ is bounded. In this case, $\lam$ is called a $\ZT$ eigenvalue of $\lb$. The round metric $\mathrm{ds}^2$ of $\sph$ pulled back to be a conical metric with cone angle $4\pi$ at each configuration point. Thus a $\ZT$ eigensection can be characterized as an eigenfunction $\tilde{f}$ of the singular Laplacian $\Delta_{\pi^\ast\mathrm{ds}^2}$ and $\lam$ is its eigenvalue. Conical surfaces $(\Sigma_\vp,\pi^\ast\mathrm{ds}^2)$ and their Laplacian spectra are of particular interest in mathematical physics \cite{Landau} and spectral geometry \cite{MR,NayataniI,shin2019metrics}.

For a $\ZT$ eigensection $f$, Taubes and Wu \cite[Proposition 2.1]{taubeswu2020examples} show that $|f|$ extends to a H\"older continuous function on $\sph$ and vanishes at $\vp$. Moreover, a $\ZT$ eigensection with $|df|$ extending to a H\"older continuous function on $\sph$ and vanishing at $\vp$, is the central object of the work of Taubes and Wu, as well as this article. We name such eigensections as \textbf{critical eigensections}. If for a configuration $\vp\in \Cn$, there exists a critical eigensection, then $\vp$ is called a critical configuration. 
The motivation for studying critical eigensections stems from gauge theory and calibrated geometry, where they can be interpreted as local model of $\ZT$ harmonic $1$-forms and spinors. These concepts were initially introduced and emphasized by Taubes \cite{taubes2013compactness,taubes2014zero,Taubescompactness}, leading to various important follow-up studies \cite{haydyswalpuski2015compactness, taubes2017behavior,  zhang2017rectifiability,doan2017existence,walpuskizhang2019compactness,sun2022z2,donaldson2022closed,parker2023concentrating}.

Given a critical eigensection $f$ with eigenvalue $\lam$, let $R$ be the radius in $\mbR^3$. By constructing a $\ZT$ harmonic 1-form $v=d(R^{\mu}\pi^*f)$, where $\pi:\mbR^3\setminus {0}\to \sph$ is the projection to the unit sphere and  $\mu=(-1+(1+4\lambda)^{1/2})/2$, one obtains a local model characterized by $2n$ distinct rays from the origin in $\mbR^3$ intersecting $\sph$ at $\vp$. Taubes \cite{taubes2014zero} showed that tangent cones at zeros of harmonic $1$-forms defined on a $3$-manifold are always given by critical $\ZT$ eigensections.

Assuming that the zero loci of $\ZT$ harmonic $1$-forms or spinors are smooth, fascinating results have been proven \cite{takahashi2015moduli,indextakahashi,donaldsondeformation2019,parker2024gluing}. However, it is not a priori true that zero loci are submanifolds.  In \cite{taubeswu2020examples}, Taubes-Wu construct critical configurations for $n=1,2,4,6,10$, using the symmetry groups of the Platonic solids, indicating the existence of singular zero loci. Additionally, an index theorem has been developed for $\ZT$ harmonic spinors with graphic singularities in \cite{haydys2023index}.

To understand the local structure of $\ZT$ harmonic $1$-forms and spinors, and follows from Taubes-Wu \cite{taubeswu2020examples,taubeswu2021topological}, fundamental questions arise regarding the existence of such singularity models:
\begin{question}
For each $n$, does a critical configuration exist in $\Cn$?
\end{question}
It's worth mentioning that a brilliant min-max approach to the existence problem was introduced in Taubes-Wu \cite{taubeswu2021topological}.

In this paper, we explore a family of configurations with dihedral group symmetry for each $n>1$. We observe that for these configurations, the existence of two summands of different one dimensional irreducible representations of the dihedral group, implies the existence of critical eigensections. Combining finite group representation theory with a spectral flow argument, we establish the existence of an infinite number of critical configurations within this family. Specifically, we prove:
\begin{theorem}
For each $n>1$, there exist infinitely many critical configurations in $\Cn$.
\end{theorem}

Another question arises from a talk by Taubes and the recent paper \cite[Page 4]{haydys2023new}, suggesting an expectation of a limited number of critical configurations. Using the deformation formula developed by Taubes-Wu \cite{taubeswu2021topological}, we establish that for generic configurations, all eigenvalues must have multiplicity one, contrary to the fact that critical eigenvalues must have higher multiplicity. 

\begin{theorem}
    For each $n$, generic $\vp\in \Cn$ is non-critical. 
\end{theorem} 

The most intriguing critical eigensections arise from the symmetric group of tetrahedron, which we name as Taubes-Wu tetrahedral eigensections. These eigensections are closely related to the theory of Harvey-Lawson cones and the problem of extremal eigenvalues \cite{Nadirashvili}, suggesting similar rigidity properties to those in \cite{Haskins2004}. 

On the other hand, a critical eigensection is called non-degenerate if the first leading coefficients are non-vanishing near each configuration point. The non-degenerate condition for $\ZT$ harmonic $1$-forms and spinors are also geometrically in great interest, which has been studied wildly \cite{donaldsondeformation2019,he2023branched,parker2023deformations,takahashi2015moduliZ2,parker2024gluing}. In particular, we establish the following results, indicating that Taubes-Wu tetrahedral eigensections are indeed geometrically interest in these viewpoints. 

\begin{theorem}
    All Taubes-Wu tetrahedral eigensections are deformation rigid and non-degenerate.
\end{theorem}

Additionally, we conduct a comprehensive study of the algebraic identities developed by Taubes-Wu \cite{taubeswu2021topological} on configurations in $\mathcal{C}_4$. Using the $\mathbb{Z}_2^{\oplus 3}$ symmetry, we obtain the following restriction on critical configurations:

\begin{theorem}
    Over $\mathcal{C}_4$, suppose the configuration $\vp=\{p_1,-p_1,p_2,p_3\}$ for some $p_1,p_2,p_3\in\sph$, i.e. $\vp$ contains a pair of antipodal points, then $\vp$ is non-critical. 
\end{theorem}
This result yields an intriguing corollary suggesting limitations on the singularity of $\ZT$ harmonic 1-forms over 3-manifolds: the tangent cone of a singularity must not be two transverse intersecting crossings lines. Consequently, the zero locus of a $\ZT$ harmonic $1$-form over $3$-manifold can not be a union of two knots that transverse intersect with each other. However, the critical configuration that is the vertices of cube, which was constructed in \cite{taubeswu2020examples}, indicates that four curves that intersect at one point transversally is possible.

\textbf{Acknowledgements.} 
	The authors wish to express their gratitude to a great many people for their interest and helpful comments. Among them are Mark Haskins, Kaihan Lin, Zhenhua Liu, Clifford Taubes and Rafe Mazzeo. Part of this article is based on the bachelor thesis of J.~Chen, completed at Sun Yat-sen University. S.~He is supported by National Key R\&D Program of China (No.2023YFA1010500) and NSFC grant (No.12288201). J.~Chen is supported by National Key R\&D Program of China (No.2023YFA1010500).

%%%%%%%%%%%%%%%%%%%%%%%%%%%%%%%%%%%%%%%%%%%%%%%%%%%%%%%%%%%%

\section{$\ZT$ eigenvalues and $\ZT$ eigensections}
In this section, we will introduce some background of the theory of $\ZT$ eigensections, which had been comprehensively studied in Taubes-Wu \cite{taubeswu2020examples,taubeswu2021topological}.
\subsection{$\ZT$ eigensections and spectral theory}
Let $\Cn$ be the space of unordered $2n$ points of $\sph$, for each $\vp\in \Cn$, one has a monodromy representation $\rho:\pi_1(\sph\setminus\vp)\to\ZT$ such that any generator of $\pi_1(\sph\setminus\vp)$ represented by a loop surrounding one point of $\vp$ does not lie in $\ker\rho$. This representation determines a flat real line bundle $\lb\to \sph\setminus\vp$. Locally, restricting the flat real line bundle $\lb$ on any small loop around $p\in \vp$ yields a M\"obius line bundle. A section $f$ of $\lb$, called a $\ZT$-section, may be seen as a multivalued function on $\sph\setminus\vp$. Using the flat structure on $\lb$, the Laplace operator extends to $\Delta:C^{\infty}(\lb)\to C^{\infty}(\lb)$, where $C^{\infty}(\lb)$ is the space of smooth sections of $\lb$.  

To begin with, we will introduce the spectral theory of $\Delta$ on $\lb$, follows from \cite[Proposition 1.1]{taubeswu2021topological}. Let $C^{\infty}_c(\lb)$ be the set of smooth sections of $\lb$ that are compact supported on $\sph\setminus\vp$. We can define the $L^2$ norm for a section $f$ by $||f||^2_{L^2}=\int_{\sph\setminus\vp}|f|^2$, and define the Sobolev norm by  
\begin{equation}
    \begin{split} ||f||^2_{\ML_1^2}=\int_{\sph\setminus\vp}|f|^2+|df|^2.
    \end{split}
\end{equation}

The space $L_c^2(\lb)$ and $\sob$ are naturally defined as completions of $C_c^\infty(\lb)$ in $L^2$ norm and Sobolev norm respectively. A section $f\in \sob$ is called a $\ZT$ eigensetion if $\Delta f+\lam f=0$ is weakly satisfied and $\lam$ is called a $\ZT$ eigenvalue. Moreover, the eigenvalues have the min-max characterisation as usual, where the $k$-th eigenvalue of $\lb$ is given by 
\begin{equation}
\label{eq_minmaxZ2eigenvalues}
\lambda_k(\vp)=\inf_{H\subset \sob}\sup_{f\in H\setminus\{0\}}||df||^2_{L^2}/||f||^2_{L^2},
\end{equation}
where $H$ runs through $k$-dimensional subspaces of $\sob$. As a constant function is a section of $\lb$ if and only if the constant is $0$, the first $\ZT$ eigenvalue is always positive. 

The $\ZT$ eigenvalues as well as eigensections could be interpreted as usual eigensection problem over surfaces with conical metric. Let $\pi:\Sigma_{\vp}\to \sph$ be the double branched covering defined by $\ker\rho$. When $|\vp|=2n$, $\Sigma_{\vp}$ is a closed surface with genus $n-1$. Let $\mathrm{ds}^2$ be the round metric of $\sph$. Then $\pi^\ast\mathrm{ds}^2$ is a metric on $\Sigma_{\vp}$ with conical singularities $\vp$, whose cone angles are all $4\pi$. 

In addition, as $\pi:\Sigma_\vp\to \sph$ is a $2$-fold branched covering, there is a canonical involution $\ei$ on $\Sigma_{\vp}$ with fixed points $\vp$. It is clear that $\pi^\ast \lb$ is a trivial bundle and hence a section $f$ of $\lb$ can be identified as an odd function $\tilde{f}$ on $\Sigma_{\vp}\setminus \vp$, i.e. functions satisfying $\tilde{f}(\ei(x))=-\tilde{f}(x), ~\forall x\in \Sigma_{\vp}\setminus\vp$. 

Let $H^1(\Sigma_{\vp})$ denote the $\ML^{2}_1$ Sobolev space. The $\pm 1$ eigenvalues of the involution $\ei$ induce an orthogonal decomposition of $H^1(\Sigma_\vp)$ into $H^1_{+}(\Sigma_{\vp})\oplus H^1_{-}(\Sigma_{\vp})$, where $H^1_{-}(\Sigma_{\vp})$ comprises functions odd with respect to $\ei$. Furthermore, since the singular Laplacian $\Delta_{\pi^*ds^2}$, denoted by $\Delta$ for simplicity, commutes with the involution $\ei$, the spectrum of $\Delta$ on $H^1(\Sigma_{\vp})$ can be separated into two parts:
$$\mathrm{Spec}(\Delta) = \mathrm{Spec}(\Delta|_{H^1_+(\Sigma_\vp)}) \cup \mathrm{Spec}(\Delta|_{H^1_{-}(\Sigma_\vp)}) = \mathrm{Spec}(\sph) \cup \mathrm{Spec}(\lb),$$
where $\mathrm{Spec}(\sph)$ represents the Laplacian spectrum of the round sphere and $\mathrm{Spec}(\lb)$ consists of $\ZT$ eigenvalues of $\lb$. Consequently, the spectral theory of $\Delta$ on $\lb$ aligns with the general study of the conical metric, as detailed in \cite[Lemma 3.1]{de2021schauder} and \cite[Section 1]{taubeswu2021topological}. 

\begin{proposition}
\label{prop_descrte_spectrum}
Given $\vp\in \Cn$, the spectrum of Friedrichs' extension of $\Delta:C_c^\infty(\lb)\to C_c^\infty(\lb)$ to $\sob$ is discrete, with finite multiplicities and no accumulation points. The $\ZT$ eigensections provide an orthonormal basis for $L_c^2(\lb)$.
\end{proposition}

Define $\lambda_k(\vp)$ $(k\ge 1)$ to be the $k$-th $\ZT$ eigenvalue of $\lb$. We shall discuss the regularity of the function $\lam_k(-)$ in the next section.

\subsection{Local analytic theory of $\ZT$ eigensections}
In this subsection, we will introduce the local analytic theory of the $\ZT$ eigensections, follows from \cite[Section 4]{taubeswu2020examples}, \cite[Section 2]{taubeswu2021topological} and \cite[Section 5]{taubes2014zero}. We also refer \cite[Section 4]{haydys2023index} for various generalization.

Given $p\in \vp$, let $z=re^{i\theta}$ be the complex coordinate centered at $p\in \vp$ defined by the stereographic projection from $-p$, then the Laplacian could locally be expressed as $\Delta=(\frac{r^2+1}{2})^2(\pa_r^2+\frac1r\pa_r+\frac{1}{r^2}\pa_{\theta}^2).$ In this coordinate, let $f\in\sob$ be an eigensection, then Taubes and Wu got an asymptotic description of $f$ near $\vp$:
\begin{lemma}{\cite[Section 4]{taubeswu2020examples}}
For eigensection $f$, near $p$, we have the following local expansions $f(z,\bar{z})=\sum_{n=0}^{+\infty} \re\,(a_n(1+u_n(r))z^{n+\frac{1}{2}})$, where $u_n$ are real analytic functions, and $a_n$ are complex constants. 
\end{lemma}

\begin{remark}
    In this paper, we will always assume that a complex coordinate $z$ near some point $p$ in $\sph$ is derived from some seterographic projection from $-p$. This convention is applied throughout the paper. It's important to note that it only determines the complex coordinate up to a complex unit.
\end{remark}

\begin{definition}
\label{def_leading_coefficient_pf}
For each $p\in \vp$, let $\vn_p(f)$ be the first $n$ in the local expansion of $f$ with non-vanishing coefficients, we define the order of $f$, $\ord(f):=\min_{p\in\vp}\vn_p(f).$ In addition, we write $\vp_f=\{p\in\vp|\vn_p(f)=0\}.$
\end{definition}

Consider a vector field $J$ on $\sph$ that commutes with the Laplacian $\Delta$. Then, for a $\ZT$ eigensection $f$ of $\lb$, with eigenvalue $\lam$, $Jf$ satisfies $\Delta Jf+\lam Jf=0$ pointwise on $\sph\setminus\vp$ as well. However, it could happen that $\int_{\sph\setminus\vp}|dJf|^2$ is unbounded. Thus $Jf$ may not be a $\ZT$ eigensection of $\lb$. On the other hand, there are still asymptotic expansions for $Jf$ near $p\in\vp$ (cf. \cite[Lemma 1]{Victor}). Thus $\ord(Jf)$ and $\vn_p(Jf)$ are still well-defined.

Among all the $\ZT$ eigenvalues and eigensections, we have special interest in the following kind:
\begin{definition}
An eigensection $f$ of $\lb$ is called \textbf{critical} if $|df|$ can be extended to a H\"{o}lder continuous function on $\sph$ that vanishes at points of $\vp$. The corresponding eigenvalue is called a critical eigenvalue of $\lb$. %And the underlying configuration $\vp$ is called a critical configuration.
\end{definition}
According to the local expansion of $f$, it is critical if and only if $\vp_f=\varnothing$. The nomenclature ``critical" originates from Taubes-Wu's construction \cite{taubeswu2021topological}, where a functional $\mathcal{\ME}(\vp,\{\pm f\}):=\int_{\sph}|df|^2$ for $f\in \sob$ satisfying $||f||_{L^2}=1$, was introduced. Taubes-Wu proved that $f$ is critical if and only if the pair $(\vp,\{\pm f\})$ is a critical point for the functional $\mathcal{\ME}$.

A particularly intriguing subset of critical eigensections are the non-degenerate ones, which can be understood as generic critical eigensections. We refer \cite{donaldsondeformation2019,takahashi2015moduliZ2} for further interesting discussions towards this concept.

\begin{definition}
A critical eigensection $f$ of $\lb$ is called \textbf{non-degenerate} if for any $p\in \vp$, we have $\vn_p(f)=1$. 
\end{definition}

\section{$\ZT$ eigensections and their deformations}
In this section, we will study the behavior of $\ZT$ eigenvalues as the configurations vary. We introduce the variation formula proved in \cite{taubeswu2021topological} and establish a generic property of $\ZT$ eigensections based on this. In addition, we extend the continuity of eigenvalue functions $\lam_k(\vp)(k\ge 1)$ to the compactification of the configuration spaces, using the method of \cite{ABoper}.

\subsection{The rigidity problem of critical eigensections}
The configuration space $\Cn$  has a natural topology and can be considered as a complex manifold. Tangent space at a configuration $\vp\in\Cn$ can be identified as $\oplus_{p\in\vp}T_p\sph$. Let $\vec{v} \in T_{\vp}\Cn$ be a tangent vector. Given complex coordinates for the configuration points. Then, near a specific point $p \in \vp$, we can express $\vec{v}(p)$ as $v(p)\frac{\partial}{\partial z_p} + \bar{v}(p)\frac{\partial}{\partial \bar{z}_p}$, where $v(p) \in \CC$ and $z_p$ is the given complex coordinate. With this notation, we represent $\vec{v}$ as the $2n$-tuple $(v(p))_{p\in\vp}$.

\begin{definition}
    A configuration $\vp\in\Cn$ that admits a critical eigensection on $\lb$ is called a critical configuration.
\end{definition} 

Next, we introduce the concept of isolation of critical configurations within the space $\Cn$. We recognize that $\Cn$ naturally accommodates an $\SO$-action, which preserves the spectrum and asymptotic behavior of eigensections near $p\in\vp$. Hence, it becomes necessary to factor out this $\SO$-action.

Given a configuration $\vp\in\Cn$, where $\vp=\{p_1,\cdots,p_{2n}\}$, and $A\in\SO$, we denote the $\SO$ action as $A.\vp=\{A(p_1),\cdots,A(p_{2n})\}$. Let $\vp\in\Cn$, we define $\delta_0(\vp):=\min_{i,j}\dis(p_i,p_j)$. For any other configuration $\vq\in\Cn$, we define the distance between them  $\dis(\vp,\vq):=\max_i\min_j\dis(p_i,q_j)$.
\begin{definition}
  A critical configuration $\vp\in\Cn$ is called $M$-isolated for $M>0$, if there exists some $\epsilon>0$ such that for any $\{q_1,\cdots,q_{2n}\}=\vq\in\Cn$ that satisfies $\dis(\vp,\vq)<\epsilon$ and $\forall A\in\SO, A.\vp\ne\vq$, the line bundle $\mathcal{I}_\vq$ has no critical eigenvalues in $(0,M)$. A critical configuration is called isolated in $\Cn$ if it is $\infty$-isolated.
\end{definition} 

In Section \ref{newfamily}, we will introduce a method for constructing infinitely many new critical configurations. This newly constructed family of configurations raises the concern that critical configurations may not be $\infty$-isolated in general. Consequently, we may need to explore deformations of critical eigensections as well.
\begin{definition}
Let $\vp$ be a critical configuration with $f$ a critical eigensection with eigenvalue $\lam$. A critical deformation of this triple $(\vp,\lambda,f)$ consists of the following data: 
    \begin{enumerate}[label=(\roman*)]
        \item A smooth path $\vp(t)$ on $\Cn$, such that $\vp(0)=\vp$.
        \item A function of eigenvalues $\lambda(t)$, such that for each $t$, $\lambda(t)$ is a critical eigenvalue of $\mathcal{I}_{\vp(t)}$, and $\lambda(0)=\lambda$.
        \item A family of critical eigensections $f_t$, such that each $f_t$ satisfies $-\Delta f_t=\lam(t)f_t$. Additionally, for any smooth family of diffeomorphism $\Phi_t$ of $\sph$, such that $\Phi_t(\vp)=\vp(t)$, $t\mapsto \Phi_t^\ast f_t$ is a smooth path of sections in $\sob$.
    \end{enumerate}
A critical deformation is called trivial if the path $\vp(t)$ can be realized as $A_t(\vp)$, where $A_t$ is a smooth path in $\SO$. It could happen that $A_t^\ast f_t\ne f$.

A critial triple $(\vp,\lambda,f)$ is called \textbf{deformation rigid} if there exists no nontrivial critical deformations for it. 
\end{definition}

 It follows from the definition that for a critical triple $(\vp,\lambda,f)$, if $\vp$ is $M$-isolated for $M>\lambda$, then $(\vp,\lambda,f)$ is deformation rigid. As shown by Taubes-Wu, critical eigensections are critical values for a functional. It would be interesting to know whether the critical points are isolated or rigid. In particular, we would like to understand the following question:

\begin{question}
    Are the critical eigensections found in Taubes-Wu \cite{taubeswu2020examples} deformation rigid? 
\end{question}

We shall discuss this problem in the tetrahedral case in Section \ref{tetrahedralcase}.

\subsection{Continuity of eigenvalues}
Consider two distinct configurations $\vp,\vq\in\Cn$, for line bundles $\mathcal{I}_{\vp}$ and $\mathcal{I}_{\vq}$, we want to show that the difference between $k$-th eigenvalues $\lambda_k(\vp)$ and $\lambda_k(\vq)$ is dominated by $\dis(\vp,\vq)$. Consider the surfaces $\pi_{\vp}:\Sigma_\vp\to\sph$ and $\pi_{\vq}:\Sigma_{\vq}\to\sph$. Assume $\dis(\vp,\vq)=:\delta\ll\delta_0=\delta_0(\vp)$ and order $\vq$ in the way that $\mathrm{dist}(p_i,q_i)=\min_{j}\mathrm{dist}(p_i,q_j)$. Consider a cut-off function $\chi_i$ near a configuration point $p_i\in\vp$, supported on $B_{4\delta_1}(p_i)$ for $\delta<\delta_1\ll\delta_0$, and $\chi_i|_{B_{2\delta_1}(p_i)}\equiv 1, \chi_i|_{\sph\setminus B_{4\delta_1}(p_i)}\equiv 0$ and $|d\chi_i|<1/\delta_1<1/\delta$. We can fix a choice of $\chi_i$ for sufficiently small $\delta$ and assume any two of them are identical up to a rotation. Let's choose a diffeomorphism $\Phi:\sph\to\sph$ satisfying:(a) $\Phi(p_i)=q_i$; (b) in a complex coordinate $z$ near each $p_i$, $\Phi$ can be written as $z\mapsto z+\chi_iz(q_i)$. Note that the assumption $|d\chi_i|<1/\delta_1$ ensures that $z\mapsto z+\chi_iz(q_i)$ is injective.

Given such a diffeomorphism, we can lift it to be a diffeomorphism $\tilde{\Phi}:\Sigma_\vp\to\Sigma_\vq$ so that $\pi_{\vq}\circ\tilde{\Phi}=\Phi\circ\pi_\vp$. The pull-back metric $(\pi_{\vq}\circ\tilde{\Phi})^\ast\mathrm{ds}^2$ is defined on $\Sigma_\vp$ and possesses conical singularities $\vp$. Denote this metric by $g'$ and $\pi_\vp^\ast\mathrm{ds}^2$ by $g$.

\begin{lemma}
There are functions $C=C(\delta_0,\delta)=1+\mathcal{O}(\delta)>0$ and $C'=C'(\delta_0,\delta)=1+\mathcal{O}(\delta)>0$, depends only on $\delta_0$ and $\delta$, such that for any function $f\in H^1(\Sigma_\vp)$, we have
\begin{enumerate}[label=(\roman*)]
    \item $C^{-1}||f||^2_{L^2(g')}\le ||f||^2_{L^2(g)}\le C||f||^2_{L^2(g')}$,
    \item $C'^{-1}||df||^2_{L^2(g')}\le ||df||^2_{L^2(g)}\le C'||df||^2_{L^2(g')}$.
\end{enumerate}
\end{lemma}
\begin{proof}
      For (1), Since $||f||^2_{L^2(g)}=\int_{\Sigma_\vp}f^2\sqrt{g}dx$, it is sufficient to compare $\sqrt{g}$ and $\sqrt{g'}$. Fix $p_i\in\vp$, $q_i\in B_{2\delta_1}(p_i)\subset\sph$ and complex coordinate $z=u+iv$ near $p_i$. Let $z(q_i)=u_q+iv_q$ in this coordinate. We have $\mathrm{ds}^2=(\frac{2}{u^2+v^2+1})^2(du^2+dv^2)$ and 
    \begin{align*}
        (\Phi^\ast\mathrm{ds}^2)(u+iv)=&\left(\frac{2}{(u+\chi_iu_q)^2+(v+\chi_iv_q)^2+1}\right)^2\Big(((1+\partial_u\chi_i u_q)^2+(\partial_u\chi_i v_q)^2)du^2\\
        &+((\partial_v\chi_i u_q)^2+(1+\partial_v\chi_i v_q)^2)dv^2\\
        &+((1+\partial_u\chi_i u_q)\partial_v\chi_i u_q+(1+\partial_v\chi_i v_q)\partial_u\chi_i v_q)dudv\Big).
    \end{align*}
    Since $g=\pi_\vp^\ast\mathrm{ds}^2$ and $g'=\pi_\vp^\ast\Phi^\ast\mathrm{ds}^2$, we only need to compare $\mathrm{ds}^2$ with $\Phi^\ast\mathrm{ds}^2$. One can easily see that $Q_1\det(\Phi^\ast\mathrm{ds}^2)\le\det(\mathrm{ds}^2)\le P_1\det(\Phi^\ast\mathrm{ds}^2)$, where $P_1(u_q,v_q,\chi_i,d\chi_i)$ and $Q_1(u_q,v_q,\chi_i,d\chi_i)$ are (universal) rational polynomials, such that $$P_1(0,0,\chi_i,d\chi_i)=Q_1(0,0,\chi_i,d\chi_i)=P_1(u_q,v_q,0,0)=Q_1(u_q,v_q,0,0)=1.$$

    For (2), recall that $||df||^2_{L^2(g)}=\int_{\Sigma_\vp}g^{ij}\partial_if\partial_jf\sqrt{g}dx$. As $\mathrm{tr} g'/\mathrm{tr} g$ is a rational function in $u_q,v_q,\chi_i$ and $d\chi_i$, one can choose a constant $C_1=1+\mathcal{O}(\delta)$ so that $C_1^{-1}\mathrm{tr} g'\le\mathrm{tr} g\le C_1\mathrm{tr} g'$. Let $\lambda_1=\lambda_2$ be the eigenvalues of $g$ and $\lambda_1'\ge \lambda_2'$ be eigenvalues of $g'$. Equivalently $C_1^{-1}(\lambda_1'+\lambda_2')\le(\lambda_1+\lambda_2)\le C_1(\lambda_1'+\lambda_2')$. Additionally, we have shown $C^{-1}\lambda_1'\lambda_2'\le\lambda_1\lambda_2\le C\lambda'_1\lambda_2'$. Thus $\lambda_1'/\lambda_2\ge 2C_1^{-1}-\lambda_2'/\lambda_2\ge 2C_1^{-1}-\sqrt{g'}/\sqrt{g}\ge 2C_1^{-1}-C$ and $\lambda_2'/\lambda_1\le 2C_1-C^{-1}$.

    Therefore, we can choose some $C'=1+\mathcal{O}(\delta)$ so that 
    \[g^{ij}\partial_if\partial_jf=\frac{1}{\lambda_2}\sum_{i=1,2}|\partial_if|^2\ge \frac{\lambda_1'}{\lambda_2}(g')^{ij}\partial_if\partial_j f\ge C'^{-1}(g')^{ij}\partial_if\partial_jf\]  Similarly for the remained inequality. 
\end{proof}

 Combine the previous lemma with min-max characterization of $\ZT$ eigenvalues as in \eqref{eq_minmaxZ2eigenvalues}, we conclude:
\begin{proposition}
 \label{contiofeig}
 There is a function $C=C(\delta_0,\delta)=1+\mathcal{O}(\delta)>0$, which depends only on $\delta_0$ and $\delta$, such that $C^{-1}\lambda_k(\vq)\le\lambda_k(\vp)\le C\lambda_k(\vq)$.
\end{proposition}

Therefore, we obtain the following well-known fact:
\begin{proposition}\label{continuousordinary}
    The $k$-th eigenvalues $\lambda_k(\vp)$ depends continuously on $\vp$ in $\Cn$.
\end{proposition}  

\subsection{Variation and generic properties of $\ZT$-eigenvalues}
Given a compact manifold, for a generic Riemannian metric on it, eigenvalues of the Laplacian are all simple (cf. \cite{Uhlenbeck} or \cite{Bando}). This generic property holds in our case as well, implying that a generic configuration is non-critical, since a critical eigenvalue has multiplicity greater than two \cite[Lemma 2.7]{taubeswu2021topological}. 

The following lemma deals with the derivative of eigenvalues.

\begin{lemma}[\cite{taubeswu2021topological}, Proposition 2.6]\label{TWformula}
Fix a $\vp\in\Cn$, let $\lambda(\vp)$ be an eigenvalue of $\lb$. Choose a complex coordinate $z_p$ near each $p\in\vp$, a $\ZT$ eigensection $f$ has an asymptotic expansion near $p$. And the leading coefficients $a_p(f)$ are determined up to a sign under these choices. 

If $\mul\,\lambda(\vp) =1$, then there is a neighborhood $U$ of $\vp$ in $\Cn$, such that there is a smooth function $\lambda(-)$ defined on it and satisfies:
\begin{enumerate}[label=(\roman*)]
    \item The value of $\lambda{(-)}$ at $\vp$ is $\lambda(\vp)$;
    \item If $\vq$ is contained in $U$, then $\lambda{(\vq)}$ is an eigenvalue of $\mathcal{I}_{\vq}$ with $\mul\,\lambda{(\vq)}=1$;
    \item Let $\vec{v}$ be a tangent vector of $\Cn$ at $\vp$, the directional derivative of $\lambda{(-)}$ along $\vec{v}$ is $\frac{\pi}{2}\sum_{p\in \vp_f}\re\,(v(p)a_{p}(f)^2)$, where $f$ denoting an eigensection with eigenvalue $\lambda(\vp)$ and whose square has $\sph$ integral equal to one.
\end{enumerate}

If $N:=\mul\,\lambda(\vp)>1$, then there exist a neighborhood $U$ of $\vp$ such that there is a set of $N$ continuous functions $\mu_{i}(-)(i=1,\cdots,N)$ defined on $U$, with the following properties:
\begin{enumerate}[label=(\roman*)]
    \item $\mu_{i}(\vp)=\lambda(\vp)$ for $i=1,\cdots,N$;
    \item If $\vq\in U$, then $\mu_{i}(\vq)$'s are $\ZT$ eigenvalues of $\mathcal{I}_\vq$;
    \item Consider the eigenspace $V_{\lambda(\vp)}$. A tangent vector $\vec{v}$ at $\vp$ determines a symmetric bilinear form on $V_{\lambda(\vp)}$ by 
    \begin{align}
    B_{\vec{v}}(f,f'):=\frac{\pi}{2}\sum_{p\in\vp_f\cap\vp_{f'}}\re\,\big(v(p)a_p(f)a_p(f')\big),\label{TWblf}\end{align}
    where $f,f'\in V_{\lambda(\vp)}$. Denote the eigenvalues of this bilinear form as $\eta_1(\vec{v}),\cdots,\eta_N(\vec{v})$. Let $\vp(t)$ be a path of configurations such that $\vp(0)=\vp$ and $\vp'(0)=\vec{v}$. Then up to an arrangement the functions $\mu_i(-)$ satisfy $\frac{d}{dt}\mu_{i}(\vp(0))=\eta_i(\vec{v})$ for $i=1,\cdots,N$.
\end{enumerate}
\end{lemma}

Recall that $\vp_f\subset \vp$ consists of points in the configuration that $f$ has non-zero order $\frac12$ coefficients in local expansions. It follows that the variation of the eigenvalues is determined by the coefficients at these points of order $\frac12$. Given a tangent vector $\vec{v}\in T_\vp\Cn$, the $B_{\vec{v}}$ given in \eqref{TWblf} is called the \textbf{Taubes-Wu bilinear form}. Although signs of leading coefficients $a_p(f)$ are not determined, this bilinear form is well-defined since $f\times f'$ is a single-valued function defined on $\sph$, and hence $a_p(f)a_p(f')$ are well-defined.

\begin{lemma}
\label{lemma_different_eigenvalue}
Let $\vp\in \Cn$ and let $\lam$ be a eigenvalue of $\lb$ with  $\mul\,\lambda>1$, then there exists a direction $\vec{v}\in T_{\vp}\Cn$ such that at least two of the eigenvalues of the bilinear form $B_{\vec{v}}$ given in Lemma \ref{TWformula} are different.
\end{lemma}
\begin{proof}
    Let $N=\mul\,\lambda>1$. Choose a complex coordinate $z_p$ for each $p\in\vp$ and fix a basis for the eigenspace $V_{\lambda}$, say $\{f_i\}_{1\le i\le N}$. If for any direction $\vec{v}$, the $\eta_1(\vec{v}),\cdots,\eta_N(\vec{v})$ defined in Lemma \ref{TWformula} are all the same, then the Taubes-Wu bilinear form satisfies $B_{\vec{v}}(f_i,f_j)=0$ for $i\ne j$ while $B_{\vec{v}}(f_i,f_i)=\eta(\vec{v})\lan f_i,f_i\ran_{L^2}$ for any $i$ and $\vec{v}$, where $\eta(\vec{v})=\eta_i(\vec{v})$ for $1\le i\le N$.
    
    Consider a point $p\in\vp$, we want to prove that there is at most one base vector of $V_\lam$, say $f_1$, can have nonzero coefficient in $z_p^{1/2}$ term. Suppose not, then there are two distinct base vectors $f_1,f_2$ with nonzero leading coefficients $a_p(f_1)$ and $a_p(f_2)$. Let $v(p)=\bar{a}_p(f_1)\bar{a}_p(f_2)$, then $\Re\,(v(p)a_p(f_1)a_p(f_2))\neq 0$, which gives a contradiction. 

    By \cite[Lemma 2.7]{taubeswu2021topological} (see also Proposition \ref{orderminusone}), each eigenspace must contain non-critical eigensections. Therefore, we could assume that for some $p\in \vp$, $a_p(f_1)\neq 0$ while $a_p(f_2)=\cdots=a_p(f_N)=0$.

    Take $v(p)=\bar{a}_p^2(f_1)$, then $0=\Re\,(v(p)a_p(f_2)^2)=\eta(\vec{v})\lan f_2,f_2\ran_{L^2}$, which implies $\eta(\vec{v})=0$. However, $\Re\,(v(p)a_p^2(f_1))=|a_p(f_1)|^4\neq 0$, while $\Re\,(v(p)a_p^2(f_1))=\eta(\vec{v})\lan f_1, f_1\ran_{L^2}=0$, leading to a contradiction.
\end{proof}

For each vector $\vp\in \Cn$, according to Proposition \ref{prop_descrte_spectrum}, the spectrum of $\Delta:\sob\to\sob$ are a discrete non-decreasing sequence
$0<\lam_1(\vp)\leq \lam_2(\vp)\leq \cdots$. Our next step is to explore the generic property of the multiplicities of these eigenvalue functions on the configuration space $\Cn$. For this purpose, we introduce the subset $\mathcal{F}_k \subset \Cn$, which consists of configurations for which the first $k+1$ eigenvalues are distinct, implying that the first $k$ are simple. Additionally, we define $\mathcal{F}_0:=\Cn$ and we have a filtration $\cdots \subset \mathcal{F}_{k+1}\subset \mathcal{F}_k\subset\cdots \subset \mathcal{F}_0=\Cn$.

\begin{proposition}
For each $k\ge 1$, $\mathcal{F}_k$ is open and dense in $\Cn$.    
\end{proposition}
\begin{proof}
    It suffices to show that $\mathcal{F}_{k}$ is open and dense in $\mathcal{F}_{k-1}$ for $k\ge 1$. For openness, let $\vp\in\mathcal{F}_{k}$, and consider the first $k+1$ eigenvalues, we have $\lambda_1(\vp)<\cdots<\lambda_k(\vp)<\lambda_{k+1}(\vp)$. We define $c_0:=\min_{i=1}^k|\lam_{i+1}(\vp)-\lam_{i}(\vp)|.$ By Proposition \ref{contiofeig}, we can find a $\delta'>0$ such that for any $\vq\in \Cn$ with $\dis(\vp,\vq)<\delta'$, we have $|\lam_i(\vq)-\lam_i(\vp)|<\frac{c_0}{4}$ for $i=1,\cdots,k+1$. This implies $|\lam_i(\vq)-\lam_{i+1}(\vq)|>\frac{c_0}{2}>0,$ thereby $\mathcal{F}_{k}$ is open.

  %  so that the function $C(\delta_0,\delta)$ satisfies $C^2<\lambda_{i+1}(\vp)/\lambda_i(\vp)$ for $1\le i\le k$. Thus for $\vq\in\Cn$ with $\rho(\vp,\vq)<\delta$, we have $\lambda_i(\vq)\le C\lambda_i(\vp)<\frac{1}{C}\lambda_{i+1}(\vp)\le\lambda_{i+1}(\vq)$ for $1\le i\le k$. It follows that $\vq\in\mathcal{F}_k$ and $\mathcal{F}_k$ is open.

For denseness, let $\vp\in\mathcal{F}_{k-1}\setminus\mathcal{F}_{k}$. According to Lemma \ref{lemma_different_eigenvalue}, there exists a tangent vector $\vec{v}\in T_\vp\Cn$ so that at least two of the eigenvalues of $B_{\vec{v}}$ on $\lambda_k(\vp)$ are different. Choose a smooth path $\vp(t)$ so that $\vp'(0)=\vec{v}$. From Lemma \ref{lemmalocal} one conclude that $\mu_{i}(\vp(t))\ne\mu_{j}(\vp(t))$ for some $i\ne j$ and $t\ne 0$ sufficiently small. Therefore $\mul\,\lambda_{k}(\vp(t))<\mul\,\lambda_k(\vp)$. Repeating such a perturbation a finite number of times while keeping the first $k-1$  eigenvalues simple, we find a configuration in $\mathcal{F}_{k}$ that as close to $\vp$ as we like. Therefore, $\mathcal{F}_{k}$ is dense in $\mathcal{F}_{k-1}$.
\end{proof}
\begin{theorem}\label{residualsimple}
     The set $\mathfrak{S}:=\{\vp\in\Cn:\text{all eigenvalues of }\lb\text{ are simple}\}$ is residual.
\end{theorem}
\begin{proof}
    Note that $\mathfrak{S}\supset\cap_k\mathcal{F}_k$. Since $\mathcal{F}_k$ are open and dense in $\Cn$ for any $k$, $\mathfrak{S}$ is residual.
\end{proof}
 Together with \cite[Lemma 2.7]{taubeswu2021topological} (or Lemma \ref{mulle4} of this article), we arrive at the following:
\begin{theorem}
    Configurations in $\mathfrak{S}$ are non-critical. As a result, generic configurations on $\Cn$ are non-critical.
\end{theorem}

\subsection{$\ZT$-degeneration and compactification}

Following Taubes-Wu we shall consider a compactification of $\Cn$, resulting in a topologically interesting space. A topology on the space $\ccn:=\mathcal{C}_0\cup\mathcal{C}_{2}\cup\cdots\cup\Cn$ was defined in \cite{taubeswu2021topological}, so that $\mathcal{C}_{2i}(0\le i\le n)$ forms a stratification on it. With this topology, we can define a process which we name as $\ZT$\textbf{-degeneration}. We prove that eigenvalues are continuous in this process in Theorem \ref{continuousztdegenerate}, using an argument similar to \cite[Section 3.3]{ABoper}. The continuity of the first eigenvalue has been proved in \cite{taubeswu2021topological}. One could also view from spectral convergence of conical metric point of view, which we refer to \cite{mazzeo2006resolution,mazzeo2020conical} for more details.

Consider a sequence of configurations $\{\vp_l\in\Cn\}_{l\in\mathbb{N}}$. Assume $\vp_l=\{p_1^l,\cdots,p_{2n}^l\}$, where $p_i^l(1\le i\le 2n)$ are distinct points on $\sph$ for each $l$. Let $q_1,\cdots,q_{m}$ be distinct points on $\sph$ for $m\le 2n$, that exhaust the limits of $\{p_i^l\}_l$ as $l\to \infty$ for all $i$. We call the set of sequences $\{\{p_i^l\}_l: \lim_{l}p^l_i=q_s,\, 1\le i\le 2n\}$ as the $q_s$-group $(1\le s\le m)$. 

Consider a disk centered at $q_s$. If the number of sequences in the $q_s$-group is odd, and suppose additionally $l$ is sufficiently large, then the boundary of this disk, when seen as an element of $\pi_1(\sph\setminus\vp_l)$, would have $-1$ as its image in $\ZT$ under the monodromy representation $\rho:\pi_1(\sph\setminus\vp_l)\to\ZT$. In contrary, if the cardinality of the $q_s$-group is even, the image of this circle would be $1\in\ZT$. With the topology introduced by Taubes and Wu, we have $\lim_l\vp_l=\vq$ for $\vq=\{q_s:|q_s\text{-group}|\text{ is odd}\}$. If $|\vq|<2n$, we say that the sequence $\{\vp_l\}_l$ $\ZT$\textbf{-degenerates} to $\vq$, and name $\vq$ as the $\ZT$\textbf{-limit} of $\{\vp_l\}_l$. We shall prove that eigenvalues are continuous in the process of $\ZT$\textbf{-degeneration}.

Consider a neighborhood of the set of limit points $\vq':=\{q_1,\cdots,q_{m}\}$, namely $B_{\epsilon}(\vq')=\cup_{q\in\vq'}B_{\epsilon}(q)\subset \sph$. Let $\Omega_{\epsilon}=\sph\setminus B_{\epsilon}(\vq')$, then for $l$ sufficiently large, $\mathcal{I}_{\vp_l}|_{\Omega_\epsilon}=\mathcal{I}_{\vq}|_{\Omega_\epsilon}$. 

\begin{lemma}\label{limsup}
    For any $k\ge 1$, $\limsup_l\lam_k(\vp_l)\le\lam_k(\vq)$.
\end{lemma}
\begin{proof}
    By the min-max characterization of the $k$-th $\ZT$ eigenvalue, for any $\varepsilon'>0$, there exists $\epsilon$ and $k$ smooth sections $f_1,\cdots,f_k$ of $\mathcal{I}_{\vq}$ which are compactly supported in $\Omega_\epsilon$ for some $\epsilon>0$, such that the vector space $H=\mathrm{span}\{f_1,\cdots,f_k\}$ satisfies \[\sup_{f\in H\setminus\{0\}}||df||^2_{\mathcal{L}_1^2}/||f||^2_{L^2}\le\lam_k(\vq)+\varepsilon'.\]  
    Consider $l$ sufficiently large such that $\vp_l\subset B_\epsilon(\vq')$ and hence $\mathcal{I}_{\vp_l}|_{\Omega_\epsilon}=\mathcal{I}_{\vq}|_{\Omega_\epsilon}$. Then $f_1,\cdots,f_k$ can be identified as sections of $\mathcal{I}_{\vp_l}$, which implies $\lam_k(\vp_l)\le\sup_{f\in H\setminus\{0\}}||df||^2_{\mathcal{L}_1^2}/||f||^2_{L^2}\leq \lam_k(\vq)+\epsilon'$. Therefore, $\limsup_l\lam_k(\vq_l)\le\lam_k(\vq)$.
\end{proof}

Following Taubes-Wu \cite[(2.2)]{taubeswu2021topological}, for $\epsilon>0$ and $\vp\in\Cn$, we define a cut-off function \begin{align}\label{cutoff}
    \chi_{\epsilon,\vp}(-)=\chi\big(2\frac{\ln{\dist(-,\vp)}}{\ln(\epsilon)}-1\big),
\end{align}
so that $\chi_{\epsilon,\vp}(x)$ is equal to $1$ when $\dist(x,\vp)>\sqrt{\epsilon}$ and equal to $0$ when $\dist(x,\vp)<\epsilon$. Moreover, with a suitable choice of $\chi$, we may assume $|d\chi_{\epsilon,\vp}(x)|<c_0 /|\dist(x,\vp)\ln\epsilon|$ for a constant $c_0$ that is independent of $\vp$. As a result, $\int_{\sph}|d\chi_{\epsilon,\vp}|^2\le c_1/|\ln\epsilon|$, with $c_1$ a constant being independent of $\vp$.

\begin{lemma}[\cite{taubeswu2021topological}, Lemma 2.1]\label{bound}
    Suppose $f$ is a $\ZT$ eigensection of $\lb$, with eigenvalue $\lam$. Then for any point $x\notin \vp$, we have $|f|(x)\le c(\lam+1)||f||_{L^2}$, where $c$ is a constant independent of $\vp$ and $q$. 
\end{lemma}

\begin{lemma}\label{limitsectionh}
    Suppose the sequence $\{\vp_l\}_l$ has $\ZT$-limit $\vq$ and the limit $\lim_l\lam_k(\vp_l)=:\lam$ exists. Consider a sequence of $\ZT$ eigensections $\{f_l\}_l$, such that $\Delta f_l+\lam_{k}(\vp_l)f_l=0$ and $||f_l||_{L^2}=1$. Then there exists a $\ZT$ eigensection $h$ of $\mathcal{I}_{\vq}$, with eigenvalue $\lam$, such that for any $\epsilon>0$, by passing to a subsequence, we have $f_{l}$ converge to $h$, strongly in $L^2(\mathcal{I}_{\vq}|_{\Omega_{\epsilon}})$, and weakly in $\mathcal{L}_1^2(\mathcal{I}_{\vq}|_{\Omega_{\epsilon}})$. Moreover, $\lim_l||f_l||_{L^2}=||h||_{L^2}$.
\end{lemma}
\begin{proof}
Consider a sequence of positive reals $\{\epsilon_l\}_l$, such that $\lim_l\epsilon_l=0$, and for any $l'>l$, $\vp_{l'}\subset B_{\epsilon_{l}}(\vq')$. Let $\chi_{l}:=\chi_{\epsilon_l,\vq'}$, which are cut-off functions defined in \eqref{cutoff}. Then $\chi_l f_l$ can be regarded as a section of $\mathcal{I}_{\vq}$, with $||\chi_l f_l||_{\mathcal{L}_1^2(\mathcal{I}_\vq)}\le 2||f_l||_{\sob}\le 2\sqrt{\lam_k(\vp_l)+1}\le 3\sqrt{\lam+1}$ for sufficiently large $l$. Here, the first inequality holds due to $\int_{\sph}|d\chi_{l}|^2=\mathcal{O}(1/|\ln\epsilon_l|)$. 

Therefore, $\{\chi_lf_l\}_l$ is a bounded sequence in $\mathcal{L}_1^2(\mathcal{I}_\vq)$. As $\mathcal{I}_{\vq}|_{\Omega_{\epsilon}}$ is well-defined, by standard local elliptic regularity, $\|\chi_lf_l\|_{L^2_k}\leq C_k$ for some constant $k$ independent of $l$. Hence by passing to a subsequence, there exists a $h\in \mathcal{L}_1^2(\mathcal{I}_{\vq})$, such that $||\chi_l f_l-h||_{L_k^2(\mathcal{I}_{\vq}|_{\Omega_{\epsilon}})}\to 0$. 

In addition, as $f_l$ is an eigensection with $\|f_l\|_{L^2}=1$, we have $\int_{S^2}|df_l|^2=\lam_l$. Note that $|f_l|$ is a single-valued function defined over $S^2$ and by Kato inequality, $|df_l|\geq d|f_l|$, we have $\int_{S^2}|f_l|^2+|d|f_l
||^2\leq C$, which implies $\lim_{l\to\infty}\|f_l\|_{L^2}=\|h\|_{L^2}.$

%Since for any $\epsilon>0$, for $l$ sufficiently large, $0<\epsilon_l<\epsilon$, we have $||(f_l-h)|_{\Omega_\epsilon}||_{L^2}\le ||\chi_l f_l-h||_{L^2}$ and $|\lan \eta,(f_l-h)|_{\Omega_\epsilon}\ran_{\mathcal{L}_1^2}|\le |\lan \bar{\eta},\chi_l f_l-h\ran_{\mathcal{L}_1^2}|$, where $\eta$ is a smooth section of $\mathcal{I}_{\vq}|_{\Omega_\epsilon}$ and $\bar{\eta}\in C_0^\infty(\mathcal{I}_\vq)$ one of its extensions. It follows that $f_{l_i}$ converge to $h$, strongly in $L^2(\mathcal{I}_{\vq}|_{\Omega_{\epsilon}})$, weakly in $\mathcal{L}_1^2(\mathcal{I}_{\vq}|_{\Omega_{\epsilon}})$, provided $\epsilon$ being fixed. Moreover, by standard elliptic theory, the sequence converge to $h$ is in the sense of $C^{\infty}_{\text{loc}}$ on $\Omega_\epsilon$.

   It remains to show that $h$ is an eigensection. Consider a section $\eta\in C_c^\infty(\mathcal{I}_{\vq})$, we have $\int_{\sph}\lan d\eta, dh\ran = \lim_l\int_{\sph}\lan d\eta,d(\chi_l f_l)\ran$. And 
   \begin{align*}
   \int_{\sph}\lan d\eta,d(\chi_l f_l)\ran=&\int_{\sph}\lan d(\chi_l\eta),d f_l\ran-\int_{\sph}\eta\lan d\chi_l,df_l\ran+\int_{\sph}f_l\lan d\eta,d\chi_l\ran\\
   =&\lam_k(\vp_l)\int_{\sph}\eta\chi_lf_l-\int_{\sph}\eta\lan d\chi_l,df_l\ran+\int_{\sph}f_l\lan d\eta,d\chi_l\ran
   \end{align*} 
Note that \[\int_{\sph}|\eta\lan d\chi_l,df_l\ran|+\int_{\sph}|f_l\lan d\eta,d\chi_l\ran|\le c\max_{x\in\sph}(|\eta|+|d\eta|)(x)||f_l||_{\mathcal{L}_1^2}\big(\int_{\sph} |d\chi_l|^2\big)^{\frac12}=\mathcal{O}(1/\sqrt{|\ln\epsilon_l|}).\]
Thus $\int_{\sph}\lan d\eta,dh\ran=\lim_l\lam_k(\vp_l)\int_{\sph}\eta\chi_lf_l=\lam\int_{\sph}\eta h$.
\end{proof}

\begin{theorem}\label{continuousztdegenerate}
$\lim_l\lam_k(\vp_l)=\lam_k(\vq)$.
\end{theorem}
\begin{proof}
    We prove this by induction. Let $\lam_1=\liminf_l\lam_1(\vp_l)$. We can choose a subsequence $\{\vp_{l_i}\}_{l_i}$, such that $\lim_{l_i}\lam_1(\vp_{l_i})=\lam_1$, and define an eigensection $h_1$ of eigenvalue $\lam_1$ as a limit of $f_{l_i}$, as in Lemma \ref{limitsectionh}. Since $\lam_1$ is an $\ZT$ eigenvalue of $\mathcal{I}_{\vq}$, there must be $\lam_1\ge \lam_1(\vq)$. Then by Lemma \ref{limsup}, $\lim_l\lam_1(\vp_l)=\lam_1(\vq)$.

    Let $\lam_s=\liminf_l\lam_s(\vp_l)$ for $1\le s\le k$. By a similar argument, for each $2\le s\le k$, we can define a nonzero $\ZT$ eigensection $h_s$ of $\mathcal{I}_{\vq}$ with eigenvalue $\lam_s$. Assume $h_s$ is the limit of a sequence of eigensections $\{f^s_{l}\}_{l}$, where $f^s_{l}$ are the $s$-th eigensection, with eigenvalue $\lam_s(\vp_l)$, such that $||f^s_{l}||_{L^2}=1$. Here, for the sake of simplicity, we do not emphasize the procedure of passing to subsequence, and use subscript $l$ only.

    Suppose we have proved that for any $1\le s\le k-1$, $\lim_l\lam_s(\vp_l)=\lam_s(\vq)$. By definition of eigensections, for any $l$ and $s<k$, $\lan f^k_{l},f^s_{l}\ran_{L^2}=\int_{\sph}f^k_{l}f^s_{l}=0$. Moreover, according to Lemma \ref{bound}, we have \[|\lan f^k_{l},f^s_{l}\ran_{L^2(B_{\epsilon}(\vq'))}|\le \max_{x\in\sph}|f^k_{l}f^s_{l}|(x)|B_{\epsilon}(\vq')|=\mathcal{O}(\epsilon^2),\]hence $|\lan h_k,h_s\ran_{L^2}|=\lim_{l}|\lan \chi_l f^k_l,\chi_l f^s_l\ran_{L^2}|\le \lim_l |\lan f^k_l,f^s_l\ran_{L^2(B_l)}|=0$ for any $1\le s\le k-1$, where $B_l=B_{\sqrt{\epsilon_l}}(\vq')$. It follows that $\lam_k\ge \lam_k(\vq)$. Consequently, by Lemma \ref{limsup}, $\lim_l\lam_k(\vp_l)=\lam_k(\vq)$. 
\end{proof}

\begin{corollary}
    The $k$-th $\ZT$ eigenvalue $\lambda_k(\vp)$, when considered as a continuous function on $\Cn$, can be extended to the compactification $\ccn$ continuously.
\end{corollary}

\section{Algebraic relations between leading coefficients}
In this section, we will present a series of algebraic identities derived using integration by parts in \cite{taubeswu2020examples}, imposing constraints on the leading coefficients of $\ZT$ eigensections. By combining these elements, we establish a uniform upper bound for orders of $\ZT$-eigensections and multiplicity of $\ZT$-eigenvalues. 

\subsection{Complexified $\ZT$ eigensections and vector fields action}\label{4.1}
In this subsection, we will introduce a complexification of the $\ZT$ eigensections and discuss differential operators generated by rotations over $\sph$.
\subsubsection{Complexified $\ZT$ eigensections}
We fix one complex local coordinate for any $p\in\vp$. For the sake of convenience in later applications of representation theory, we consider the complexification of $\lb$, denoted as $\lbc$, throughout this section.
The definition of the Laplacian remains consistent with that of the real line bundle, as do the $\ZT$ eigenvalues. Since the eigenvalue equations are real equations, most of the analytic results and definitions of $\lb$ automatically generalize to sections of $\lbc$. Although the $\ZT$ eigenvalues of $\lbc$ coincide with those of $\lb$, the multiplicity is defined as the complex dimension of the eigensubspace of $\sobc$. 

Let $f$ be an eigensection of $\lbc$, and let $z$ be a local complex coordinate of $p\in \vp$. Then, we can express $f(z,\bar{z})$ as follows:
\begin{align}\label{fourier}
f(z,\bar{z})=\sum_{n=0}^{+\infty} a_n(1+u_n(r))z^{n+\frac{1}{2}}+\sum_{n=0}^{+\infty}b_n(1+u'_n(r))(\bar{z})^{n+\frac{1}{2}},
\end{align}
where $r=|z|$, $u_n$ and $u'_n$ are real analytic functions of $r$, and $a_n$ and $b_n$ are complex constants.

\begin{definition}\label{leadterm}
We define $\vn_p(f)$ to be the smallest $n$ in the expansion \eqref{fourier} such that $|a_n|+|b_n|\neq 0$. Additionally, we denote $(a_n,b_n)$ as the pair of leading coefficients of $f$ at $p$ for $n=\vn_p(f)$. Furthermore, we define the order of $f$ as $\ord(f):=\min_{p\in \vp}\vn_p(f)$. We consider $f$ to be critical if $\vp_f=\varnothing$. Moreover, $f$ is called non-degenerate if $\vn_p(f)=1$ for any $p\in \vp$.

Let $\lambda$ be an eigenvalue with $V_{\lambda}\subset \sobc$ representing the corresponding eigenspace. We define the multiplicity $\mul\,\lambda:=\dim_{\mathbb{C}}V_{\lambda}$.
\end{definition}

Note that $\lb$ is a subbundle of $\lbc$, and hence a real section $f$ can be seen as a section of $\lbc$. For such a section, in the expansion \eqref{fourier}, the leading coefficients satisfying $a_n=\bar{b}_n$. Additionally, the existence of critical eigensections in $\lb$ is equivalent with that in $\lbc$. 

In the expansion \eqref{fourier}, the terms $z^{n+1/2}$ and $\bar{z}^{n+1/2}$ are ambiguous due to their multi-valued nature. Consequently, the pair of leading coefficients $(a_n,b_n)$ is determined only up to a sign. It is essential to handle this ambiguity with care. Typically, to fix the sign, one need to select a complex coordinate $w$ at $p\in\vp\subset \Sigma_\vp$, such that $w^2=z$; the sign depends on the choice of such coordinates. Once this complex coordinate is choosen, for the corresponding odd function $\tilde{f}$ defined on $\Sigma_\vp$, we have the asymptotic expansion \begin{align}\label{mentionsign}\tilde{f}(w,\bar{w})=a_nw^{2n+1}+b_n\bar{w}^{2n+1}+\mathcal{O}(|w|^{2n+3}),\end{align}where $n=\vn_p(f)$. When stating that $f$ has leading coefficients $(a_n,b_n)$, it is understood that such a complex coordinate and asymptotic expansion of $\tilde{f}$ have been fixed implicitly.

\subsubsection{Vector fields from rotations}
We now consider three vector fields on $\sph$ that can be interpreted as infinitesimal generators of the group of rotations $\SO$. Let $(x_1,x_2,x_3)$ be the Euclidean coordinate, and we define the vector fields as follows:
\begin{align*}
L_1=x_2\frac{\partial}{\partial x_3}-x_3\frac{\partial}{\partial x_2},\;
L_2=x_3\frac{\partial}{\partial x_1}-x_1\frac{\partial}{\partial x_3},\;
L_3=x_1\frac{\partial}{\partial x_2}-x_2\frac{\partial}{\partial x_1},
\end{align*}
which are the tangent vector fields generated by $\SO$ action on $(0,0,-1)$. We also define the operators $J_j = -iL_j$ for $j = 1, 2, 3$. For the $z$-coordinates near $(0, 0, -1) \in \sph$ given by stereographic projection from the north pole, namely $z=(x_1+ix_2)/(1-x_3)$, straightforward calculations yield: 
\begin{equation}
\label{Jcplxcoor} 
    \begin{split}
        J_1z&=\frac{1}{2}-\frac{1}{2}z^2,
\;J_2z=\frac{i}{2}+\frac{i}{2}z^2,\;J_3z=z,\\
J_1\bar{z}&=-\frac{1}{2}+\frac{1}{2}\bar{z}^2,\;J_2\bar{z}=\frac{i}{2}+\frac{i}{2}\bar{z}^2,\;J_3\bar{z}=-\bar{z}.
    \end{split}
\end{equation}

Since the vector fields $J_j(j=1,2,3)$ commute with the Laplacian $\Delta$, given an eigensection $f$ of eigenvalue $\lambda$ on any $\lbc$, $J_jf$ (resp. $L_jf$) still satisfy the equation $\Delta J_jf+\lambda J_jf=0$ $(\mathrm{resp.} \Delta L_jf+\lambda L_jf=0)$. Remark after Definition \ref{def_leading_coefficient_pf} is applied for these sections. Throughout this article, when referring to the complex coordinate $z$ near the south pole, we consistently mean the one defined by the stereographic projection from the north pole, as presented above. The choice of coordinates for other points on $\sph$ will be fixed later.

The most important point of the above formula is that even $J_1$, $J_2$ will lower the vanishing order of sections at $(0,0,-1)$ by $1$, $J_3$ will preserve this order.

%Let $p\in\vp$. There is a coset of $\mathbb{S}^1\triangleleft \SO$, which we denote as $K_p$, such that $A\in K_p$ if and only if $A.p=(0,0,-1)\in\sph$. Any element $A$ of $K_p$ defines a complex coordinate near $p$, say $z\circ A$. If $A'\in K_p$, then $z\circ A'=e^{i\phi}z\circ A$ for $A'\circ A^{-1}=e^{i\phi}\in \mathbb{S}^1$.

Consider a point $p\in\vp$ that is not the south pole. We can express the spherical coordinates of $\mathbb{R}^3\setminus{0}$ as $(x_1,x_2,x_3)=R(\sin\theta\cos\varphi,\sin\theta\sin\varphi,\cos\theta)$. Let $(\varphi,\theta)$ be the spherical coordinates of $p$ on the unit sphere. We define $A_p\in \SO$ that maps $p$ to $(0,0,-1)$ as follows:
\[\left(\begin{array}{ccc}
-\cos\theta & 0 & \sin\theta\\
0 & 1 & 0\\
-\sin\theta & 0 & -\cos\theta
\end{array}\right)\left(\begin{array}{ccc}
\cos\varphi & \sin\varphi & 0\\
-\sin\varphi & \cos\varphi & 0\\
0 & 0 & 1
\end{array}\right)\]
We denote the complex coordinate $A_p^\ast z$ by $z_p$, where $z$ represents the complex coordinate near $(0,0,-1)$ defined previously. We define $L_j^p:=(A_p^{-1})_\ast L_j$ and $J_j^p:=-iL_j^{p}$, and these vector fields are also infinitesimal rotations. Additionally, \eqref{Jcplxcoor} holds for $J_j^pz_p$ and $J_j^p\bar{z}_p$ for $j=1,2,3$. As a result, $(J^p_jz_p)|_{z_p=0}=(J_jz)|_{z=0}$ and $(J^p_j\bar{z}_p)|_{\bar{z}_p=0}=(J_j\bar{z})|_{\bar{z}=0}$ $(j=1,2,3)$. We refer to $J_3^p$ as infinitesimal rotations based at $p$. 

Straightforward calculations yield the following:

\begin{lemma}\label{Liealgetrans}
    \begin{align*}
\left(\begin{array}{ccc}
L_1\\L_2\\L_3 
\end{array}\right)=\left(\begin{array}{ccc}
-\cos\theta\cos\varphi      &        -\sin\varphi       &       
-\sin\theta\cos\varphi      \\
-\cos\theta\sin\varphi      &       \cos\varphi      &       -\sin\theta\sin\varphi\\
\sin\theta     &  0       &       -\cos\theta
\end{array}\right)\left(\begin{array}{ccc}
L^{p}_1\\L^{p}_2\\L^{p}_3 
\end{array}\right),\end{align*}
and
   \begin{align*}
\left(\begin{array}{ccc}
L^{p}_1\\L^{p}_2\\L^p_3 
\end{array}\right)=\left(\begin{array}{ccc}
-\cos\theta\cos\varphi      &        -\cos\theta\sin\varphi       &       
\sin\theta      \\
-\sin\varphi      &       \cos\varphi     &      0\\
-\sin\theta\cos\varphi     & -\sin\theta\sin\varphi       &       -\cos\theta
\end{array}\right)\left(\begin{array}{ccc}
L_1\\L_2\\L_3 
\end{array}\right).\end{align*}
\end{lemma}

For the north pole $p=(0,0,1)$, $z_p$ is not defined, so as $J^p_j(j=1,2,3)$. Indeed, $z_p$ is determined only up to a unit complex $e^{i\varphi}$. If $\varphi$ is fixed, then these notations carry on, and the lemma remains true.

\begin{proposition}\label{orderminusone}
Let $f$ be an eigensection of $\lb$. For $p\in\vp$, let $n:=\vn_p(f)$. Then, we have $\vn_p(J_1^p(f))=\vn_p(J_2^p(f))=n-1$ and $\vn_p(J_3^p(f))=n$.

For $q\neq \pm p$, where $q$ is not one of the poles, let $J_1^q$, $J_2^q$, $J_3^q$ denote the infinitesimal rotations based at $q$. Then 
$J_j^qf\ne 0$ and $\vn_p(J_j^qf)=n-1$ for $j=1,2,3$.  
\end{proposition}
\begin{proof}
   We start by assuming, without loss of generality, that $p=(0,0,-1)$. The claims for $J_1^p$, $J_2^p$, and $J_3^p$ directly follow from \eqref{Jcplxcoor}.

   Next, let $(\varphi,\theta)$ be the spherical coordinates of $q$, and let $z$ be the complex coordinate near $p$, defined by stereographic projection from $-p$. Then, we have $$J_3^q=-\sin\theta\cos\varphi J_1-\sin\theta\sin\varphi J_2-\cos\theta J_3.$$ Consequently, $(J_3^qz)|_{r=0}=-\frac{1}{2}\sin\theta e^{i\varphi}$ and $(J_3^q\bar{z})|_{r=0}=\frac{1}{2}\sin\theta e^{-i\varphi}$. It follows that 
\[
\begin{aligned}
J_3^qf(z,\bar{z}) &= -\frac{1}{2}(n+\frac{1}{2})\sin\theta e^{i\varphi}a_nz^{n-1/2}+\frac{1}{2}(n+\frac{1}{2})\sin\theta e^{-i\varphi}b_n\bar{z}^{n-1/2}+\mathcal{O}(r^{n+1/2}),
\end{aligned}
\]
and hence $J_3^qf\ne 0$ and $\vn_p(J_3^qf)=n-1$.
Similar computations hold for $J_1^qf$ and $J_2^qf$, and are omitted here to avoid redundancy.
\end{proof}

Now, we will consider the multiplicity of the eigenspaces.
\begin{lemma}{\cite[Lemma 2.7]{taubeswu2021topological}}
\label{mulle4}
Let $\vq$ be a configuration consisting of four or more distinct points on $\sph$. If $\lambda$ is an critical eigenvalue of $\mathcal{I}_{\vq}$, then $\mul\,\lambda\ge 2$. Here, $\mul\,\lambda=\dim_{\RR}V_\lambda$, where $V_\lambda$ is the eigensubspace of $\sob$.
\end{lemma}

As mentioned in \cite{taubeswu2021topological}, with a little more work, which has been done above, we could extend the result to obtain a better multiplicity estimate.

\begin{proposition}
    \label{mulge4}
    Assume that $\vp$ is not the configuration of a pair of antipodal points or two pairs of antipodal points. Suppose that $\lambda$ is a critical eigenvalue of $\lbc$, then $\mul\,\lambda\ge 4$. Here, $\mul\,\lambda=\dim_{\CC}V_\lambda$, where $V_\lambda$ is the eigensubspace of $\sobc$.
\end{proposition}
\begin{proof}
    Let $f$ be a critical eigensection with eigenvalue $\lambda$. We shall show that $J_1f, J_2f$ and $J_3f$ are linearly independent eigensections.
    
    Note that $\sum_j\mu_jJ_j=\eta J_3^q$ for some $q\in\sph$ and $\eta\in\RR\setminus\{0\}$. By assumption, there is some $p\in\vp$ such that $q\ne \pm p$. We deduce from Proposition \ref{orderminusone} that $\sum_j\mu_jJ_j f\ne 0$, and hence $f,J_1f,J_2f$ and $J_3f$ are linearly independent.
\end{proof}

When $\vp$ is the configuration of a pair of antipodal points, then one could explicitly compute the eigensections and the above lemma is incorrect. The hardest situation will be $\vp\in \mathcal{C}_4$ which are the configuration of two paris of antipodal points, which will be studied in Section \ref{5.4}.

\subsection{Algebraic Identities of the leading coefficients}
In this subsection, we will introduce several algebraic identities on the leading coefficients of the asymptotic expansion of a critical eigensection $f$ near points in $\vp$, which have been studied in \cite[Appendix]{taubeswu2020examples}. 

From now on, for each $p\in\vp$, we fix the complex coordinate as $z_p=A_p^\ast z$. Asssume that $f$ is a critical eigensection of $\lbc$ with local expansion $f(z,\bar{z})=a_pz_p^{3/2}+b_p\bar{z}_p^{3/2}+\mathcal{O}(r^{5/2})$ for each $p\in\vp$. Take three arbitrary smooth (complex) vector fields $V_1,V_2$ and $V_3$ on $\sph$. We have
\begin{align*}
V_1f(z_p,\bar{z}_p)&=\frac{3}{2}a_pz_p^{1/2}(V_1z_p)|_{r=0}+\frac{3}{2}b_p\bar{z}_p^{1/2}(V_1\bar{z}_p)|_{r=0}+\mathcal{O}(r^{3/2}),\\
V_2V_3f(z_p,\bar{z}_p)&=\frac{3}{4}a_pz_p^{-1/2}(V_2z_pV_3z_p)|_{r=0}+\frac34b_p(\bar{z}_p)^{-1/2}(V_2\bar{z}_pV_3\bar{z}_p)|_{r=0}+\mathcal{O}(r^{1/2}).
\end{align*}
Let $\rho>0$ be much smaller than $\delta_0(\vp)$, i.e. much smaller than the smallest distance between any two points in $\vp$. Consider a smooth nonnegative test function $\chi_\rho$ on $\sph$, such that $\chi_\rho$ vanishes on the $\rho/2$-neighborhood of $\vp$ on $\sph$ and is identically equal to $1$ outside the $\rho$-neighborhood. Additionally, suppose that within the $\rho$-neighborhood of $p\in\vp$, the cut-off function $\chi_{\rho}$ depends only on the radius $r=|z_p|$.

\begin{lemma}\label{lemmalocal}
Consider an eigensection $\phi_1$ of eigenvalue $\lambda$. And $\phi_2:\sph\setminus\vp\to \lbc$ is a smooth section satisfying $\Delta \phi_2+\lambda\phi_2=0$. And consider the corresponding odd functions $\tilde{\phi}_1$ and $\tilde{\phi}_2$ on $\Sigma_{\vp}$, such that, with complex coordinates $w_p$ and $z_p=w_p^2$ being fixed near each $p\in\vp$ on $\Sigma_\vp$ and $\sph$ respectively, the odd functions have the following asymptotic expansions
\begin{align*}
\tilde{\phi}_1(w_p,\bar{w}_p)=\gamma_pw_p+\delta_p\bar{w}_p+\mathcal{O}(r^{3}),\;\tilde{\phi}_2(w_p,\bar{w}_p)=\gamma'_pw_p^{-1}+\delta'_p\bar{w}^{-1}_p+\mathcal{O}(r).
\end{align*}
Then $\sum_{p\in\vp}(\gamma'_p\gamma_p+\delta'_p\delta_p)=0$.
\label{localize}
\end{lemma}
\begin{proof}
We compute
\begin{align*}
d\ast(\phi_2d\phi_1-\phi_1d\phi_2)&=d\phi_2\wedge\ast d\phi_1-d\phi_1\wedge\ast d\phi_2+\phi_2\Delta\phi_1-\phi_1\Delta\phi_2=0,
\end{align*}
where the $\ast$ represents the Hodge star operator. Therefore, $\lim_{\rho\to 0}\int_{\sph}\chi_\rho d\ast(\phi_2d\phi_1-\phi_1d\phi_2)=0.$ 

For $\phi_2\ast d\phi_1$, near $p$, we have
\begin{align*}
\phi_2\ast d\phi_1=&\frac{1}{2}(\gamma'_pz_p^{-1/2}+\delta'_p(\bar{z}_p)^{-1/2})(\gamma_pz_p^{-1/2}\ast dz_p+\delta_p(\bar{z}_p)^{-1/2}\ast d\bar{z}_p)+\mathcal{O}(r),\\
-\phi_1\ast d\phi_2=&\frac{1}{2}(\gamma_pz_p^{1/2}+\delta_p\bar{z}_p^{1/2})(\gamma'_pz_p^{-3/2}\ast dz_p+\delta'_p(\bar{z}_p)^{-3/2}\ast d\bar{z}_p)+\mathcal{O}(r).
\end{align*}
With $\ast dz_p=-idz_p, \ast d\bar{z}_p=id\bar{z}_p$, we then have
\begin{align*}
\phi_2\ast d\phi_1=&\frac{i}{2}(-\gamma'_p\gamma_pz_p^{-1}dz_p+\delta'_p\delta_p(\bar{z}_p)^{-1}d\bar{z}_p+\gamma'_p\delta_pr^{-1}d\bar{z}_p-\delta'_p\gamma_pr^{-1}dz_p)+\mathcal{O}(r),\\
-\phi_1\ast d\phi_2=&\frac{i}{2}(-\gamma_p\gamma'_p z_p^{-1}dz_p+\delta_p\delta'_p(\bar{z}_p)^{-1}d\bar{z}_p+\gamma_p\delta'_p\frac{z_p^{2}}{r^3} d\bar{z}_p-\delta_p\gamma'_p\frac{\bar{z}_p^{2}}{r^3}dz_p)+\mathcal{O}(r).
\end{align*}
Since $\int_{|z_p|=r}\frac{z_p^2}{r^3}\mathrm{d}\bar{z}_p=0$, we conclude
\begin{align*}
0&=\int_{\sph}\chi_\rho\wedge d\ast(\phi_2d\phi_1-\phi_1d\phi_2)=\lim_{\rho\to 0}\int_{\sph}d\chi_\rho\wedge\ast(\phi_2d\phi_1-\phi_1d\phi_2)\\
&=\lim_{\rho\to 0}-i\int_{\rho/2}^{\rho}d\chi_{\rho}dr\int_{|z_p|=r}(\gamma_p\gamma'_p+\delta_p\delta'_p)\frac{1}{z_p}dz_p
\\
&=2\pi(\gamma_p\gamma'_p+\delta_p\delta'_p)\lim_{\rho\to 0}\int_{\rho/2}^{\rho}d\chi_{\rho}dr=2\pi (\gamma_p\gamma'_p+\delta_p\delta'_p).
\end{align*}
\end{proof}

We fixed the complex coordinate $w_p$ near $p\in\Sigma_\vp$ while using $z_p^{1/2}$ and $\bar{z}_p^{1/2}$ only in the proof. However, this is necessary to eliminate the ambiguity of sign, which was done in \cite{taubeswu2020examples} implicitly. Nevertheless, for the calculation presented later in this article, it is important to treat this explicitly. 

Although we make a slight modification, for in our case the complexification $\lbc$ was concerned, the proof provided here is essentially the same with that in the Appendix of \cite{taubeswu2020examples}. Using \cite[(3.8)]{Victor}, one can obtain another proof. The following Lemma is a direct application of Lemma \ref{lemmalocal}, which will be convenient when $\vp\in\mathcal{C}_4$.

\begin{lemma}\label{4.10}
Let $f$ be a critical eigensection with expansions $f(z,\bar{z})=a_pz_p^{\frac32}+b_p\bar{z}_p^{\frac32}+\mathcal{O}(r^{\frac52})$ near each $p$, we denote $\alpha_p=a^2_p$ and $\beta_p=b_p^2$, assume that vector fields $V_1,V_2$ and $V_3$ commute with the Laplacian, then we have
\begin{align}
\sum_{p\in \vp}\left(\alpha_p\prod_{l=1}^3(V_lz_p)+\beta_p\prod_{l=1}^3(V_l\bar{z}_p)\right)(p)=0. \label{flocal}
\end{align}
\end{lemma}
\begin{proof}
    Let $\phi_1=V_1f$ and $\phi_2=V_2V_3f$. Then \eqref{flocal} follows from Lemma \ref{lemmalocal} directly.
\end{proof}

We should note that the ambiguity of sign of leading coefficients is eliminated and the final algebraic relation is presented as linear combination of $\alpha_p=a_p^2$ and $\beta_p=b_p^2$ $(p\in\vp)$.

\subsection{Applications of the algebraic relationships}
In this subsection, we will introduce several applications of the algebraic relationships of the leading coefficients, which impose some constraints on the multiplicity of eigenvalues and the existence of critical eigensections, based on the positions of points in the configuration.
\subsubsection{Upper bound of order and multiplicity}
Let $\lambda$ be an eigenvalue and $V_\lambda$ be the eigenspace, let $V^c_\lambda\subset V_{\lambda}$ be the subspace that consists of critical eigensections of $V_\lambda$, which is $\{0\}$ if $\lambda$ is not critical.

\begin{lemma}\label{morethanonepointleastorder}
    Given an eigensection $f$ of $\lbc$, there exists at least two points $p_1,p_2\in\vp$ such that $\vn_{p_1}(f)=\vn_{p_2}(f)=\ord(f)$.
\end{lemma}
\begin{proof}
    Suppose there is only one point $p\in\vp$ such that $\vn_{p}(f)=\ord(f)$, and any other $p'\neq p\in\vp$ satisfies $\vn_{p'}(f)>\ord(f)\ge 0$. Expand $f$ near $p$ as $f(z_p,\bar{z}_p)=az_p^{n_0+1/2}+b\bar{z}_p^{n_0+1/2}+\mathcal{O}(r^{n_0+3/2})$, where $n_0=\ord(f)$ and $(a,b)\ne (0,0)$.

    Consider a point $q\notin \pm\vp=\{\pm p:p\in\vp\}$, with spherical coordinate $(\varphi,\theta)$ on $\sph$. Set $\phi_1=(J_3^q)^{n_0}f$ and $\phi_2=(J_3^q)^{n_0+1}f$. By Proposition \ref{orderminusone}, $\vn_{p}(\phi_1)=\vn_{p}(\phi_2)+1=0$ while $\vn_{p'}(\phi_1)=\vn_{p'}(\phi_2)+1>0$ for other $p'\in\vp$. From \eqref{Jcplxcoor}, Lemma \ref{Liealgetrans} and \ref{lemmalocal}, one can conclude $0=\sin^{2n_0+1}\theta(e^{i(2n_0+1)\varphi}a^2-e^{-i(2n_0+1)\varphi}b^2)$. Varying $\varphi$ in the spherical coordinate of $q$ yields $(a,b)=(0,0)$ and gives a contradiction. 
\end{proof}

\begin{corollary}\label{non-criticalmulbound}
$\dim_{\mathbb{C}}\,V_\lambda/V^c_\lambda\le 4n-2$, where $2n=|\vp|$.
\end{corollary}
\begin{proof}
Let $\vp=\{p_1,\cdots,p_{2n}=(0,0,-1)\}$. Suppose $\dim_{\mathbb{C}}V_\lambda/V^c_\lambda\ge 4n-1$. Then, we can choose $4n-1$ linearly independent eigensections, denoted as $f_1,\cdots,f_{4n-1}$, none of them are critical. Given that each base vector must vanish to order $\frac{1}{2}$ at one of the $2n$ points, and considering that the number of base vectors is $4n-1$, we can assume $\vn_{p_1}(f_1)=\vn_{p_1}(f_2)=0$ for $p_1\in \vp$ after a rotation if necessary. Using the leading coefficients of $f_1,f_2$, by linear combinations, 
there exists $4n-3$ linearly independent eigensections, which we denote as $\phi_1,\cdots,\phi_{4n-3}$, such that $\vn_{p_1}(\phi_i)\ge 1$. 

Repeating this reduction process, we eventually obtain at least one nonzero eigensection $g$, such that $\vn_{p_i}(g)\ge 1$ for $i=1,\cdots,2n-1$. The assertion $\vn_{p_{2n}}(g)=0$ contradicts Lemma \ref{morethanonepointleastorder}, while $\vn_{p_{2n}}(g)\ge 1$ contradicts  the non-critical nature of $f_1,\cdots,f_{4n-1}$. Thus $\dim_{\CC}\,V_\lam/V_\lam^c\le 4n-2$.
\end{proof}

\begin{corollary}\label{upperorder}
    Consider an eigensection $f$ of $\lbc$. Assume $\vp$ contains no antipodal points, then $\ord(f)\le n-1$. 
\end{corollary}
\begin{proof}
    We prove this by contradiction. Suppose there is an eigensection $f$ with $\ord(f)\ge n$, i.e. $\vn_p(f)\ge n$ for any $p\in\vp$. After some rotations, we may assume $(0,0,-1)=p_1\in\vp$, and $\vn_{p_1}(f)=\ord(f)=n$. Consider a point $q\in\sph$ with spherical coordinate $(\varphi,\theta)$, let $\phi_1=J^{q}_3(\prod_{i=2}^{n}J_3^{p_i})f$ and $\phi_2=J^{q}_3(\prod_{i=n+1}^{2n}J_3^{p_i})f$. Note that $\vn_p(\phi_1)>\vn_{p_1}(\phi_1)= 0$ and $\vn_{p}(\phi_2)> \vn_{p_1}(\phi_2)= -1$ for any $p\in\vp\setminus\{p_1\}$.

    By Lemma \ref{lemmalocal}, we have $$\sin\theta\big(\prod_{j=2}^{2n}\sin\theta_j\big)(a_1^2\exp(i\sum_{j=2}^{2n}\varphi_j+i\varphi)-b_1^2\exp(-i(\sum_{j=2}^{2n}\varphi_j+i\varphi)))=0,$$ where $(a_1,b_1)$ is the leading coefficients of $f$ at $p_1$ and $(\varphi_i,\theta_i)$ are spherical coordinates of $p_i$. By assumption, $-p_1\notin \vp$, hence  varying $\varphi$ yields $a_1=b_1=0$. This makes a contradiction. 
\end{proof}

\begin{corollary}\label{uppermul}
    Let $\vp$ be a configuration containing no antipodal points and $\lambda$ be an eigenvalue of $\lbc$. Then $\mul\,\lambda\le 2n(2n-1)$.
\end{corollary}
\begin{proof}
    Let $V_\lambda^j=\{f\in V_\lambda:\ord(f)\ge j\}$ for $j\ge 0$. By Corollary \ref{upperorder}, we have a filtration $V_\lambda=V_\lambda^0\supset V_\lambda^1\supset\cdots\supset V_\lambda^{n-1}\supset 0$.

    Let $W_{\lambda}^{j-1}$ be the orthogonal (with respect to the $L^2$ inner product) complement of $V^{j-1}_\lambda$ in $V^j_\lambda$, then $W_{\lam}^{j-1}\oplus V_{\lam}^j=V_{\lam}^{j-1}$. By Corollary \ref{non-criticalmulbound}, $\dim_{\CC}\, W_\lambda^0\le 4n-2$. Suppose $W_\lambda^1$ has dimension $l$, and let $f_1,\cdots,f_l$ be its basis. Introduce the map $\Psi:W_\lambda^1\to V_\lambda: f\mapsto J_3^q f$ for a fixed $q\notin\pm\vp$. Let $\pi_0:V_\lambda\to W_\lambda^0$ be the quotient map. By Proposition \ref{orderminusone}, $\pi_0\circ \Psi$ is injective. Thus $\dim_{\CC}\, W_\lambda^1\le \dim_{\CC}\, W_\lambda^0\le 4n-2$. Similarly, $\dim_{\CC}\, W_\lambda^j\le 4n-2$ for $1\le j\le n-1$, and $\mul \,\lambda=\dim_{\CC}\, V_\lambda\le 2n(2n-1)$.
\end{proof}

The upper bound of $\mul\,\lambda$ provided above is not optimal, and pursuing such optimization does not hold particular interest for our current purposes. In general, when considering the multiplicity of $k$-th Laplacian eigenvalue on a compact manifold, the best we can anticipate is a bound that is a polynomial in $k$. Such a polynomial growth bound can be found in \cite{HHN1999}.

\subsubsection{Restrictions on the configurations in $\mathcal{C}_4$}
Let's now turn our attention to configurations comprising four points. Consider $\vp=\{p_1,p_2,p_3,p_4\}$, where $p_4=(0,0,-1)$, and $p_i=(\varphi_i,\theta_i)$ for $i=1,2,3$. We denote $z_{p_l}$ by $z_l$ for $l=1,2,3,4$. Similar subscripts are also used for other quantities such as spherical coordinates. Let's begin with a warm-up identity.
\begin{lemma}\label{p4algrel}Let $f$ be a critical eigensection of $\lb$, with asymptotic expansion $f(z_l,\bar{z}_l)=a_lz_l^{3/2}+b_l\bar{z}_l^{3/2}+\mathcal{O}(r^{5/2})$ near each point $z_l$ $(l=1,2,3,4)$. (It could happen that $a_l=b_l=0$ for some $l$.) Set $\alpha_l=a_l^2$ and $\beta_l=b_l^2$ as in Lemma \ref{4.10}, we then have the following relation:
\begin{align}
\sin\theta_1\sin\theta_2\sin\theta_3[ e^{i(\varphi_1+\varphi_2+\varphi_3)}\alpha_4-e^{-i(\varphi_1+\varphi_2+\varphi_3)}\beta_4]=0.\label{alrforpole}
\end{align}
\end{lemma}
\begin{proof}
    Let $V_l=J_3^{p_l}(l=1,2,3)$ in \eqref{flocal}. Recall that $J_3^{p_l}=-\sin\theta_l\cos\varphi_l J_1-\sin\theta_l\sin\varphi_l J_2-\cos\theta_l J_3$. By \eqref{Jcplxcoor} we have $(J_3^{p_l}z_l)|_{r=0}=0$ for $r=|z_l|$, $l=1,2,3$. Furthermore, we have
    \begin{align*}
    (J_3^{p_l}z_4)|_{r=0}&=-\sin\theta_l\cos\varphi_l (J_1z_4)|_{r=0}-\sin\theta_l\sin\varphi_l (J_2z_4)|_{r=0}-\cos\theta_l (J_3z_4)|_{r=0}\\
    &=-\frac{1}{2}\sin\theta_l e^{i\varphi_l},
    \end{align*} and $(J_3^{p_l}\bar{z}_4)|_{r=0}=\frac{1}{2}\sin\theta_le^{-i\varphi_l}$, where $r=|z_4|$.
Substituting these into \eqref{flocal} completes the proof.
\end{proof}

\begin{proposition}\label{partialanti}
    Assuming $\vp\in\mathcal{C}_4$ contains one and only one pair of antipodal points, an eigensection $f$ of $\lbc$ cannot be critical.
\end{proposition}
\begin{proof}

Without loss of generality, we can assume $p_1=-p_3$ and $p_2\neq -p_4$. By Proposition \ref{orderminusone}, $J_3^{p_1}=-J_3^{p_3}$ keeps the orders $\vn_{p_1}$ and $\vn_{p_3}$ invariant while reducing $\vn_{p_1}$ and $\vn_{p_4}$ by one. Therefore, suppose an eigensection $f$ is critical, after applying some infinitesimal rotations, we can always assume $\ord(f)=1$ and $\vn_{p_4}(f)=1$.

If we denote $(a_4,b_4)$ as the leading coefficients of the critical eigensection $f$ at $p_4$, then at least one of $a_4$ or $b_4$ is nonzero. Set $V_l=J_3^{p_l}$ for $l=1,2$, and $V_3=J_3^q$ in \eqref{flocal}, where $q\in \sph$ with coordinates $(\varphi,\theta)$ to be determined. Note that $(V_lz_l)|_{r=0}=0$ for $l=1,2$ and $(V_1z_3)|_{r=0}=0$, combining with Lemma \ref{4.10}, we then obtain:
\begin{align}
0=\sin\theta_1\sin\theta_2\sin\theta[e^{i(\varphi_1+\varphi_2+\varphi)}a_4^2-e^{-i(\varphi_1+\varphi_2+\varphi)}b_4^2].\label{noncrt1}
\end{align}
In the equation above, $(\varphi_l,\theta_l)$ are coordinates of $p_l$ for $l=1,2$. Let $\theta\neq 0,\pi$ and vary $\varphi$, we conclude that $a_4=b_4=0$, which contradicts with $\vn_{p_4}(f)=1$.
\end{proof}

\begin{lemma}\label{npf==1}
    Suppose $\vp\in\mathcal{C}_4$ contains no antipodal points and that $f$ is a critical eigensection of $\lbc$, then $\forall p\in\vp, \vn_p(f)=1$. In particular, $f$ is a non-degenerate critical eigensection.
\end{lemma}
\begin{proof}

Suppose for some $p\in\vp$, $\vn_{p}(f)>1$. After rearrangement of points in $\vp$ and applying some infinitesimal rotations to the eigensection if necessary, without loss of generality, we can assume $\vn_{p_4}(f)=1$ and $\vn_{p_3}(f)\geq 2$. Set $V_l=J_3^{p_l}$ for $l=1,2$, and $V_3=J_3^q$ in \eqref{flocal}, with $q=(\varphi,\theta)$ to be determined. Since $(V_lz_l)|_{r=0}=0$ for $l=1,2$, and $\vn_{p_3}(V_3f)\geq 1$, we obtain:
$$\sin\theta_1\sin\theta_2\sin\theta[e^{i(\varphi_1+\varphi_2+\varphi)}a_4^2-e^{-i(\varphi_1+\varphi_2+\varphi)}b_4^2]=0.$$ Varying $q$ implies $a_4=b_4=0$, which leads to a contradiction as $\vn_{p_4}(f)=1.$
\end{proof}

\begin{proposition}
    \label{rk1}
Suppose $\vp\in \mathcal{C}_4$ contains no antipodal points and that $f_1,f_2$ are two critical eigensections with the same eigenvalue, then $f_1=cf_2$ for some complex number $c$. As a consequence, we have $\dim_{\CC}\,V^c_\lambda\le 1$.
\end{proposition}
\begin{proof}

Let $p_4=(0,0,-1)$. From Lemma \ref{npf==1}, we deduce that $\vn_{p_4}(f_s)=1$ for $s=1,2$. Fix a complex coordinate $w_4$ at $p_4$ in the branched covering $\Sigma_\vp$, and let $(a_{4,s},b_{4,s})$ for $s=1,2$ be the leading coefficients of $f_s$ at $p_4$. By \eqref{alrforpole}, we could assume $a_{4,2}\neq 0$, then define $c=\frac{a_{4,1}}{a_{4,2}}$. If $f_3:=f_1-cf_2$ is nonvanishing, it is clear that it is another critical eigensection. Moreover, the leading coefficients of $f_3$ at $p_4$ are $(a_{4,3}=0,b_{4,3})$. By Lemma \ref{p4algrel}, we find $b_{4,3}^2=e^{2i(\varphi_{1}+\varphi_{2}+\varphi_{3})}a_{4,3}^2=0$. Consequently, $b_{4,3}=0$ and $\vn_{p_4}(f_3)>1$. It follows from Lemma \ref{npf==1} that $f_3\equiv 0$, and hence $f_1=cf_2$.
\end{proof}

\section{Finite group representations and critical eigensections}

Consider a finite subgroup $G$ of $\mathrm{O}(3)$, that preserves a configuration $\vp\in\Cn$. It can be naturally lifted to a finite group $\hat{G}$ acting isometrically on the conical surface $(\Sigma_{\vp},\pi^\ast\mathrm{ds}^2)$. Consequently, each Laplacian eigenspace of $\Sigma_\vp$ can be interpreted as a finite-dimensional representation of $\widehat{G}$. 

In this section, we begin by delve into a detailed discussion of the spectra of configurations in $\mathcal{C}_2$, serving as an instructive example for the subsequent. Next we present some basic properties of lifting of group actions in Section \ref{grpaction}.  

Finite group representation theory enables us to deduce intriguing results. For instance, in Section \ref{tetrahedralcase}, we establish the rigidity of the Taubes-Wu tetrahedral eigensections. Finally, in Section \ref{5.4}, we demonstrate the non-existence of critical eigenvalues on $\lb$ when $\vp$ forms a cross.

\subsection{The case of $n=1$}
In this subsection, we will provide a description of $\ZT$ eigenvalues and eigensections for the configuration consisting of precisely one pair of antipodal points. Furthermore, we will explore the behavior of $\ZT$ eigenvalues as one of the two points in the configuration varies.

The following result can be derived by standard Fourier series separation of variables, as mentioned in the Appendix of \cite{taubeswu2020examples}. An application of representation theory of $\mathfrak{so}(3)$ provide another interesting proof. This is the classical method used in the calculation of spherical harmonics, as described in \cite[Chapter VI, Complements A]{quantum}. Additionally, this indicates that the assumption of nonexistence of antipodal points in the configuration in the Corollary \ref{uppermul} is necessary.

\begin{proposition}\label{antipodalspectrum}
    Let $\vp=\{(0,0,\pm1)\}\in\mathcal{C}_2$. The spectrum of $\Delta$ on $\lb$ is $\{(l-\frac12)(l+\frac12):l=1,2,\cdots\}$. For $\lambda=(l-\frac12)(l+\frac12)$, the multiplicity $\mul\,\lambda=2l$. Let $(\vp,\theta)$ be the spherical coordinate of $\sph$, then a basis of $V_\lambda$ could be written as 
    $\{\cos ((j-\frac{1}{2})\varphi) G_l^j(\theta),\; \sin((j- \frac{1}{2})\varphi) G_l^j(\theta)\}_{1\le j\le l}$,
    where $G_l^j(\theta)=\sin^{\frac{1}{2}-j}\theta\mathrm{d}^{l-j}(\sin^{2l-1}\theta)/\mathrm{d}(\cos\theta)^{l-j}$.
\end{proposition}

\begin{proposition}[\cite{taubeswu2020examples}, Appendix, Proposition A]\label{nocriticaltwopoints}
    For configuration $\vp\in\mathcal{C}_2\setminus\{\{\pm p\}:p\in\sph\}$, there are no critical eigenvalues. Moreover, for an eigensection $f$ of $\lbc$, $\vn_p(f)=\ord(f)=0$ for any $p\in\vp$.
\end{proposition}
\begin{proof}
    Suppose an eigensection $f$ of $\lb$ is critical. Consider $p\in\vp$, then $-p\notin \vp$, and $\vn_p(J_3^pf)>\ord(J_3^pf)$ by Proposition \ref{orderminusone}. This contradicts with Lemma \ref{morethanonepointleastorder}. And we can also deduce from this lemma that $\vn_p(f)=\ord(f)=0$ for any $p\in\vp$. 
\end{proof}

\begin{lemma}
    Assume $p_2\ne -p_1$ and $p_1=(0,0,1)$. Consider an eigensection $f$ of $\mathcal{I}^{\CC}_{\vp(s)}$, with asymptotic expansion $f(z_l,\bar{z}_l)=a_lz_l^{1/2}+b_l\bar{z}_l^{1/2}+\mathcal{O}(r^{3/2})$, where $z_l$ are the complex coordinate $z_{p_l}$ for $l=1,2$. It follows that $a_2^2=b_2^2$. \label{coefficientn1}
\end{lemma} 
\begin{proof}
    Consider $\phi_1=f$ and $\phi_2=J_3f$. It follows that $\vn_{p_1}(\phi_2)=\vn_{p_2}(f)$ and $\vn_{p_2}(\phi_2)=-1$. Recall that $J_3=\sin\theta_2 J_1^{p_2}-\cos\theta_2 J_3^{p_2}$ and hence $(J_3z_2)|_{r=0}=\frac12 \sin\theta_2$. Applying Lemma \ref{lemmalocal} we then conclude $\frac12 \sin\theta_2(a_2^2-b_2^2)=0$. Since $p_2\ne p_1$, $a_2^2=b_2^2$.
\end{proof}

Next we discuss the symmetries of configurations in $\mathcal{C}_2$. Recall that a rotation always preserves Laplacian spectrum and the vanishing order of eigensections. Therefore, for a given configuration $\vp=\{p_1,p_2\}$, we may assume $p_1=(0,0,1)$ and $p_2=(\sin\theta,0,\cos\theta)$ for $\theta\ne 0$. 

We characterize $\sph$ as $\overline{\CC}$ by stereographic projection from the north pole, and identify $p_1$ with $\infty$ and $p_2$ with $t$ for some $t\in [0,+\infty)$ depending on $\theta$. When $\theta=\pi$ or $t=0$, $\vp$ is the antipodal configuration. Denote $\vp=\{+\infty,t\}$ as $\vp(t)$. Let $\pi_t:\Sigma_t\to\sph$ be the corresponding covering branched at $\vp(t)$.

Consider the involution $\iota: \overline{\CC}\to\overline{\CC}:z\mapsto \bar{z}$, which can be lifted to the Riemann surface $\Sigma_s:=\Sigma_{\vp(s)}\to \overline{\CC}$. Such a lift is not unique. We fix a choice as follows. 

Consider a projective curve in $\mathbb{C}P^2=\{[w:z:y]:w,z,y\in\CC\}$ defined by the equation $w^2=y(z-ty)$. This curve is nonsingular and can be identified with $\Sigma_t$. The projection $\pi_{\vp(t)}=:\pi_t$ can be interpreted as $\pi_t:\Sigma_t\to\overline{\CC}:[w:z:y]\mapsto z/y$. In particular, $\pi_t([w:z:1])=z$ and $\pi_t([0:1:0])=\infty$. The canonical involution $\ei:\Sigma_t\to\Sigma_t$ maps $[w:z:1]$ to $[-w:z:1]$, and $[0:t:1]$ and $[0:1:0]$ are its fixed points. The conjugation on $\overline{\CC}$ can be lifted to $\tilde{\iota}:\Sigma_t\to\Sigma_t:[w:z:y]\mapsto [\bar{w}:\bar{z}:\bar{y}]$ or $\ei\circ\tilde{\iota}:[w:z:y]\mapsto [-\bar{w}:\bar{z}:\bar{y}]$. We shall use the former in the following. 

Note that $\tilde{\iota}$ commutes with $\ei$, giving rise to a $\ZT$-representation on each Laplacian eigensubspace of $\osb(\Sigma_t)$. Suppose $\lambda$ is a $\ZT$ eigenvalue of $\lbct$. The corresponding complex eigenspace $V_\lambda$ can be decomposed into $V_\lambda=V_\lambda^+\oplus V_\lambda^-$, where $V_\lambda^\pm$ is the eigenspace of $\tilde{\iota}$ with eigenvalue $\pm1$.

An interesting phenomenon, known as the simple splitting of eigenvalues, occurs when $p_2=0$ is perturbed. The proof presented here serves as a warm-up of the subsequent subsections.

\begin{proposition}\label{mul1twopt}
    For configuration $\vp\in\mathcal{C}_2\setminus\{\{\pm p\}:p\in\sph\}$ and a $\ZT$ eigenvalue $\lambda$ of $\lbc$, it always holds that $\mul\,\lambda=\dim_{\CC}\,V_{\lambda}=1$. Indeed, each $V_\lambda$ is an irreducible $\ZT$-representation.
\end{proposition}
\begin{proof}
    Given an eigensection $f\in V_\lambda$, where $\lambda$ is some $\ZT$ eigenvalue of $\lbct$. Then locally, we can expand the corresponding odd function $\tilde{f}$ as $\tilde{f}(w,\bar{w})=aw^{2n+1}+b\bar{w}^{2n+1}+\mathcal{O}(r^{2n+3})$ near $p_2=[0:t:1]\in\Sigma_t$, for $w$ a component of homogeneous coordinate $[w:z:y]$.
    
    If $f\in V_\lambda^{\pm}$, then $\pm \tilde{f}(w,\bar{w})=\tilde{\iota}.\tilde{f}(w,\bar{w})=\tilde{f}\circ\tilde{\iota}(w,\bar{w})=bw^{2n+1}+a\bar{w}^{2n+1}+\mathcal{O}(r^{2n+3})$. Thus $b=\pm a$.

    Suppose $\dim\, V_\lambda^+>1$. One can choose $f_1,f_2\in V_\lambda^{+}$ that are linearly independent. By Lemma \ref{coefficientn1}, $\vn_{p_2}(f_1)=\vn_{p_1}(f_2)=0$. By multiplying some complex numbers, we may assume the leading coefficients of both $f_1$ and $f_2$ at $p_2$ are $(1,1)$. Then $h_1:=f_1-f_2\in V_\lambda^+$ and $\vn_{p_2}(h_1)>0$, which contradicts Proposition \ref{nocriticaltwopoints}. Therefore, $\dim\, V_\lambda^+\le 1$. The same holds for $V_\lambda^-$. 

    If both $V_\lambda^+$ and $V_\lambda^-$ are nontrivial, then we can select $f_+\in V_\lambda^+$ and $f_-\in V_\lambda^-$ with leading coefficients at $p_2$ being $(1,1)$ and $(i,-i)$ respectively. Note also that if $f\in V_\lambda^{\pm}$ with $\vn_{p_2}(f)>0$, then $J_3f$ is still an eigensection in $V_\lambda^{\pm}$. Thus we may assume $\vn_{p_2}(f_+)=\vn_{p_2}(f_-)$. Let $h_2:=f_++f_-$, then its leading coefficients at $p_2$ is $(1+i,1-i)$, contradicting Lemma \ref{coefficientn1}. Consequently, $\mul\,\lambda =1$.
\end{proof}

It follows that on the space $\mathcal{C}_2$, outside the subset $\{\{\pm p\}:p\in\overline{\CC}\}\cong \overline{\CC}$, each configuration has simple $\ZT$ eigenvalues only. This fact serves as an example of Theorem \ref{residualsimple}.

Consider the path $\vp(t)=\{\infty,t\}$. As configuration varies on this path, the spectrum varies. We can describe this spectral flow more precisely. Firstly, by Lemma \ref{TWformula}, the eigenvalue function $\lambda_k(t)$, i.e. the $k$-th $\ZT$ eigenvalue of $\lbct$, is differentiable in $t$. 

Secondly, the two types of $\ZT$-symmetries given by $\tilde{\iota}$, divide the $\ZT$-spectrum of $\lbct$ into two parts. These correspond to two eigenvalue problems on the disk $\overline{\mathbb{H}}=\{z\in \overline{\CC}: \mathrm{Im}~z\ge 0\}$.

The set of fixed points of $\tilde{\iota}$ on $\Sigma_t$ is $\Gamma_t:=\{[w:x:1]:w\in\RR, x\ge t, w^2=x-t\}$. And $\Gamma_t^c:=\{[w:x:1]:w\in i\RR, x\le t, w^2=t-x\}$ is the set of fixed points of $\ei\circ\tilde{\iota}$. 
\begin{proposition}
    Each $\ZT$ eigenvalue of $\lbct$, such that $V_\lambda=V_\lam^+$, is an eigenvalue of spherical Laplacian on $\overline{\mathbb{H}}$, with mixed boundary conditions:
    \begin{align}\frac{\partial}{\partial \nu}u|_{\Gamma_t}\equiv0\text{ while } u|_{\Gamma_t^c}\equiv 0.\label{twoptbdcond1}\end{align}
    Similarly, each $\ZT$ eigenvalue $\lam$ such that $V_\lam=V_\lam^-$ is that with mixed boundary conditions:  
    \begin{align}\frac{\partial}{\partial \nu}u|_{\Gamma_t^c}\equiv 0\text{ while }u|_{\Gamma_t}\equiv 0. \label{twoptbdcond2}\end{align}
\end{proposition}
\begin{proof}
    For an odd function $\tilde{f}\in V_\lambda^+$, defined on $\Sigma_t$, since $\tilde{f}\circ\tilde{\iota}=f$, while $\tilde{f}\circ\ei\circ\tilde{\iota}=-\tilde{f}$, the boundary conditions of $\tilde{f}$ on $\partial\overline{\mathbb{H}}=\Gamma_t\cup\Gamma_t^c$ follow. 

    Conversely, consider an eigenfunction of the spherical Laplacian $\Delta$, with boundary conditions $\frac{\partial}{\partial \nu}u|_{\Gamma_t}\equiv0\text{ while } u|_{\Gamma_t^c}\equiv 0$. We can extend it to be an odd function $\tilde{f}$ on $\Sigma_t$, by setting $\tilde{f}|_{\overline{\mathbb{H}}}=u$, $\tilde{\iota}.\tilde{f}=\tilde{f}$ and $\tilde{f}\circ\ei=-\tilde{f}$.

    Similar arguments are applied to $ V_\lambda^-$. 
\end{proof}

As an illustration, when $t=0$, and $\lambda=(l-\frac12)(l+\frac12)$, $\dim\, V_\lambda^+=\dim\, V_\lambda^-=l$. The basis of $V_\lambda^+$ is given by sections $\cos (j-\frac{1}{2})\varphi G_l^j(\theta)$, as presented in Proposition \ref{antipodalspectrum}. The remaining sections form a basis of $V_\lambda^-$. These sections satisfy the boundary conditions given above. We name $\ZT$ eigenvalues corresponding to boundary condition \eqref{twoptbdcond1}, in other words, with trivial $\ZT$ (generated by $\tilde{\iota}$) symmetric condition, as the $\tilde{\iota}$-invariant part. On the other hand, those corresponding to \eqref{twoptbdcond2} or the alternating $\ZT$ symmetric condition, are named as the $\tilde{\iota}$-equivariant part.

\begin{proposition}
    As $t\to +\infty$, the $\tilde{\iota}$-invariant part of $\ZT$ eigenvalues strictly increases, while the $\tilde{\iota}$-equivariant part strictly decreases.
\end{proposition}
\begin{proof}
    The fact that the $\tilde{\iota}$-invariant (resp. $\tilde{\iota}$-equivariant) part is increasing (resp. decreasing) can be deduced from the observation that the Dirichlet (resp. Neumann) boundary $\Gamma_t^c$ broadens while the Neumann (resp. Dirichlet) boundary $\Gamma_t$ contracts as $t\to+\infty$.

    The strictly monotonicity can be derived from Lemma \ref{TWformula} and leading coefficients calculated in the proof of Proposition \ref{mul1twopt}.
\end{proof}

By Theorem \ref{continuousztdegenerate}, as $t\to +\infty$, $\lambda_k(t)$ converges to the $k$-th Laplacian eigenvalue of the round sphere. 

Suppose $\lambda_k(0)=(l-\frac12)(l+\frac12)$ is an eigenvalue of $\mathcal{I}_{\vp(0)}^\CC$, where $\vp(\infty)$ is the antipodal configuration. The function $\lambda_k(t)$ is strictly monotonic in $t$. If $\lambda_k(0)$ belongs to $\tilde{\iota}$-invariant part, it increases and converges to $l(l+1)$ as $t\to+\infty$. If $\lam_k(0)$ belongs to $\tilde{\iota}$-equivariant part, it decreases and converges to $(l-1)l$ as $t\to+\infty$.

Additionally, when $t$ is perturbed, the $2l$ dimensional eigenspace $V_{(l-1/2)(l+1/2)}$ splits into two parts: the $\tilde{\iota}$-invariant and the $\tilde{\iota}$-equivariant part. Each part contains $l$ simple eigensections, which converge to the nearby spherical eigenvalues as $t\to +\infty$. 

On the other hand, for the spherical eigenvalue $l(l+1)$, which has multiplicity $2l+1$, there are $l$ eigenfunctions come from eigensections with eigenvalue $(l-\frac12)(l+\frac12)$, and the remaining $l+1$ come from those with eigenvalue $(l+\frac12)(l+\frac32)$.  

This description of spectral flow on $\mathcal{C}_2$ is summarized in the Figure \ref{spectralflowc2}. It should be emphasized that most of this description have been provided in \cite[Appendix]{taubeswu2020examples}.

\begin{figure}[!h]
    \centering

\tikzset{every picture/.style={line width=0.75pt}} %set default line width to 0.75pt        

\begin{tikzpicture}[x=0.75pt,y=0.75pt,yscale=-1,xscale=1]

%Curve Lines [id:da4869746094284375] 
\draw    (221.47,69) .. controls (234.47,56) and (243.47,48) .. (260.23,47.3) ;
%Curve Lines [id:da5078084419798192] 
\draw    (361.23,195.3) .. controls (316.47,159) and (276.47,93) .. (260.23,47.3) ;
%Curve Lines [id:da6598632606173811] 
\draw [color={rgb, 255:red, 74; green, 144; blue, 226 }  ,draw opacity=0.68 ][line width=1.5]    (52.8,137.6) .. controls (57.47,162.6) and (74.47,193.6) .. (108.6,194) .. controls (142.73,194.4) and (168.47,154.6) .. (167.47,118.6) .. controls (166.47,82.6) and (142.47,44.6) .. (108.6,44.8) ;
%Shape: Circle [id:dp7586338813273177] 
\draw   (34,119.4) .. controls (34,78.2) and (67.4,44.8) .. (108.6,44.8) .. controls (149.8,44.8) and (183.2,78.2) .. (183.2,119.4) .. controls (183.2,160.6) and (149.8,194) .. (108.6,194) .. controls (67.4,194) and (34,160.6) .. (34,119.4) -- cycle ;
%Curve Lines [id:da16464740294031488] 
\draw [color={rgb, 255:red, 208; green, 2; blue, 27 }  ,draw opacity=0.81 ][line width=2.25]    (52.8,137.6) .. controls (46.47,92.6) and (66.47,44.6) .. (108.6,44.8) ;
%Shape: Circle [id:dp044776371722426545] 
\draw  [fill={rgb, 255:red, 74; green, 144; blue, 226 }  ,fill opacity=0.49 ] (104.3,44.8) .. controls (104.3,42.43) and (106.23,40.5) .. (108.6,40.5) .. controls (110.97,40.5) and (112.9,42.43) .. (112.9,44.8) .. controls (112.9,47.17) and (110.97,49.1) .. (108.6,49.1) .. controls (106.23,49.1) and (104.3,47.17) .. (104.3,44.8) -- cycle ;
%Curve Lines [id:da3176434666100738] 
\draw [color={rgb, 255:red, 74; green, 144; blue, 226 }  ,draw opacity=0.81 ][line width=2.25]    (52.8,137.6) .. controls (56.47,158.6) and (70.47,192.6) .. (108.6,194) ;
%Straight Lines [id:da8308429730855826] 
\draw [line width=1.5]    (221,47) -- (503.47,47.6) ;
%Straight Lines [id:da561911350035057] 
\draw [line width=1.5]    (220,195) -- (502.47,195.6) ;
%Shape: Circle [id:dp6536354358186045] 
\draw  [color={rgb, 255:red, 0; green, 0; blue, 0 }  ,draw opacity=0.65 ][fill={rgb, 255:red, 74; green, 144; blue, 226 }  ,fill opacity=0.49 ] (104.3,194) .. controls (104.3,191.63) and (106.23,189.7) .. (108.6,189.7) .. controls (110.97,189.7) and (112.9,191.63) .. (112.9,194) .. controls (112.9,196.37) and (110.97,198.3) .. (108.6,198.3) .. controls (106.23,198.3) and (104.3,196.37) .. (104.3,194) -- cycle ;
%Shape: Circle [id:dp5320440698825193] 
\draw  [fill={rgb, 255:red, 143; green, 216; blue, 60 }  ,fill opacity=0.55 ] (48.5,137.6) .. controls (48.5,135.23) and (50.43,133.3) .. (52.8,133.3) .. controls (55.17,133.3) and (57.1,135.23) .. (57.1,137.6) .. controls (57.1,139.97) and (55.17,141.9) .. (52.8,141.9) .. controls (50.43,141.9) and (48.5,139.97) .. (48.5,137.6) -- cycle ;
%Curve Lines [id:da26786942198789654] 
\draw    (361.23,195.3) .. controls (269.47,193) and (258.47,80) .. (260.23,47.3) ;
%Curve Lines [id:da9811954487759897] 
\draw    (361.23,195.3) .. controls (295.47,182.6) and (263.47,98.6) .. (260.23,47.3) ;
%Curve Lines [id:da54162729137894] 
\draw    (361.23,195.3) .. controls (353.47,156) and (298.47,47) .. (260.23,47.3) ;
%Curve Lines [id:da5635830002259539] 
\draw    (360.07,195.3) .. controls (451.83,193) and (465.47,80) .. (465.8,47.3) ;
%Curve Lines [id:da629795855679542] 
\draw    (360.07,195.3) .. controls (425.83,182.6) and (456.47,99) .. (465.8,47.3) ;
%Curve Lines [id:da19689393188133542] 
\draw    (360.07,195.3) .. controls (367.83,156) and (422.2,48) .. (465.8,47.3) ;
%Curve Lines [id:da7545408394438915] 
\draw    (360.07,195.3) .. controls (404.83,159) and (443.47,88) .. (465.8,47.3) ;
%Straight Lines [id:da4904777230145183] 
\draw [color={rgb, 255:red, 126; green, 211; blue, 33 }  ,draw opacity=1 ]   (227.47,138) -- (501.47,138) ;
%Straight Lines [id:da6913622466973446] 
\draw [color={rgb, 255:red, 126; green, 211; blue, 33 }  ,draw opacity=1 ] [dash pattern={on 4.5pt off 4.5pt}]  (52.8,137.6) -- (227.47,138) ;
%Shape: Circle [id:dp9684819845810066] 
\draw  [fill={rgb, 255:red, 255; green, 255; blue, 255 }  ,fill opacity=1 ] (355.33,195.3) .. controls (355.33,192.69) and (357.45,190.57) .. (360.07,190.57) .. controls (362.68,190.57) and (364.8,192.69) .. (364.8,195.3) .. controls (364.8,197.91) and (362.68,200.03) .. (360.07,200.03) .. controls (357.45,200.03) and (355.33,197.91) .. (355.33,195.3) -- cycle ;
%Shape: Circle [id:dp9576735048709759] 
\draw  [fill={rgb, 255:red, 255; green, 255; blue, 255 }  ,fill opacity=1 ] (272.53,138.28) .. controls (272.53,136.23) and (274.2,134.57) .. (276.25,134.57) .. controls (278.3,134.57) and (279.97,136.23) .. (279.97,138.28) .. controls (279.97,140.34) and (278.3,142) .. (276.25,142) .. controls (274.2,142) and (272.53,140.34) .. (272.53,138.28) -- cycle ;
%Shape: Circle [id:dp5745859562604714] 
\draw  [fill={rgb, 255:red, 255; green, 255; blue, 255 }  ,fill opacity=1 ] (286.53,138.28) .. controls (286.53,136.23) and (288.2,134.57) .. (290.25,134.57) .. controls (292.3,134.57) and (293.97,136.23) .. (293.97,138.28) .. controls (293.97,140.34) and (292.3,142) .. (290.25,142) .. controls (288.2,142) and (286.53,140.34) .. (286.53,138.28) -- cycle ;
%Shape: Circle [id:dp5156157518105928] 
\draw  [fill={rgb, 255:red, 255; green, 255; blue, 255 }  ,fill opacity=1 ] (305.53,138.28) .. controls (305.53,136.23) and (307.2,134.57) .. (309.25,134.57) .. controls (311.3,134.57) and (312.97,136.23) .. (312.97,138.28) .. controls (312.97,140.34) and (311.3,142) .. (309.25,142) .. controls (307.2,142) and (305.53,140.34) .. (305.53,138.28) -- cycle ;
%Shape: Circle [id:dp048465067626062064] 
\draw  [fill={rgb, 255:red, 255; green, 255; blue, 255 }  ,fill opacity=1 ] (335.53,138.28) .. controls (335.53,136.23) and (337.2,134.57) .. (339.25,134.57) .. controls (341.3,134.57) and (342.97,136.23) .. (342.97,138.28) .. controls (342.97,140.34) and (341.3,142) .. (339.25,142) .. controls (337.2,142) and (335.53,140.34) .. (335.53,138.28) -- cycle ;
%Shape: Circle [id:dp8713545189844998] 
\draw  [fill={rgb, 255:red, 255; green, 255; blue, 255 }  ,fill opacity=1 ] (379.53,138.28) .. controls (379.53,136.23) and (381.2,134.57) .. (383.25,134.57) .. controls (385.3,134.57) and (386.97,136.23) .. (386.97,138.28) .. controls (386.97,140.34) and (385.3,142) .. (383.25,142) .. controls (381.2,142) and (379.53,140.34) .. (379.53,138.28) -- cycle ;
%Shape: Circle [id:dp02974287733883685] 
\draw  [fill={rgb, 255:red, 255; green, 255; blue, 255 }  ,fill opacity=1 ] (408.53,138.28) .. controls (408.53,136.23) and (410.2,134.57) .. (412.25,134.57) .. controls (414.3,134.57) and (415.97,136.23) .. (415.97,138.28) .. controls (415.97,140.34) and (414.3,142) .. (412.25,142) .. controls (410.2,142) and (408.53,140.34) .. (408.53,138.28) -- cycle ;
%Shape: Circle [id:dp6183607377180069] 
\draw  [fill={rgb, 255:red, 255; green, 255; blue, 255 }  ,fill opacity=1 ] (428.53,138.28) .. controls (428.53,136.23) and (430.2,134.57) .. (432.25,134.57) .. controls (434.3,134.57) and (435.97,136.23) .. (435.97,138.28) .. controls (435.97,140.34) and (434.3,142) .. (432.25,142) .. controls (430.2,142) and (428.53,140.34) .. (428.53,138.28) -- cycle ;
%Shape: Circle [id:dp5224078277192887] 
\draw  [fill={rgb, 255:red, 255; green, 255; blue, 255 }  ,fill opacity=1 ] (442.53,138.28) .. controls (442.53,136.23) and (444.2,134.57) .. (446.25,134.57) .. controls (448.3,134.57) and (449.97,136.23) .. (449.97,138.28) .. controls (449.97,140.34) and (448.3,142) .. (446.25,142) .. controls (444.2,142) and (442.53,140.34) .. (442.53,138.28) -- cycle ;
%Straight Lines [id:da27069384060251345] 
\draw [line width=1.5]  [dash pattern={on 5.63pt off 4.5pt}]  (191.47,195) -- (220,195) ;
%Straight Lines [id:da5758730727771171] 
\draw [line width=1.5]  [dash pattern={on 5.63pt off 4.5pt}]  (502.47,195.6) -- (525,195.6) ;
%Straight Lines [id:da2074842392720302] 
\draw [line width=1.5]  [dash pattern={on 5.63pt off 4.5pt}]  (192.47,47) -- (221,47) ;
%Straight Lines [id:da26864181676387466] 
\draw [line width=1.5]  [dash pattern={on 5.63pt off 4.5pt}]  (503.47,47.6) -- (525,47.6) ;
%Straight Lines [id:da0327028321430598] 
\draw [color={rgb, 255:red, 126; green, 211; blue, 33 }  ,draw opacity=1 ][line width=0.75]  [dash pattern={on 4.5pt off 4.5pt}]  (501.47,138) -- (525,138) ;
%Curve Lines [id:da9524609979231771] 
\draw  [dash pattern={on 4.5pt off 4.5pt}]  (195.47,108) .. controls (204.47,90) and (208.47,83) .. (221.47,69) ;
%Curve Lines [id:da8650597703308842] 
\draw    (224.47,79) .. controls (235.47,65) and (243.47,55) .. (260.23,47.3) ;
%Curve Lines [id:da9474278196074466] 
\draw  [dash pattern={on 4.5pt off 4.5pt}]  (195.47,122) .. controls (204.47,105) and (214.47,93) .. (224.47,79) ;
%Curve Lines [id:da43737664174360136] 
\draw    (224.47,127) .. controls (234.47,115) and (255.47,74) .. (260.23,47.3) ;
%Curve Lines [id:da31535104469821285] 
\draw  [dash pattern={on 4.5pt off 4.5pt}]  (199.47,158) .. controls (212.47,143) and (213.47,141) .. (224.47,127) ;
%Shape: Circle [id:dp6478894534366619] 
\draw  [fill={rgb, 255:red, 255; green, 255; blue, 255 }  ,fill opacity=1 ] (255.5,47.3) .. controls (255.5,44.69) and (257.62,42.57) .. (260.23,42.57) .. controls (262.85,42.57) and (264.97,44.69) .. (264.97,47.3) .. controls (264.97,49.91) and (262.85,52.03) .. (260.23,52.03) .. controls (257.62,52.03) and (255.5,49.91) .. (255.5,47.3) -- cycle ;
%Curve Lines [id:da10831413521864719] 
\draw    (510.47,75) .. controls (497.47,62) and (482.23,48) .. (465.47,47.3) ;
%Curve Lines [id:da12435703349009342] 
\draw  [dash pattern={on 4.5pt off 4.5pt}]  (535.47,110) .. controls (526.47,92) and (523.47,89) .. (510.47,75) ;
%Curve Lines [id:da39672009734355806] 
\draw    (510.47,94) .. controls (499.47,80) and (482.23,55) .. (465.47,47.3) ;
%Curve Lines [id:da5803402582164301] 
\draw  [dash pattern={on 4.5pt off 4.5pt}]  (534.47,129) .. controls (525.47,112) and (520.47,108) .. (510.47,94) ;
%Curve Lines [id:da7662394294573278] 
\draw    (509.47,146.6) .. controls (496.47,133.6) and (472.47,95.6) .. (465.47,47.3) ;
%Curve Lines [id:da174932657015483] 
\draw  [dash pattern={on 4.5pt off 4.5pt}]  (532.47,165.6) .. controls (525.47,162.6) and (519.47,156.6) .. (509.47,146.6) ;
%Curve Lines [id:da6934813312934065] 
\draw    (506.47,63) .. controls (494.47,55) and (491.47,48) .. (465.47,47.3) ;
%Curve Lines [id:da8706900445528318] 
\draw  [dash pattern={on 4.5pt off 4.5pt}]  (535.47,93) .. controls (528.47,79) and (516.47,73) .. (506.47,63) ;
%Curve Lines [id:da8200088087533728] 
\draw    (465.8,47.3) .. controls (486.47,97.6) and (497.47,112.6) .. (509.47,126.6) ;
%Curve Lines [id:da48881460144757827] 
\draw  [dash pattern={on 4.5pt off 4.5pt}]  (532.47,149.6) .. controls (520.47,138.6) and (517.47,135.6) .. (509.47,126.6) ;
%Shape: Circle [id:dp5694733322595165] 
\draw  [fill={rgb, 255:red, 255; green, 255; blue, 255 }  ,fill opacity=1 ] (461.07,47.3) .. controls (461.07,44.69) and (463.19,42.57) .. (465.8,42.57) .. controls (468.42,42.57) and (470.53,44.69) .. (470.53,47.3) .. controls (470.53,49.91) and (468.42,52.03) .. (465.8,52.03) .. controls (463.19,52.03) and (461.07,49.91) .. (461.07,47.3) -- cycle ;

% Text Node
\draw (104.3,202.4) node [anchor=north west][inner sep=0.75pt]    {$0$};
% Text Node
\draw (98,19.73) node [anchor=north west][inner sep=0.75pt]    {$\infty $};
% Text Node
\draw (60,121.73) node [anchor=north west][inner sep=0.75pt]    {$\textcolor[rgb]{0.25,0.46,0.02}{t}$};
% Text Node
\draw (161,96.4) node [anchor=north west][inner sep=0.75pt]    {$\overline{\mathbb{H}}$};
% Text Node
\draw (70,70.4) node [anchor=north west][inner sep=0.75pt]    {$\textcolor[rgb]{0.82,0.01,0.11}{\Gamma }\textcolor[rgb]{0.82,0.01,0.11}{_{t}}$};
% Text Node
\draw (125,154.4) node [anchor=north west][inner sep=0.75pt]    {$\textcolor[rgb]{0.29,0.56,0.89}{\Gamma }\textcolor[rgb]{0.29,0.56,0.89}{_{t}^{c}}$};
% Text Node
\draw (286,208.4) node [anchor=north west][inner sep=0.75pt]    {$\lambda =( l-1/2)( l+1/2)$};
% Text Node
\draw (232,20.4) node [anchor=north west][inner sep=0.75pt]    {$( l-1) l$};
% Text Node
\draw (442,20.4) node [anchor=north west][inner sep=0.75pt]    {$l( l+1)$};
% Text Node
\draw (435,161.4) node [anchor=north west][inner sep=0.75pt]    {$V_{\lambda }^{+}$};
% Text Node
\draw (268,159.4) node [anchor=north west][inner sep=0.75pt]    {$V_{\lambda }^{-}$};
% Text Node
\draw (550,185.4) node [anchor=north west][inner sep=0.75pt]    {$\mathrm{Spec}(\lb)$};
% Text Node
\draw (544,36.4) node [anchor=north west][inner sep=0.75pt]    {$\mathrm{Spec}\left(\mathbb{S}^{2}\right)$};
% Text Node
\draw (544.47,128.4) node [anchor=north west][inner sep=0.75pt]    {$\mathrm{Spec}\big(\mathcal{I}_{\mathfrak{p}(\textcolor[rgb]{0.25,0.46,0.02}{t})}\big)$};

\end{tikzpicture}

    \caption{The spectral flow on $\mathcal{C}_2$.}
    \label{spectralflowc2}
\end{figure}
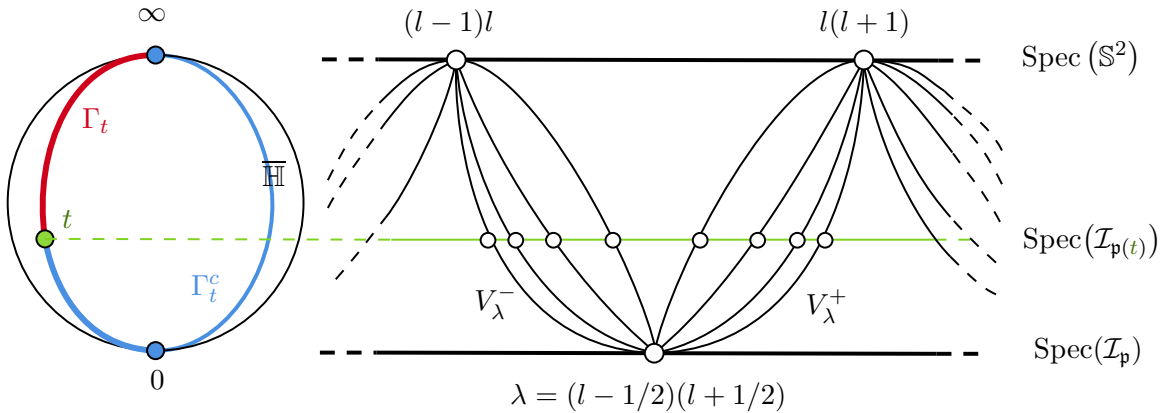

\subsection{Finite group actions on $\lb$}\label{grpaction}
We summarize some basic facts of lifting of finite group action from $\sph$ to $\Sigma_\vp$. 

\begin{lemma}
    Given a (anti-)holomorphic automorphism $g:\sph\to\sph$ that preseves the set $\vp$, there exists a (anti-)holomorphic automorphism of $\Sigma_\vp\setminus\vp$, such that the following commutes 
\[
\begin{tikzcd}
\Sigma_\vp\setminus\vp \arrow[r,"\tilde{g}"]\arrow[d,"\pi"]
&
\Sigma_\vp\setminus\vp \arrow[d,"\pi"]\\
\sph\setminus\vp \arrow[r,"g"]&\sph\setminus\vp
\end{tikzcd}
\]
\end{lemma}
\begin{proof}
    Note that both $\pi:\Sigma_\vp\setminus\vp\to\sph\setminus\vp$ and $g\circ\pi:\Sigma_\vp\setminus\vp\to\sph\setminus\vp$ are $2$-fold coverings. Moreover, $\pi_\ast(\pi_1(\Sigma_\vp\setminus\vp))=(g\circ\pi)_\ast(\pi_1(\Sigma_\vp\setminus\vp))$, hence an isomorphism $\tilde{g}$ between these two covering spaces exists. It remains to show that $\tilde{g}$ is a (anti-)holomorphic isomorphism. This can be done locally. Since for any point $x\in\Sigma_{\vp}\setminus\vp$, there is a neighborhood $U$ of $x$ such that $\pi|_U$ is biholomorphic, it follows that $\tilde{g}|_U=(\pi^{-1}\circ g\circ\pi)|_U$ is a (anti-)holomorphic isomorphism as long as $g|_{\pi(U)}$ is. 
\end{proof}

$\tilde{g}$ naturally extends to the entire $\Sigma_\vp$ and the resulting map is a (anti-)holomorphic automorphism of $\Sigma_\vp$. Now consider two different choices of the lift of $g$, namely $\tilde{g}_1$ and $\tilde{g}_2$. These satisfy $\tilde{g}_1\circ \tilde{g}_2^{-1}=\ei$, where $\ei$ is the canonical involution on $\Sigma_\vp$. In particular, the identity $\mathrm{id}_{\sph}$ can be lifted to $\mathrm{id}_{\Sigma_\vp}=:\mathrm{id}$ or $\ei$. It follows that:

\begin{lemma}
For any isomorphism $g:\sph\to \sph$ preserving $\vp$, there are exactly two different lifts on $\Sigma_\vp$, namely $\tilde{g}_1$ and $\tilde{g}_2$. Both of them commute with $\ei$.\label{commuteswithinv}
\end{lemma}
\begin{proof}
    Assume $\tilde{g}_j\circ\ei\ne \ei\circ\tilde{g}_j$ for one of $j=1,2$. These automorphisms are both lifts of $g$. We conclude that either $\tilde{g}_j\circ\ei=\tilde{g}_j$ or $\ei\circ\tilde{g}_j=\tilde{g}_j$, both implies $\ei=\mathrm{id}$, leading to a contradiction.
\end{proof}

Let $G$ be a subgroup of $\mathrm{O}(3)$ preserving $\vp$. By the previous two lemmas, there exists a group $\hat{G}$ consisting of holomorphic and anti-holomorphic automorphisms on $\Sigma_\vp$, that preserves $\vp$. Furthermore, we have:
\begin{lemma}
    There exists an exact sequence $0\to\ZT\to\hat{G}\to G\to 1$, where $\ZT$ is generated by $\ei$. The homomorphism $\hat{G}\to G$ is given by $g\mapsto \pi\circ g\circ \pi^{-1}$, which is well-defined by construction of $\hat{G}$.
\end{lemma}

The group $\hat{G}$, along with the sequence, is known as a central extension by $\ZT$ of $G$. Isomorphic classes of such extensions correspond one-to-one with cohomology classes in $H^2(G;\ZT)$. Unfortunately, in general, this cohomology group could be nontrivial, and hence it's possible that $\hat{G}\ne G\times\ZT$. When $\hat{G}$ is isomorphic with $G\times\ZT$, we say that the $G$-action on $\sph$ can be lifted to $\Sigma_\vp$. Note that if the exact sequence $0\to\ZT\to\hat{G}\to G\to 1$ is split, then $\hat{G}\cong G\times\ZT$.

Since $\ei$ lies in the center of $\hat{G}$, the group $\hat{G}$ acts on the space of odd functions $H^1_-(\Sigma_\vp)$ and hence $\sobc$. On the other hand, we have:
\begin{lemma}
    For an isometry $g\in\mathrm{O}(3)$, any lift $\tilde{g}$ is an isometry under the pullback metric $\pi^\ast\mathrm{ds}^2$ on $\Sigma_\vp$.
\end{lemma}
\begin{proof}
    Since $g\circ\pi=\pi\circ \tilde{g}$, $\pi^\ast\mathrm{ds}^2=(g\circ\pi)^\ast\mathrm{ds}^2=(\pi\circ \tilde{g})^\ast\mathrm{ds}^2=\tilde{g}^\ast(\pi^\ast\mathrm{ds}^2)$. It follows that $\tilde{g}$ is isometric.
\end{proof}

From this lemma, we conclude that $\hat{G}$ commutes with the singular Laplacian $\Delta$ on $\Sigma_\vp$. Therefore, each eigenspace $V_\lambda\subset \sobc$ of $\Delta$ serves as a finite dimensional complex representation of $\hat{G}$. If $G$ can be lifted to $\Sigma_\vp$, i.e. $\hat{G}= G\times\ZT$, then $V_\lam$ can be regarded as a representation of $G$. From classical representation theory, $V_\lam$ can be decomposed into direct sum of irreducible $G$-representation.

Given an eigensection $f\in V_\lambda$, assume that $f$ lies in one of the irreducible summands of the decomposition of $V_\lambda$. If irreducible representations of $G$ are classified completely, and suppose it is known that how $G$ acts on the summand in which $f$ lies, then the action of $G$ on $f$ would be clear. Consequently, we may discuss the relations between the leading coefficients of $f$ near points in $\vp$. Combining this with algebraic relations that were proved in the previous sections, leads to the main results of this article, which will be discussed in the following.

\subsection{The tetrahedral configuration}\label{tetrahedralcase}
In this subsection, we consider the tetrahedral configuration $\vp$, i.e. the configurations comprising vertices of the regular tetrahedron on $\sph$. We first discuss the finite group action on $\sph$ that preserves $\vp$, and lift it to the surface $\Sigma_\vp$. Next, using this group action, we provide another description of the Taubes-Wu tetrahedral eigensections, which were introduced in \cite{taubeswu2020examples}. As previously mentioned, we shall combine the algebraic identities proved in the preceding section, with a representation-theoretic viewpoint, to study these tetrahedral eigensections. We aim to establish relations between leading coefficients of eigensections with certain symmetric conditions. It follows from this that different types of irreducible representations are mutually exclusive. This leads to the deformation rigidity of Taubes-Wu tetrahedral eigensections.

The tetrahedron subgroup of $\SO$, denoted as $A_4$, acts on $\sph\setminus\vp$ when $\vp$ represents the vertices of a regular tetrahedron. To be precise, let $\vp=\{p_1=(\frac{2\sqrt{2}}{3},0,\frac{1}{3}),p_2=(-\frac{\sqrt{2}}{3},-\frac{\sqrt{2}}{\sqrt{3}},\frac{1}{3}),p_3=(-\frac{\sqrt{2}}{3},\frac{\sqrt{2}}{\sqrt{3}},\frac{1}{3}),p_4=(0,0,-1)\}$. Designate $\sigma_l\in A_4$ as $\frac{2\pi}{3}$ rotations in the clockwise direction about the oriented line from the origin to $p_l~(1\le l\le 4)$. 
We aim to show that this $A_4$-action can be lifted to $\Sigma_\vp$.

It follows from Lemma \ref{commuteswithinv} that, for $\sigma_l\in A_4$, one can choose a unique lift $\tilde{\sigma}_l$ satisfying $\tilde{\sigma_l}^3=\mathrm{id}$. We fix these choices from now on. Next, we provide a more detailed discussion about the lifts $\tilde{\sigma}_l$ $(1\le l\le 4)$.

In local coordinates $z_l$ near $p_l\in\vp$ on $\sph$, that was introduced in Section \ref{4.1}, $\sigma_l$ corresponds to multiplication with $e^{\frac{4\pi}{3}i}$. Consider a complex coordinate $w_l$ near $p_l\in\Sigma_\vp$, in which the projection $\pi:\Sigma_\vp\to\sph$ can be interpreted as $w_l\mapsto w_l^2=z_l$. The lift of $\sigma_l$ must be multiplication with either $e^{\frac{2\pi}{3}i}$ or $-e^{\frac{2\pi}{3}i}$ in the $w_l$-coordinate. Recall the condition $\tilde{\sigma}_l^3=\mathrm{id}$, the only possible choice is the former. 

The points in $\vp$ and geodesics connecting them establish a triangulation of $\sph$ (refer to Figure \ref{trionsph}). We will denote the geodesic triangular domains by $A$, $B$, $C$ and $D$. Each of these domains has a preimage in $\Sigma_\vp$ consisting of two isometric components, which we distinguish using the superscript $+$ and $-$.

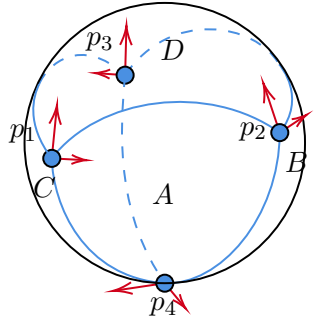
\begin{figure}[!h]
    \centering
\tikzset{every picture/.style={line width=0.75pt}} %set default line width to 0.75pt        

\begin{tikzpicture}[x=0.75pt,y=0.75pt,yscale=-1,xscale=1]

%Straight Lines [id:da2789117500016476] 
\draw [color={rgb, 255:red, 208; green, 2; blue, 27 }  ,draw opacity=1 ][line width=0.75]    (149.8,64.6) -- (150.41,41.53) ;
\draw [shift={(150.47,39.53)}, rotate = 91.52] [color={rgb, 255:red, 208; green, 2; blue, 27 }  ,draw opacity=1 ][line width=0.75]    (9.84,-2.96) .. controls (6.25,-1.25) and (2.97,-0.27) .. (0,0) .. controls (2.97,0.27) and (6.25,1.26) .. (9.84,2.96)   ;
%Straight Lines [id:da3677562638469227] 
\draw [color={rgb, 255:red, 208; green, 2; blue, 27 }  ,draw opacity=1 ][line width=0.75]    (149.8,64.6) -- (136.46,63.67) ;
\draw [shift={(134.47,63.53)}, rotate = 3.98] [color={rgb, 255:red, 208; green, 2; blue, 27 }  ,draw opacity=1 ][line width=0.75]    (7.65,-2.3) .. controls (4.86,-0.97) and (2.31,-0.21) .. (0,0) .. controls (2.31,0.21) and (4.86,0.98) .. (7.65,2.3)   ;
%Straight Lines [id:da5072773534562194] 
\draw [color={rgb, 255:red, 208; green, 2; blue, 27 }  ,draw opacity=1 ][line width=0.75]    (227.8,93.6) -- (220.1,70.43) ;
\draw [shift={(219.47,68.53)}, rotate = 71.61] [color={rgb, 255:red, 208; green, 2; blue, 27 }  ,draw opacity=1 ][line width=0.75]    (9.84,-2.96) .. controls (6.25,-1.25) and (2.97,-0.27) .. (0,0) .. controls (2.97,0.27) and (6.25,1.26) .. (9.84,2.96)   ;
%Straight Lines [id:da1914582759748995] 
\draw [color={rgb, 255:red, 208; green, 2; blue, 27 }  ,draw opacity=1 ][line width=0.75]    (112.8,106.6) -- (115.27,81.52) ;
\draw [shift={(115.47,79.53)}, rotate = 95.63] [color={rgb, 255:red, 208; green, 2; blue, 27 }  ,draw opacity=1 ][line width=0.75]    (9.84,-2.96) .. controls (6.25,-1.25) and (2.97,-0.27) .. (0,0) .. controls (2.97,0.27) and (6.25,1.26) .. (9.84,2.96)   ;
%Straight Lines [id:da9602099110929967] 
\draw [color={rgb, 255:red, 208; green, 2; blue, 27 }  ,draw opacity=1 ][line width=0.75]    (112.8,106.6) -- (128.47,107.43) ;
\draw [shift={(130.47,107.53)}, rotate = 183.02] [color={rgb, 255:red, 208; green, 2; blue, 27 }  ,draw opacity=1 ][line width=0.75]    (7.65,-2.3) .. controls (4.86,-0.97) and (2.31,-0.21) .. (0,0) .. controls (2.31,0.21) and (4.86,0.98) .. (7.65,2.3)   ;
%Straight Lines [id:da3492980213696839] 
\draw [color={rgb, 255:red, 208; green, 2; blue, 27 }  ,draw opacity=1 ][line width=0.75]    (169.8,169.6) -- (145.44,173.24) ;
\draw [shift={(143.47,173.53)}, rotate = 351.5] [color={rgb, 255:red, 208; green, 2; blue, 27 }  ,draw opacity=1 ][line width=0.75]    (9.84,-2.96) .. controls (6.25,-1.25) and (2.97,-0.27) .. (0,0) .. controls (2.97,0.27) and (6.25,1.26) .. (9.84,2.96)   ;
%Straight Lines [id:da2037309472969273] 
\draw [color={rgb, 255:red, 208; green, 2; blue, 27 }  ,draw opacity=1 ][line width=0.75]    (169.8,169.6) -- (179.19,180.99) ;
\draw [shift={(180.47,182.53)}, rotate = 230.49] [color={rgb, 255:red, 208; green, 2; blue, 27 }  ,draw opacity=1 ][line width=0.75]    (7.65,-2.3) .. controls (4.86,-0.97) and (2.31,-0.21) .. (0,0) .. controls (2.31,0.21) and (4.86,0.98) .. (7.65,2.3)   ;
%Curve Lines [id:da8840582791505471] 
\draw [color={rgb, 255:red, 74; green, 144; blue, 226 }  ,draw opacity=1 ] [dash pattern={on 4.5pt off 4.5pt}]  (106.47,69.13) .. controls (110.47,59.13) and (127.47,44.13) .. (149.8,64.6) ;
%Curve Lines [id:da5440624164659711] 
\draw [color={rgb, 255:red, 74; green, 144; blue, 226 }  ,draw opacity=1 ] [dash pattern={on 4.5pt off 4.5pt}]  (149.8,64.6) .. controls (172.47,31.13) and (215.47,38.13) .. (227.47,59.13) ;
%Curve Lines [id:da8941096527749581] 
\draw [color={rgb, 255:red, 74; green, 144; blue, 226 }  ,draw opacity=1 ]   (169.8,169.6) .. controls (197.47,170.73) and (227.47,136.73) .. (227.8,93.6) ;
%Curve Lines [id:da42766566582150123] 
\draw [color={rgb, 255:red, 74; green, 144; blue, 226 }  ,draw opacity=1 ]   (227.8,93.6) .. controls (241.47,76.13) and (228.47,60.13) .. (227.47,59.13) ;
%Curve Lines [id:da621681756798627] 
\draw [color={rgb, 255:red, 74; green, 144; blue, 226 }  ,draw opacity=1 ]   (112.8,106.6) .. controls (136.47,75.73) and (191.47,67.73) .. (227.8,93.6) ;
%Curve Lines [id:da9270462908193762] 
\draw [color={rgb, 255:red, 74; green, 144; blue, 226 }  ,draw opacity=1 ]   (112.8,106.6) .. controls (113.47,130) and (128.47,166.6) .. (169.8,169.6) ;
%Shape: Circle [id:dp12093361078083942] 
\draw  [fill={rgb, 255:red, 74; green, 144; blue, 226 }  ,fill opacity=0.49 ] (108.5,106.6) .. controls (108.5,104.23) and (110.43,102.3) .. (112.8,102.3) .. controls (115.17,102.3) and (117.1,104.23) .. (117.1,106.6) .. controls (117.1,108.97) and (115.17,110.9) .. (112.8,110.9) .. controls (110.43,110.9) and (108.5,108.97) .. (108.5,106.6) -- cycle ;
%Shape: Circle [id:dp5406841936599398] 
\draw  [color={rgb, 255:red, 0; green, 0; blue, 0 }  ,draw opacity=0.64 ][fill={rgb, 255:red, 74; green, 144; blue, 226 }  ,fill opacity=0.3 ] (145.5,64.6) .. controls (145.5,62.23) and (147.43,60.3) .. (149.8,60.3) .. controls (152.17,60.3) and (154.1,62.23) .. (154.1,64.6) .. controls (154.1,66.97) and (152.17,68.9) .. (149.8,68.9) .. controls (147.43,68.9) and (145.5,66.97) .. (145.5,64.6) -- cycle ;
%Shape: Circle [id:dp3950150252793465] 
\draw  [fill={rgb, 255:red, 74; green, 144; blue, 226 }  ,fill opacity=0.49 ] (165.5,169.6) .. controls (165.5,167.23) and (167.43,165.3) .. (169.8,165.3) .. controls (172.17,165.3) and (174.1,167.23) .. (174.1,169.6) .. controls (174.1,171.97) and (172.17,173.9) .. (169.8,173.9) .. controls (167.43,173.9) and (165.5,171.97) .. (165.5,169.6) -- cycle ;
%Curve Lines [id:da29384188325059557] 
\draw [color={rgb, 255:red, 74; green, 144; blue, 226 }  ,draw opacity=1 ] [dash pattern={on 4.5pt off 4.5pt}]  (149.8,64.6) .. controls (146.47,98) and (146.47,136) .. (169.8,169.6) ;
%Shape: Circle [id:dp24847923698631025] 
\draw   (99.07,98.87) .. controls (99.07,59.8) and (130.74,28.13) .. (169.8,28.13) .. controls (208.86,28.13) and (240.53,59.8) .. (240.53,98.87) .. controls (240.53,137.93) and (208.86,169.6) .. (169.8,169.6) .. controls (130.74,169.6) and (99.07,137.93) .. (99.07,98.87) -- cycle ;
%Curve Lines [id:da722513401216627] 
\draw [color={rgb, 255:red, 74; green, 144; blue, 226 }  ,draw opacity=1 ]   (112.8,106.6) .. controls (99.47,89.13) and (103.47,75.13) .. (106.47,69.13) ;
%Straight Lines [id:da29413965440147893] 
\draw [color={rgb, 255:red, 208; green, 2; blue, 27 }  ,draw opacity=1 ][line width=0.75]    (227.8,93.6) -- (239.8,85.64) ;
\draw [shift={(241.47,84.53)}, rotate = 146.44] [color={rgb, 255:red, 208; green, 2; blue, 27 }  ,draw opacity=1 ][line width=0.75]    (7.65,-2.3) .. controls (4.86,-0.97) and (2.31,-0.21) .. (0,0) .. controls (2.31,0.21) and (4.86,0.98) .. (7.65,2.3)   ;
%Shape: Circle [id:dp46533331798079947] 
\draw  [fill={rgb, 255:red, 74; green, 144; blue, 226 }  ,fill opacity=0.49 ] (223.5,93.6) .. controls (223.5,91.23) and (225.43,89.3) .. (227.8,89.3) .. controls (230.17,89.3) and (232.1,91.23) .. (232.1,93.6) .. controls (232.1,95.97) and (230.17,97.9) .. (227.8,97.9) .. controls (225.43,97.9) and (223.5,95.97) .. (223.5,93.6) -- cycle ;

% Text Node
\draw (206.1,87) node [anchor=north west][inner sep=0.75pt]    {$p_{2}$};
% Text Node
\draw (128.8,41.7) node [anchor=north west][inner sep=0.75pt]    {$p_{3}$};
% Text Node
\draw (161,175.4) node [anchor=north west][inner sep=0.75pt]    {$p_{4}$};
% Text Node
\draw (90.1,88) node [anchor=north west][inner sep=0.75pt]    {$p_{1}$};
% Text Node
\draw (162,118.4) node [anchor=north west][inner sep=0.75pt]    {$A$};
% Text Node
\draw (228.72,102.08) node [anchor=north west][inner sep=0.75pt]  [rotate=-358.72]  {$B$};
% Text Node
\draw (102.04,114.36) node [anchor=north west][inner sep=0.75pt]  [rotate=-0.35]  {$C$};
% Text Node
\draw (165.79,45.62) node [anchor=north west][inner sep=0.75pt]  [rotate=-358.25]  {$D$};

\end{tikzpicture}

    \caption{A triangulation on $\sph$.}
    \label{trionsph}
\end{figure}

Since $\Sigma_\vp$ is an elliptic curve, it can be equipped with a flat metric compatible with the conformal structure, and consequently there is an isometric covering $\CC\to\Sigma_\vp$, where $\CC$ is the flat plane. This triangulation on $\sph$ can be lifted to $\Sigma_\vp$ and hence $\CC$. Note that a holomorphic automorphism of $\Sigma_\vp$ can also be lifted to $\CC$ as an isometry in flat metric. For instance, assume the origin lies in the preimage of a point in $\vp$, the elliptic involution $\ei$ can be lifted to the map $-\mathbbm{1}:\CC\to\CC:w\mapsto-w$. Furthermore, the $\tilde{\sigma}_l$ can be lifted to a $2\pi/3$-rotation.

In the case under consideration, determining whether the $A_4$-action on $\sph$ can be lifted to $\Sigma_\vp$ is simplified by verifying that the relations between generators $\tilde{\sigma}_l~(1\le l\le 4)$ persist in the automorphism group of $\Sigma_\vp$. Given that the covering $\CC\to\Sigma_\vp$ is isometric (in flat metrics), this verification can be efficiently conducted over the plane.

Taking it a step further, we aim to prove that the symmetry group $S_4$ can be lifted to $\Sigma_\vp$. Note that $S_4$ can be generated by $A_4$ and $\sigma_0$, where $\sigma_0$ represents the reflection of $\sph$ along great circle passing through $p_1$ and $p_4$. 

The $\sigma_0$ can be represented as the transposition $(2\,3)$ in $S_4$, and serves as an anti-holomorphic as well as an isometric involution. Locally, in the $z_4$-coordinate, it can be interpreted as $z_4\mapsto \bar{z}_4$. Thus when lifted to $\Sigma$, it maps $w_4$ to either $\bar{w}_4$ or $-\bar{w}_4$, implying the involution condition $\tilde{\sigma}_0^2=\mathrm{id}$. We choose the latter, $w_4\mapsto -\bar{w}_4$, as our $\tilde{\sigma}_0$.

The $\tilde{\sigma}_0$ is an anti-conformal automorphism of $\Sigma_\vp$ and, as a result, an orientation-reversing isometric involution under the chosen flat metric on $\Sigma_\vp$. This holds for the other choice of the lift of $\sigma_0$ as well. The map $w_4\mapsto \bar{w}_4$ fixes the set $\{\Im\,w_4=0\}$, which is the preimage of a geodesics in the triangulation of $\sph$, connecting $p_1$ and $p_4$. Consequently, this fixed point set itself forms a geodesic on the flat torus. Therefore, the geodesic triangulation of $\sph$ is lifted to a family of straight lines in $\CC$, with their intersection points corresponding the preimages of $p_l~(1\le l\le 4)$.  

Let's explain the Figure \ref{THEFIGURE}. For a sake of simplicity in notations, we will continue to denote the preimages of $p_l$ in $\CC$ as $p_l$. Alternatively, one can view this figure as the triangulation on the flat torus. The projection $\CC\to\Sigma_\vp$ is realized through a quotient by the lattice with vertices $p_4$. It's worth noting that, with such a lattice, the flat torus here is the equilateral one, realizing the maximal conformal area among all conformal structure on $T^2$ (see \cite{Nadirashvili}). In the figure, the longer arrow, colored in red, represents the real direction of a chosen coordinate $w_l$ near each $p_l$, while the shorter one represents the imaginary direction. The choices of $w_l$-coordinates is arbitrary in sign, but we will fix them using these red arrows henceforth.

\begin{figure}[!h]
    \centering

\tikzset{every picture/.style={line width=0.75pt}} %set default line width to 0.75pt        



    \caption{The triangulation on $\CC$.}
    \label{THEFIGURE}
\end{figure}

The $\sigma_0=(2\,3)\in S_4$ is lifted to a reflection of the plane along the specified imaginary direction of one of the preimages of $p_4$. Therefore, it can be interpreted either as $w_4\mapsto -\bar{w}_4$, or as $w_1\mapsto \bar{w}_1$ near $p_1$. There are some subtleties here. As complex coordinates of points, we have the action $\tilde{\sigma}_l. w_l=w_l\circ \tilde{\sigma}_l=e^{\frac{2\pi}{3}i}w_l$. But $w_l$ can also be regarded as complex-valued functions defined near $p_l$. Thus, rather than the former definition, we may define the action on these functions by pushforward, namely $\tilde{\sigma}_k. w_l:=w_l\circ \tilde{\sigma}_k^{-1}$, ensuring it to be a left action. The blue arced arrow represents the rotation $\tilde{\sigma}_4$, which is a $2\pi/3$-rotation centered in $p_4\in \CC$. However, under this definition of actions, we have $\tilde{\sigma}_4.w_4=e^{-\frac{2\pi}{3}i}w_4$. Furthermore, considering the $w_l$-coordinates $(1\le l\le 4)$ presented in the Figure \ref{THEFIGURE}, we can establish the following lemma:
\begin{lemma}\label{rotacoorrel}
\begin{itemize}
    \item $\tilde{\sigma}_l.w_l=e^{-\frac{2\pi}{3}i}w_l$ for $l=1,2,3,4$.
    \item $\tilde{\sigma}_0.w_4=-\bar{w}_4$, $\tilde{\sigma}_0.w_1=\bar{w}_1$ and $\tilde{\sigma}_0.w_2=\bar{w}_3$.
    \item $\tilde{\sigma}_4.w_1=w_3$, $\tilde{\sigma}_4.w_3=w_2$ and $\tilde{\sigma}_4.w_2=w_1$. 
    \item $\tilde{\sigma}_3.w_1=e^{\frac{2\pi}{3}i}w_2$, $\tilde{\sigma}_3.w_2=e^{\frac{\pi}{2}i}w_4$ and $\tilde{\sigma}_3.w_4=e^{\frac{2\pi}{3}i}w_1$.
    \item $\tilde{\sigma}_2.w_1=e^{-\frac{\pi}{6}i}w_4$, $\tilde{\sigma}_2.w_4=e^{-\frac{\pi}{2}i}w_3$ and $\tilde{\sigma}_2.w_3=e^{\frac{2\pi}{3}i}w_1$.
    \item $\tilde{\sigma}_1.w_4=e^{\frac{\pi}{6}i}w_2$, $\tilde{\sigma}_1.w_3=-e^{\frac{\pi}{6}i}w_4$ and $\tilde{\sigma}_1.w_2=e^{\frac{2\pi}{3}i}w_3$.
\end{itemize}
\end{lemma}
\begin{proof}
    We shall verify part of the fifth bullet; the remaining steps are similar and straightforward. In the Figure \ref{sigma2proof}, we depict $w_l$-coordinates at $p_l$ using red arrows, where the longer arrow represents the real direction and the shorter arrow represents the imaginary direction. Recall the definition $\tilde{\sigma}_2.w_1=w_1\circ \tilde{\sigma}_2^{-1}$. Let's denote $w_4=u_4+iv_4$. Then the longer blue arrow in the figure represents the direction $(\tilde{\sigma}_2^{-1})_\ast\partial/\partial u_4$ and the shorter one represents the direction $(\tilde{\sigma}_2^{-1})_\ast\partial/\partial v_4$. As a result, for a point $q$ with $w_4(q)=1$, under the map $\tilde{\sigma}_2^{-1}$, there must be $w_1\circ \tilde{\sigma}_2^{-1}(q)=e^{-\frac{\pi}{6}i}$. Consequently, we can conclude that $\tilde{\sigma}_2.w_1=e^{-\frac{\pi}{6}i}w_4$.

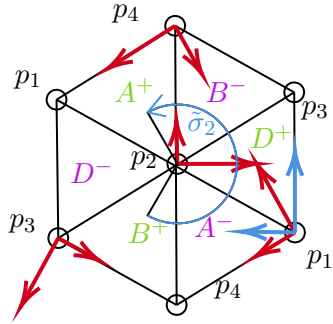
\begin{figure}[!hbp]
        \centering
        
\tikzset{every picture/.style={line width=0.75pt}} %set default line width to 0.75pt        

\begin{tikzpicture}[x=0.75pt,y=0.75pt,yscale=-1,xscale=1]

%Straight Lines [id:da8175519517129] 
\draw    (183.04,198.13) -- (182.3,57.22) ;
%Straight Lines [id:da8561655489682454] 
\draw    (183.04,198.13) -- (122.92,164.61) ;
%Straight Lines [id:da8228369768545125] 
\draw    (242.71,161.53) -- (183.04,198.13) ;
%Straight Lines [id:da842081276548124] 
\draw    (242.33,90.8) -- (123.34,164.35) ;
%Straight Lines [id:da8211821152846768] 
\draw    (123.34,164.35) -- (122.46,94.51) ;
%Straight Lines [id:da9629726350231336] 
\draw    (242.71,161.53) -- (242.33,90.8) ;
%Straight Lines [id:da38645193978319803] 
\draw    (242.71,161.53) -- (122.46,94.51) ;
%Straight Lines [id:da2901859342876496] 
\draw    (182.3,57.22) -- (122.46,94.51) ;
%Straight Lines [id:da43771819428635617] 
\draw    (242.33,90.8) -- (182.3,57.22) ;
%Shape: Circle [id:dp7298262484022455] 
\draw   (187.3,195.51) .. controls (188.75,197.87) and (188.01,200.94) .. (185.65,202.39) .. controls (183.3,203.83) and (180.22,203.09) .. (178.78,200.74) .. controls (177.33,198.39) and (178.07,195.31) .. (180.43,193.86) .. controls (182.78,192.42) and (185.86,193.16) .. (187.3,195.51) -- cycle ;
%Shape: Circle [id:dp7275215416584413] 
\draw   (246.97,158.92) .. controls (248.42,161.27) and (247.68,164.35) .. (245.33,165.79) .. controls (242.97,167.24) and (239.89,166.5) .. (238.45,164.14) .. controls (237.01,161.79) and (237.74,158.71) .. (240.1,157.27) .. controls (242.45,155.82) and (245.53,156.56) .. (246.97,158.92) -- cycle ;
%Shape: Circle [id:dp8320289645724257] 
\draw   (187.67,124.32) .. controls (189.11,126.67) and (188.37,129.75) .. (186.02,131.19) .. controls (183.67,132.64) and (180.59,131.9) .. (179.14,129.54) .. controls (177.7,127.19) and (178.44,124.11) .. (180.79,122.67) .. controls (183.15,121.22) and (186.23,121.96) .. (187.67,124.32) -- cycle ;
%Shape: Circle [id:dp5438715098318683] 
\draw   (246.59,88.18) .. controls (248.03,90.54) and (247.29,93.61) .. (244.94,95.06) .. controls (242.59,96.5) and (239.51,95.76) .. (238.06,93.41) .. controls (236.62,91.06) and (237.36,87.98) .. (239.71,86.53) .. controls (242.07,85.09) and (245.14,85.83) .. (246.59,88.18) -- cycle ;
%Shape: Circle [id:dp31902552323202626] 
\draw   (127.6,161.74) .. controls (129.05,164.09) and (128.31,167.17) .. (125.96,168.61) .. controls (123.6,170.06) and (120.52,169.32) .. (119.08,166.97) .. controls (117.64,164.61) and (118.37,161.53) .. (120.73,160.09) .. controls (123.08,158.65) and (126.16,159.38) .. (127.6,161.74) -- cycle ;
%Shape: Circle [id:dp2290507406942004] 
\draw   (186.56,54.61) .. controls (188.01,56.96) and (187.27,60.04) .. (184.92,61.48) .. controls (182.56,62.93) and (179.48,62.19) .. (178.04,59.84) .. controls (176.6,57.48) and (177.33,54.4) .. (179.69,52.96) .. controls (182.04,51.52) and (185.12,52.25) .. (186.56,54.61) -- cycle ;
%Shape: Circle [id:dp9243727362104248] 
\draw   (126.73,91.89) .. controls (128.17,94.25) and (127.43,97.32) .. (125.08,98.77) .. controls (122.72,100.21) and (119.65,99.47) .. (118.2,97.12) .. controls (116.76,94.77) and (117.5,91.69) .. (119.85,90.24) .. controls (122.2,88.8) and (125.28,89.54) .. (126.73,91.89) -- cycle ;
%Straight Lines [id:da7282867936624708] 
\draw [color={rgb, 255:red, 208; green, 2; blue, 27 }  ,draw opacity=1 ][line width=1.5]    (182.2,57.28) -- (194.4,79.08) ;
\draw [shift={(195.87,81.7)}, rotate = 240.77] [color={rgb, 255:red, 208; green, 2; blue, 27 }  ,draw opacity=1 ][line width=1.5]    (12.79,-3.85) .. controls (8.13,-1.64) and (3.87,-0.35) .. (0,0) .. controls (3.87,0.35) and (8.13,1.64) .. (12.79,3.85)   ;
%Straight Lines [id:da1454644743056579] 
\draw [color={rgb, 255:red, 208; green, 2; blue, 27 }  ,draw opacity=1 ][line width=1.5]    (182.3,57.22) -- (150.65,77.28) ;
\draw [shift={(148.12,78.89)}, rotate = 327.63] [color={rgb, 255:red, 208; green, 2; blue, 27 }  ,draw opacity=1 ][line width=1.5]    (12.79,-3.85) .. controls (8.13,-1.64) and (3.87,-0.35) .. (0,0) .. controls (3.87,0.35) and (8.13,1.64) .. (12.79,3.85)   ;
%Straight Lines [id:da48653982754643943] 
\draw [color={rgb, 255:red, 208; green, 2; blue, 27 }  ,draw opacity=1 ][line width=1.5]    (183.41,126.93) -- (220.44,126.85) ;
\draw [shift={(223.44,126.85)}, rotate = 179.88] [color={rgb, 255:red, 208; green, 2; blue, 27 }  ,draw opacity=1 ][line width=1.5]    (12.79,-3.85) .. controls (8.13,-1.64) and (3.87,-0.35) .. (0,0) .. controls (3.87,0.35) and (8.13,1.64) .. (12.79,3.85)   ;
%Straight Lines [id:da6413547981629228] 
\draw [color={rgb, 255:red, 208; green, 2; blue, 27 }  ,draw opacity=1 ][line width=1.5]    (183.41,126.93) -- (182.97,108.07) ;
\draw [shift={(182.9,105.07)}, rotate = 88.68] [color={rgb, 255:red, 208; green, 2; blue, 27 }  ,draw opacity=1 ][line width=1.5]    (12.79,-3.85) .. controls (8.13,-1.64) and (3.87,-0.35) .. (0,0) .. controls (3.87,0.35) and (8.13,1.64) .. (12.79,3.85)   ;
%Straight Lines [id:da4392697405157482] 
\draw [color={rgb, 255:red, 208; green, 2; blue, 27 }  ,draw opacity=1 ][line width=1.5]    (123.34,164.35) -- (105.85,196.33) ;
\draw [shift={(104.41,198.96)}, rotate = 298.69] [color={rgb, 255:red, 208; green, 2; blue, 27 }  ,draw opacity=1 ][line width=1.5]    (12.79,-3.85) .. controls (8.13,-1.64) and (3.87,-0.35) .. (0,0) .. controls (3.87,0.35) and (8.13,1.64) .. (12.79,3.85)   ;
%Straight Lines [id:da7402349011322988] 
\draw [color={rgb, 255:red, 208; green, 2; blue, 27 }  ,draw opacity=1 ][line width=1.5]    (123.34,164.35) -- (141.51,174.32) ;
\draw [shift={(144.14,175.76)}, rotate = 208.75] [color={rgb, 255:red, 208; green, 2; blue, 27 }  ,draw opacity=1 ][line width=1.5]    (12.79,-3.85) .. controls (8.13,-1.64) and (3.87,-0.35) .. (0,0) .. controls (3.87,0.35) and (8.13,1.64) .. (12.79,3.85)   ;
%Straight Lines [id:da05772103305779841] 
\draw [color={rgb, 255:red, 208; green, 2; blue, 27 }  ,draw opacity=1 ][line width=1.5]    (242.71,161.53) -- (224.9,129.47) ;
\draw [shift={(223.44,126.85)}, rotate = 60.94] [color={rgb, 255:red, 208; green, 2; blue, 27 }  ,draw opacity=1 ][line width=1.5]    (12.79,-3.85) .. controls (8.13,-1.64) and (3.87,-0.35) .. (0,0) .. controls (3.87,0.35) and (8.13,1.64) .. (12.79,3.85)   ;
%Straight Lines [id:da5110771272649235] 
\draw [color={rgb, 255:red, 208; green, 2; blue, 27 }  ,draw opacity=1 ][line width=1.5]    (242.71,161.53) -- (225.03,172.31) ;
\draw [shift={(222.47,173.87)}, rotate = 328.64] [color={rgb, 255:red, 208; green, 2; blue, 27 }  ,draw opacity=1 ][line width=1.5]    (12.79,-3.85) .. controls (8.13,-1.64) and (3.87,-0.35) .. (0,0) .. controls (3.87,0.35) and (8.13,1.64) .. (12.79,3.85)   ;
%Shape: Arc [id:dp7565662466875873] 
\draw  [draw opacity=0] (168.4,100.94) .. controls (172.82,98.39) and (177.94,96.93) .. (183.41,96.93) .. controls (199.98,96.93) and (213.41,110.36) .. (213.41,126.93) .. controls (213.41,143.5) and (199.98,156.93) .. (183.41,156.93) .. controls (177.94,156.93) and (172.82,155.47) .. (168.4,152.91) -- (183.41,126.93) -- cycle ; \draw  [color={rgb, 255:red, 74; green, 144; blue, 226 }  ,draw opacity=1 ] (168.4,100.94) .. controls (172.82,98.39) and (177.94,96.93) .. (183.41,96.93) .. controls (199.98,96.93) and (213.41,110.36) .. (213.41,126.93) .. controls (213.41,143.5) and (199.98,156.93) .. (183.41,156.93) .. controls (177.94,156.93) and (172.82,155.47) .. (168.4,152.91) ;  
%Straight Lines [id:da6632021773022048] 
\draw [color={rgb, 255:red, 74; green, 144; blue, 226 }  ,draw opacity=1 ]   (168.4,100.94) -- (175.87,91.9) ;
%Straight Lines [id:da9449996901480766] 
\draw [color={rgb, 255:red, 74; green, 144; blue, 226 }  ,draw opacity=1 ]   (168.4,100.94) -- (177.87,102.9) ;
%Straight Lines [id:da31850136687062314] 
\draw [color={rgb, 255:red, 74; green, 144; blue, 226 }  ,draw opacity=1 ][line width=1.5]    (242.71,161.53) -- (242.67,123.1) ;
\draw [shift={(242.67,120.1)}, rotate = 89.94] [color={rgb, 255:red, 74; green, 144; blue, 226 }  ,draw opacity=1 ][line width=1.5]    (12.79,-3.85) .. controls (8.13,-1.64) and (3.87,-0.35) .. (0,0) .. controls (3.87,0.35) and (8.13,1.64) .. (12.79,3.85)   ;
%Straight Lines [id:da03822588285830575] 
\draw [color={rgb, 255:red, 74; green, 144; blue, 226 }  ,draw opacity=1 ][line width=1.5]    (242.71,161.53) -- (218.67,161.15) ;
\draw [shift={(215.67,161.1)}, rotate = 0.91] [color={rgb, 255:red, 74; green, 144; blue, 226 }  ,draw opacity=1 ][line width=1.5]    (12.79,-3.85) .. controls (8.13,-1.64) and (3.87,-0.35) .. (0,0) .. controls (3.87,0.35) and (8.13,1.64) .. (12.79,3.85)   ;

% Text Node
\draw (201.04,184.53) node [anchor=north west][inner sep=0.75pt]    {$p_{4}$};
% Text Node
\draw (247.33,169.19) node [anchor=north west][inner sep=0.75pt]    {$p_{1}$};
% Text Node
\draw (158.43,119) node [anchor=north west][inner sep=0.75pt]    {$p_{2}$};
% Text Node
\draw (100.02,78.1) node [anchor=north west][inner sep=0.75pt]  [rotate=-0.83]  {$p_{1}$};
% Text Node
\draw (149.37,44.75) node [anchor=north west][inner sep=0.75pt]  [rotate=-0.02]  {$p_{4}$};
% Text Node
\draw (244.69,93.87) node [anchor=north west][inner sep=0.75pt]  [rotate=-359.72]  {$p_{3}$};
% Text Node
\draw (97.28,151.15) node [anchor=north west ][inner sep=0.75pt]    {$p_{3}$};
% Text Node
\draw (190.43,150.74) node [anchor=north west][inner sep=0.75pt]  [color={rgb, 255:red, 189; green, 16; blue, 224 }  ,opacity=1 ]  {$A^{-}$};
% Text Node
\draw (156.91,153.56) node [anchor=north west][inner sep=0.75pt]  [color={rgb, 255:red, 126; green, 211; blue, 33 }  ,opacity=1 ]  {$B^{+}$};
% Text Node
\draw (218.41,105.9) node [anchor=north west][inner sep=0.75pt]  [color={rgb, 255:red, 126; green, 211; blue, 33 }  ,opacity=1 ,rotate=-359]  {$D^{+}$};
% Text Node
\draw (127.77,123.3) node [anchor=north west][inner sep=0.75pt]  [color={rgb, 255:red, 189; green, 16; blue, 224 }  ,opacity=1 ]  {$D^{-}$};
% Text Node
\draw (196.2,82.28) node [anchor=north west][inner sep=0.75pt]  [color={rgb, 255:red, 189; green, 16; blue, 224 }  ,opacity=1 ]  {$B^{-}$};
% Text Node
\draw (150.12,82.29) node [anchor=north west][inner sep=0.75pt]  [color={rgb, 255:red, 126; green, 211; blue, 33 }  ,opacity=1 ]  {$A^{+}$};
% Text Node
\draw (187,99) node [anchor=north west][inner sep=0.75pt]  [color={rgb, 255:red, 74; green, 144; blue, 226 }  ,opacity=1 ]  {$\tilde{\sigma }_{2}$};
\end{tikzpicture}
        \caption{The $\tilde{\sigma}_2$-action.}
        \label{sigma2proof}
    \end{figure}
\end{proof}

From classical group theory, we know that the symmetry group $S_4$ has the group presentation $\langle a, b| a^2=b^3=(ab)^4=\mathrm{id}\rangle$. On the other hand, with the branched covering $\Sigma_\vp\to \sph$ being given, recall that we have the central extension of $S_4$, namely $0\to\ZT\to \hat{S}_4\to S_4\to 1$. Here the group $\hat{S}_4$ is a subgroup of the group of automorphisms of $\Sigma_\vp$, say $\mathrm{Aut}(\Sigma_\vp)$. If we can find a section of this sequence, i.e. a homomorphism $S_4\to \hat{S}_4$ that is a right inverse of the epimorphism $\hat{S}_4\to S_4$, then we can conclude $\hat{S}_4\cong S_4\times \ZT$, and consequently the $S_4$-action (and hence the $A_4$-action) can be lifted to $\Sigma_\vp$. In the previous discussions, we associated each generators of $S_4$, denoted as $\sigma_l$, an element $\tilde{\sigma}_l\in\mathrm{Aut}(\Sigma_\vp)$. It's sufficient to prove the following lemma:

\begin{lemma}
    $\tilde{\sigma}_0^2=(\tilde{\sigma}_2^2)^3=(\tilde{\sigma}_0 \tilde{\sigma}_2^2)^4=\mathrm{id}$.
\end{lemma}
\begin{proof}
    The first two relations, $\tilde{\sigma}_0^2=(\tilde{\sigma}_2^2)^3=\mathrm{id}$, are satisfied by definition. In $S_4$, $\sigma_0=(2\,3)$ and $\sigma_2=(1\,4\,3)$. Thus we have $\sigma_0 \sigma_2^2=(1\,2\,3\,4)$ and hence $(\sigma_0 \sigma_2^2)^4=\mathrm{id}_{\sph}$. As a result, if $(\tilde{\sigma}_0 \tilde{\sigma}_2^2)^4\ne \mathrm{id}$, then $(\tilde{\sigma}_0 \tilde{\sigma}_2^2)^4=\ei$. When the latter case holds, it can be easily verified, from the Figure \ref{THEFIGURE}, that the map $(\tilde{\sigma}_0 \tilde{\sigma}_2^2)^4$ would maps $C^+$ to $C^-$.

    Note that $\sigma_0 \sigma_2^2\sigma_0=(1\,2\,4)=\sigma_3$ and $\sigma_2 \sigma_3\sigma_2^{-1}=(2\,3\,4)=\sigma_1$. But $(\tilde{\sigma}_0 \tilde{\sigma}_2^2\tilde{\sigma}_0)^3=\tilde{\sigma}_0\tilde{\sigma}_2^6\tilde{\sigma}_0=\mathrm{id}$, and$(\tilde{\sigma}_2 \tilde{\sigma}_3 \tilde{\sigma}_2^{-1})^3=\mathrm{id}$. Thus by definition of $\tilde{\sigma}_l$, we must have $\tilde{\sigma}_0  \tilde{\sigma}_2^2 \tilde{\sigma}_0=\tilde{\sigma}_3$ and $\tilde{\sigma}_2  \tilde{\sigma}_3 \tilde{\sigma}_2^{-1}=\tilde{\sigma}_1$. Consequently, $(\tilde{\sigma}_0  \tilde{\sigma}_2^2)^4=(\tilde{\sigma}_0 \tilde{\sigma}_2^2 \tilde{\sigma}_0) \tilde{\sigma}_2^2 (\tilde{\sigma}_0 \tilde{\sigma}_2^2 \tilde{\sigma}_0) \tilde{\sigma}_2^{-1}=\tilde{\sigma}_3 \tilde{\sigma}_2 \tilde{\sigma}_1$. None of maps in the composition $\tilde{\sigma}_3 \tilde{\sigma}_2 \tilde{\sigma}_1$ would map a domain with positive superscript in the Figure \ref{THEFIGURE}, to one with negative superscript. Consequently, $(\tilde{\sigma}_0  \tilde{\sigma}_2^2)^4=\mathrm{id}$; otherwise $(\tilde{\sigma}_0  \tilde{\sigma}_2^2)^4$ would maps $C^+$ to $C^-$.
\end{proof}

Now we can lift the $S_4$-action to $\Sigma_\vp$, since by the previous lemma, the map $S_4\to \hat{S}_4: \sigma_l\mapsto\tilde{\sigma}_l$ does define a homomorphism. For simplicity of notations, we use $\sigma_l$ instead of $\tilde{\sigma}_l$ to denote the generators of $S_4< \mathrm{Aut}(\Sigma_\vp)$ henceforth. As dicussed in Section \ref{grpaction}, for a $\ZT$ eigenvalue of $\lbc$, the eigenspace $V_\lam$ can be interpreted as a complex $S_4$-representation. If we identify $A_4$ as a subgroup of $S_4$, then $V_\lam$ is a $A_4$-representation either.

Consider an eigensection $f\in V_\lambda$. Near each $p_l$ we assume $f(z_l,\bar{z}_l)=a_lz_l^{n_l+1/2}+b_l\bar{z}_l^{n_l+1/2}+\mathcal{O}(r^{n_l+3/2})$, where $n_l=\vn_{p_l}(f)$ and $r=|z_l|$. The odd function $\tilde{f}$ lifted from $f$ then has an expansion $\tilde{f}(w_l,\bar{w}_l)=a_{l}w_l^{2n_l+1}+b_l\bar{w}_l^{2n_l+1}+\mathcal{O}(r^{2n_l+3})$. Given an element $\sigma\in S_4$, which maps $p_l$ to $p_k$, we then have
\begin{align}\label{sigmafw}
(\sigma. \tilde{f})(w_k,\bar{w}_k)=a_l(\sigma.w_l)^{2n_l+1}+b_l(\sigma.\bar{w}_l)^{2n_l+1}+\mathcal{O}(r^{2n_l+3}),
\end{align}
where $\sigma.w_l=w_l\circ \sigma^{-1}$ and $\sigma.\bar{w}_l=\bar{w}_l\circ\sigma^{-1}$. There is something subtle here. In the left hand side of the equation, we use $w_k$ as a coordinate, but in the right hand side we use $w_l$ as a complex-valued function. And the $\sigma.w_l$ and $\sigma.\bar{w}_l$ are complex-valued functions defined near $p_k$. This could be understood as a decomposition of $\tilde{f}$ near $p_k$, into series of functions $\sigma.w_l$ and $\sigma.\bar{w}_l$. It follows from \eqref{sigmafw} that $\vn_{p_k}(\sigma.f)=\vn_{p_l}(f)$. Furthermore, combining \eqref{sigmafw} with Lemma \ref{rotacoorrel} we deduce some convenient formula, such as $\sigma_3.\tilde{f}(w_2,\bar{w}_2)= a_1(e^{\frac{2\pi}{3}i}w_2)^{2n_1+1}+b_1(e^{\frac{-2\pi}{3}i}\bar{w}_2)^{2n_1+1}+\mathcal{O}(r^{2n_1+3})$, where $(a_1,b_1)$ is the leading coefficients of $f$ at $p_1$, and the order of the leading terms at this point is $n_1+\frac12$, for $n_1=\vn_{p_1}(f)$.

\begin{lemma}\label{irreS4rep}
It is well-known that irreducible $S_4$-representations can be listed as follows:
\begin{enumerate}[label=(\roman*)]
    \item the trivial representation $U=(\CC,\rho)$, where $\rho:\sigma_l\mapsto 1$ for $0\le l\le 4$;
    \item the alternating representation $U'=(\CC,\rho')$, where $\rho'(\sigma_0)=-1$ and $\rho'(\sigma_l)=1$ for $1\le l\le 4$;
    \item the standard representation $V=(\CC^4/\{e_1+e_2+e_3+e_4\},\rho_1)$, where $\{e_i\}_{1\le i\le 4}$ denote the standard basis of $\CC^4$ and $S_4$ acts on $\CC^4$ through permutations on the subscripts of the base vectors;
    \item $V':=V\otimes U'$;
    \item A two dimensional representation $W$, with a basis $\{v,\sigma_0.v\}$, such that $\sigma_l.v=e^{\frac{2\pi}{3}i}v$ and $\sigma_l.(\sigma_0.v)=\sigma_0.(\sigma_{l-1}^2.v)=-e^{\frac{\pi}{3}i}\sigma_0.v$ for $1\le l\le 4$.
\end{enumerate}
\end{lemma}

It is worth noting that $V_\lambda^c$, the subspace of $V_\lambda$ that consisting of critical eigensections, forms an $S_4$-subrepresentation. From the Proposition \ref{rk1} and \ref{irreS4rep} we conclude:
\begin{lemma}
    If $\dim_{\mbC}V_\lambda^c>0$, then it is either the trivial or alternating representation of $S_4$.
\end{lemma}

 Next, we shall show the existence of infinitely many eigenvalues $\lambda$ with $V_\lambda^c$ nontrivial, i.e. existence of infinitely many critical eigenvalues.

A fundamental domain of $A_4$ on $\sph$ is given by the spherical triangle with verticies $p_1,p_4$ and $-p_3$. With one more symmetry $\sigma_0$, the group $S_4$ possesses a smaller fundamental domain. This domain is the right triangle with vertices $p_4,-p_3$ and $p_5$, for which we denote by $\Omega$ (see Figure \ref{geodesictriangular}). Importantly, $\lbc|_{\Omega}$ is trivial, allowing us to consider a section of $\lbc$ as a complex-valued function when restricted to $\Omega$. 

\begin{figure}[!h]
    \centering

\tikzset{every picture/.style={line width=0.75pt}} %set default line width to 0.75pt        

\begin{tikzpicture}[x=0.75pt,y=0.75pt,yscale=-1,xscale=1]

%Straight Lines [id:da901480852206755] 
\draw [color={rgb, 255:red, 208; green, 2; blue, 27 }  ,draw opacity=1 ][line width=0.75]    (228.8,183.6) -- (204.44,187.24) ;
\draw [shift={(202.47,187.53)}, rotate = 351.5] [color={rgb, 255:red, 208; green, 2; blue, 27 }  ,draw opacity=1 ][line width=0.75]    (9.84,-2.96) .. controls (6.25,-1.25) and (2.97,-0.27) .. (0,0) .. controls (2.97,0.27) and (6.25,1.26) .. (9.84,2.96)   ;
%Straight Lines [id:da9791983952998549] 
\draw [color={rgb, 255:red, 208; green, 2; blue, 27 }  ,draw opacity=1 ][line width=0.75]    (228.8,183.6) -- (238.19,194.99) ;
\draw [shift={(239.47,196.53)}, rotate = 230.49] [color={rgb, 255:red, 208; green, 2; blue, 27 }  ,draw opacity=1 ][line width=0.75]    (7.65,-2.3) .. controls (4.86,-0.97) and (2.31,-0.21) .. (0,0) .. controls (2.31,0.21) and (4.86,0.98) .. (7.65,2.3)   ;
%Curve Lines [id:da43369942318019317] 
\draw [color={rgb, 255:red, 74; green, 144; blue, 226 }  ,draw opacity=1 ] [dash pattern={on 4.5pt off 4.5pt}]  (165.47,83.13) .. controls (169.47,73.13) and (192.47,49.33) .. (221.8,73.6) ;
%Curve Lines [id:da37425205595516253] 
\draw [color={rgb, 255:red, 74; green, 144; blue, 226 }  ,draw opacity=1 ] [dash pattern={on 4.5pt off 4.5pt}]  (221.8,73.6) .. controls (238.47,54.33) and (256.47,50.33) .. (275.47,61.33) ;
%Curve Lines [id:da8759401817616395] 
\draw [color={rgb, 255:red, 74; green, 144; blue, 226 }  ,draw opacity=1 ]   (228.8,183.6) .. controls (256.47,184.73) and (286.47,150.73) .. (286.8,107.6) ;
%Curve Lines [id:da976536187367852] 
\draw [color={rgb, 255:red, 74; green, 144; blue, 226 }  ,draw opacity=1 ]   (286.8,107.6) .. controls (301.47,77.33) and (268.47,52.53) .. (275.47,61.33) ;
%Curve Lines [id:da15950479409111695] 
\draw [color={rgb, 255:red, 74; green, 144; blue, 226 }  ,draw opacity=1 ]   (171.8,120.6) .. controls (195.47,89.73) and (250.47,81.73) .. (286.8,107.6) ;
%Curve Lines [id:da30453033599556023] 
\draw [color={rgb, 255:red, 74; green, 144; blue, 226 }  ,draw opacity=1 ]   (171.8,120.6) .. controls (172.47,144) and (187.47,180.6) .. (228.8,183.6) ;
%Shape: Circle [id:dp528898006798836] 
\draw  [color={rgb, 255:red, 0; green, 0; blue, 0 }  ,draw opacity=0.64 ][fill={rgb, 255:red, 74; green, 144; blue, 226 }  ,fill opacity=0.3 ] (217.5,73.6) .. controls (217.5,71.23) and (219.43,69.3) .. (221.8,69.3) .. controls (224.17,69.3) and (226.1,71.23) .. (226.1,73.6) .. controls (226.1,75.97) and (224.17,77.9) .. (221.8,77.9) .. controls (219.43,77.9) and (217.5,75.97) .. (217.5,73.6) -- cycle ;
%Shape: Circle [id:dp5152191832310002] 
\draw  [fill={rgb, 255:red, 74; green, 144; blue, 226 }  ,fill opacity=0.49 ] (224.5,183.6) .. controls (224.5,181.23) and (226.43,179.3) .. (228.8,179.3) .. controls (231.17,179.3) and (233.1,181.23) .. (233.1,183.6) .. controls (233.1,185.97) and (231.17,187.9) .. (228.8,187.9) .. controls (226.43,187.9) and (224.5,185.97) .. (224.5,183.6) -- cycle ;
%Curve Lines [id:da8824059839351022] 
\draw [color={rgb, 255:red, 74; green, 144; blue, 226 }  ,draw opacity=1 ] [dash pattern={on 4.5pt off 4.5pt}]  (221.8,73.6) .. controls (218.47,91.33) and (205.47,132.33) .. (228.8,183.6) ;
%Shape: Circle [id:dp6731190450110127] 
\draw   (158.07,112.87) .. controls (158.07,73.8) and (189.74,42.13) .. (228.8,42.13) .. controls (267.86,42.13) and (299.53,73.8) .. (299.53,112.87) .. controls (299.53,151.93) and (267.86,183.6) .. (228.8,183.6) .. controls (189.74,183.6) and (158.07,151.93) .. (158.07,112.87) -- cycle ;
%Curve Lines [id:da9099356705654735] 
\draw [color={rgb, 255:red, 74; green, 144; blue, 226 }  ,draw opacity=1 ]   (171.8,120.6) .. controls (158.47,103.13) and (162.47,89.13) .. (165.47,83.13) ;
%Straight Lines [id:da8904255554828842] 
\draw  [dash pattern={on 4.5pt off 4.5pt}]  (221.8,73.6) -- (230.47,147.33) ;
%Curve Lines [id:da06506053773957765] 
\draw [color={rgb, 255:red, 74; green, 144; blue, 226 }  ,draw opacity=1 ]   (171.8,120.6) .. controls (184.47,133.33) and (208.47,146.63) .. (230.47,151.63) ;
%Curve Lines [id:da4361743252525565] 
\draw [color={rgb, 255:red, 74; green, 144; blue, 226 }  ,draw opacity=1 ]   (190.47,164.33) .. controls (222.47,156.33) and (272.47,144.33) .. (286.8,107.6) ;
%Curve Lines [id:da2100072480102939] 
\draw [color={rgb, 255:red, 74; green, 144; blue, 226 }  ,draw opacity=1 ]   (228.8,183.6) .. controls (231.47,170.33) and (231.47,157.33) .. (230.47,147.33) ;
%Shape: Circle [id:dp5199795984116826] 
\draw  [fill={rgb, 255:red, 74; green, 144; blue, 226 }  ,fill opacity=0.49 ] (186.17,164.33) .. controls (186.17,161.96) and (188.09,160.03) .. (190.47,160.03) .. controls (192.84,160.03) and (194.77,161.96) .. (194.77,164.33) .. controls (194.77,166.71) and (192.84,168.63) .. (190.47,168.63) .. controls (188.09,168.63) and (186.17,166.71) .. (186.17,164.33) -- cycle ;
%Shape: Circle [id:dp7797740263623225] 
\draw  [fill={rgb, 255:red, 74; green, 144; blue, 226 }  ,fill opacity=0.49 ] (167.5,120.6) .. controls (167.5,118.23) and (169.43,116.3) .. (171.8,116.3) .. controls (174.17,116.3) and (176.1,118.23) .. (176.1,120.6) .. controls (176.1,122.97) and (174.17,124.9) .. (171.8,124.9) .. controls (169.43,124.9) and (167.5,122.97) .. (167.5,120.6) -- cycle ;
%Shape: Circle [id:dp8065527128253478] 
\draw  [fill={rgb, 255:red, 74; green, 144; blue, 226 }  ,fill opacity=0.49 ] (282.5,107.6) .. controls (282.5,105.23) and (284.43,103.3) .. (286.8,103.3) .. controls (289.17,103.3) and (291.1,105.23) .. (291.1,107.6) .. controls (291.1,109.97) and (289.17,111.9) .. (286.8,111.9) .. controls (284.43,111.9) and (282.5,109.97) .. (282.5,107.6) -- cycle ;
%Shape: Circle [id:dp014917325988958341] 
\draw  [fill={rgb, 255:red, 74; green, 144; blue, 226 }  ,fill opacity=0.49 ] (226.17,151.63) .. controls (226.17,149.26) and (228.09,147.33) .. (230.47,147.33) .. controls (232.84,147.33) and (234.77,149.26) .. (234.77,151.63) .. controls (234.77,154.01) and (232.84,155.93) .. (230.47,155.93) .. controls (228.09,155.93) and (226.17,154.01) .. (226.17,151.63) -- cycle ;
%Straight Lines [id:da5554520271109353] 
\draw [color={rgb, 255:red, 208; green, 2; blue, 27 }  ,draw opacity=1 ]   (200,162) -- (209.47,169.33) ;
%Straight Lines [id:da7869338806867625] 
\draw [color={rgb, 255:red, 208; green, 2; blue, 27 }  ,draw opacity=1 ]   (209.47,169.33) -- (198.47,173.33) ;
%Straight Lines [id:da3317598825727235] 
\draw    (490.47,150.33) -- (372.47,151.33) ;
%Straight Lines [id:da35450420100237867] 
\draw    (490.47,81.33) -- (372.47,151.33) ;
%Shape: Arc [id:dp24899187012323298] 
\draw  [draw opacity=0] (402.47,151.33) .. controls (402.47,145.87) and (401.01,140.74) .. (398.45,136.33) -- (372.47,151.33) -- cycle ; \draw   (402.47,151.33) .. controls (402.47,145.87) and (401.01,140.74) .. (398.45,136.33) ;  
%Straight Lines [id:da6630277395037947] 
\draw    (490.47,81.33) -- (490.47,150.33) ;

% Text Node
\draw (265.1,101) node [anchor=north west][inner sep=0.75pt]    {$p_{2}$};
% Text Node
\draw (187.8,55.7) node [anchor=north west][inner sep=0.75pt]    {$p_{3}$};
% Text Node
\draw (220,189.4) node [anchor=north west][inner sep=0.75pt]    {$p_{4}$};
% Text Node
\draw (137.1,111.07) node [anchor=north west][inner sep=0.75pt]    {$p_{1}$};
% Text Node
\draw (239.8,139.7) node [anchor=north west][inner sep=0.75pt]    {$-p_{3}$};
% Text Node
\draw (167,162.4) node [anchor=north west][inner sep=0.75pt]    {$p_{5}$};
% Text Node
\draw (212,159.4) node [anchor=north west][inner sep=0.75pt]    {$\Omega $};
% Text Node
\draw (349,139) node [anchor=north west][inner sep=0.75pt]    {$p_{4}$};
% Text Node
\draw (486,60.4) node [anchor=north west][inner sep=0.75pt]    {$-p_{3}$};
% Text Node
\draw (497,139) node [anchor=north west][inner sep=0.75pt]    {$p_{5}$};
% Text Node
\draw (448,116.4) node [anchor=north west][inner sep=0.75pt]    {$\Omega $};
% Text Node
\draw (433.47,154.23) node [anchor=north west][inner sep=0.75pt]    {$E_{1}$};
% Text Node
\draw (496,103.4) node [anchor=north west][inner sep=0.75pt]    {$E_{2}$};
% Text Node
\draw (420,85.4) node [anchor=north west][inner sep=0.75pt]    {$E_{3}$};

\end{tikzpicture}

    \caption{fundamental domains.}
    \label{geodesictriangular}
\end{figure}
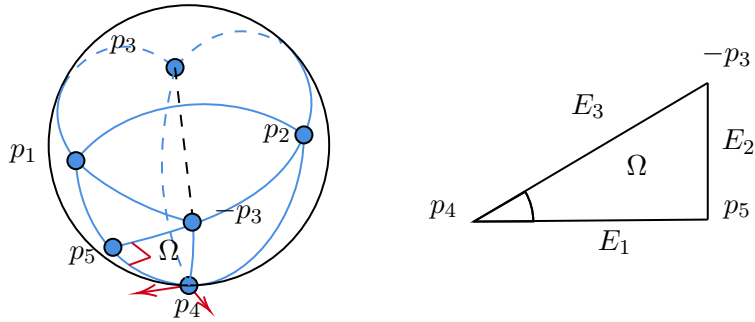

Given a critical eigensection $f$, as mentioned before $V_\lambda^c=\CC\{f\}$ is either a trivial or an alternating representation of $S_4$. We can see from the Figure \ref{THEFIGURE} that, $\ei  \sigma_0$ fixes $E_1$, and $\sigma_0 \sigma_4$ fixes $E_3$. Additionally, $\sigma_4 \sigma_1 \sigma_0$, which is a lift of $(1\,4)\in S_4< \mathrm{O}(3)$, preserves $E_2$. The actions of these maps on $f$ can be listed as the following table:
\begin{align*}
    \begin{array}{|c|c|c|c|}
    \hline
         &\ei \sigma_0 & \sigma_0 \sigma_4 & \sigma_4 \sigma_1 \sigma_0\\
         \cline{1-4}
    f\in U & -1      & 1    &  1 \\
    \cline{1-4}
    f\in U'&  1     &  -1    &  -1\\
    \cline{1-4}
\text{Fixed point set} & E_1 & E_2 & E_3\\
    \cline{1-4}
    \end{array}
\end{align*}

It follows that a critical eigensection $f$ has the following mixed boundary condition on $\Omega$:
\begin{align*}
    \begin{array}{|c|c|c|c|}
    \hline
         & E_1 & E_2 & E_3\\     
         \cline{1-4}
    f\in U & \text{Dirichlet}      & \text{Neumann}    &  \text{Neumann} \\
    \cline{1-4}
    f\in U'&  \text{Neumann}     &  \text{Dirichlet}    &  \text{Dirichlet}\\
    \cline{1-4}
    \end{array}
\end{align*}

Conversely, consider the Laplacian $\Delta_\Omega: C^{\infty}(\Omega)\to C^{\infty}(\Omega)$ on the spherical domain $\Omega$, with two kinds of mixed boundary conditions presented in the following, it has two different Friedrich extensions:  
\begin{align}
    &f|_{E_1}\equiv 0, ~~\frac{\partial f}{\partial \nu}|_{E_2\cup E_3}\equiv 0,\label{bdycond1}\\
    &f|_{E_2\cup E_3}\equiv 0, ~~\frac{\partial f}{\partial \nu}|_{E_1}\equiv 0.\label{bdycond2}
\end{align}
Consider an eigenfunction $f$ of $\Delta_\Omega$ of eigenvalue $\lambda$, with boundary condition \eqref{bdycond1} (resp. \eqref{bdycond2}). By employing the maps $\ei$ and $\sigma_l$ for $0\le l\le 4$, and conditions $\sigma_0.\tilde{f}=\tilde{f}$ (resp. $\sigma_0.\tilde{f}=-\tilde{f}$), $\sigma_l.\tilde{f}=\bar{f}$ for $l=1,\cdots,4$ and $\ei.\tilde{f}=-\tilde{f}$, we can extend $f$ to be an odd function on $\Sigma_\vp$. Moreover, by construction, such an extension $\tilde{f}$ satisfies $\Delta \tilde{f}+\lambda \tilde{f}=0$ on $\Sigma_\vp$ weakly. Therefore, $\tilde{f}$ can be characterized as an eigensection of $\lbc$, for which we still denote by $f$. And finally, $\CC\{f\}$ is the trivial (resp. alternating) representation of $S_4$.

To establish that $f$ is a critical eigensection, we consider the expansion $\tilde{f}(w_l,\bar{w}_l)=a_lw_l^{2n_l+1}+b_l\bar{w}_l^{2n_l+1}+\mathcal{O}(r^{2n_l+3})$. From Lemma \ref{rotacoorrel} and \eqref{sigmafw}, we deduce that $\tilde{f}(w_l,\bar{w}_l)=\sigma_l.\tilde{f}(w_l,\bar{w}_l)=a_le^{2(2n_l+1)\pi i/3}w_l^{2n_l+1}+b_le^{-2(2n_l+1)\pi i/3}\bar{w}_l^{2n_l+1}+\mathcal{O}(r^{2n_l+3})$. This implies that $2n_l+1\equiv 0 \mod 3$, and consequently, $n_l\ge 1$.

In summary, we can conclude the following:

\begin{proposition}\label{tetracriuni}
    There are infinitely many critical eigenvalues of $\Delta$ on $\lbc$. A $\ZT$ eigenvalue $\lambda$ is critical if and only if $V_\lambda^c$ is nontrivial and serves as an one dimensional $S_4$-representation. $V_\lambda^c$ is isomorphic with the trivial representation $U$ if and only if $\lambda$ is an eigenvalue of $\Delta_\Omega$ on $\Omega$ with boundary conditions \eqref{bdycond1}. Similarly, $V_\lambda^c$ is isomorphic with the alternating representation $U'$ if and only if $\lambda$ serves as an eigenvalue of $\Delta_\Omega$ with boundary condition \eqref{bdycond2} on $\Omega$. 
\end{proposition}

\begin{lemma}
    Suppose $\lam$ is a critical eigenvalue of $\lb$, then there is a real section $f\in V_\lam^c$, such that $V_\lam^c=\CC\{f\}$.
\end{lemma}
\begin{proof}
    Consider the involution $V_\lam^c\to V_\lam^c:h\mapsto \bar{h}$. This is a real linear map, and satisfies $||h||_{L^2}=||\bar{h}||_{L^2}$. Therefore, it has a nonzero eigenvector $f$, such that $\bar{f}=\pm f$. And note that $\overline{if}=\mp i f$. Thus we can assume $\bar{f}=f$, and hence $f=\re\, f$ is real. Since $\dim_{\CC}\, V_\lam^c=1$, $V_\lam^c=\CC\{f\}$. 
\end{proof}

This lemma shows that critical eigensections in Proposition \ref{tetracriuni}, which are constructed from eigenfunctions on $\Omega$, are complexifications of real eigensections actually. And we name the family of complex eigensections, as well as the corresponding real eigensections as \textbf{Taubes-Wu tetrahedral eigensections}.

It's important to emphasize that this construction of eigensections is equivalent to the construction demonstrated in \cite{taubeswu2020examples}. Furthermore, Proposition \ref{tetracriuni} implies that these are the only potential families of critical eigensections when the underlying configuration $\vp$ is the tetrahedral one.

\begin{proposition}
    All Taubes-Wu tetrahedral eigensections are non-degenerate.
\end{proposition}
\begin{proof}
    It directly follows from Lemma \ref{npf==1}, that a complex Taubes-Wu tetrahedral eigensection $f$ is non-degenerate. Consider $V_\lam^c=\CC\{f\}$, with $f=\bar{f}$. We shall show that $f$ is non-degenerate.

     Near a point $p\in\vp$, according to Lemma \ref{npf==1}, $f$ can be expanded as $f(z,\bar{z})=az^{3/2}+b \bar{z}^{3/2}+\mathcal{O}(r^{5/2})$. Since $f=\bar{f}$, $\bar{b}=a$, and hence $f(z,\bar{z})=\re\,2az^{3/2}+\mathcal{O}(r^{5/2})$. Thus $f$ is non-degenerate.
\end{proof}

Next, we delve into the study of the entire eigenspace $V_\lambda\subset\lbc$. According to the classical representation theory, $V_\lambda$ is a direct sum of irreducible $S_4$-representations, which have been listed in Lemma \ref{irreS4rep}. 

\begin{lemma}\label{noVV}
Both the standard representation $V$ and $V'=V\otimes U'$ can appear at most once in the decomposition of irreducible $S_4$-representations of $V_\lambda$.
\end{lemma}
\begin{proof}
    Let's assume there are two copies of $V$ within $V_\lambda$. This implies the existence of two linear independent eigensections, denoted as $f_1$ and $f_2$, such that $\sigma_k. f_l=f_l$ for $k=0,\cdots,4$ and $l=1,2$. We deduce from the $\sigma_4$-symmetry that $\vn_{p_4}(f_l)>0(l=1,2)$. 
    
    Since $\sigma_4.f_l=f_l$, $\vn_{p_i}(f_l)=\vn_{\sigma_4(p_i)}(f_l)$ for $l=1,2$, and consequently, $\vn_{p_i}(f_l)=0$ $(i=1,2,3)$ for $l=1,2$. Otherwise, at least one of $f_l$ would be critical, but $\sigma_3.f_l\ne f_l$ for $l=1,2$, contradicting with the assertion that a critical eigensection spans a one dimensional $S_4$-representation. Next, we expand the odd functions $\tilde{f}_l$ near $p_1$ as $\tilde{f}_l(w_1,\bar{w}_1)=a_lw_1+b_l\bar{w}_1+\mathcal{O}(r^{3})$ $(l=1,2)$. 
    
    Recalling that near $p_1\in\Sigma_\vp$, $\sigma_0$ maps $w_1$ to $\bar{w}_1$, as presented in Lemma \ref{rotacoorrel}, we conclude that $a_l=b_l$ for $l=1,2$. Let $f_3=a_{2}f_1-a_{1}f_2$, then $\vn_{p_1}(f_3)\ge 1$. However, since $\sigma_4.f_3=f_3$,we have $\vn_{p_l}(f_3)\ge 1$ for $l=1,2,3,4$. Thus it must be critical, belonging to either $U$ or $U'$. However, intersection between two different irreducible subrepresentations must be trivial, implying $f_3=0$. This indicates that $f_1$ and $f_2$ are linearly dependent, contradicting with our initial assumption.
\end{proof}

\begin{lemma}\label{noVV'}
The representations $V$ and $V'$ cannot simultaneously appear in the irreducible $S_4$-representation decomposition of $V_\lambda$.
\end{lemma}
\begin{proof}
    Suppose that there is a subrepresentation $V\oplus V'$ being contained in $V_\lambda$. Let's take $f_1$ and $f_2$ from $V$ and $V'$ respectively, so that $\sigma_4.f_s=f_s$ for $s=1,2$. Expand $\tilde{f}_1$ and $\tilde{f}_2$ in $w_l$-coordinates near $p_l$ as 
        $\tilde{f}_{s}(w_l,\bar{w}_l)=a_{l,s} w_l+b_{l,s}\bar{w}_l+\mathcal{O}(r^3)$, for $s=1,2$ and $l=1,2,3$. 
    By definition, $\sigma_0.\tilde{f}_1=\tilde{f}_1$ while $\sigma_0.\tilde{f}_2=-\tilde{f}_2$, implying that $a_{1,1}=b_{1,1}$ while $a_{1,2}=-b_{1,2}$. Furthermore, from the $\sigma_4$-symmetry, we deduce $a_{1,s}=a_{2,s}=a_{3,s}$ for $s=1,2$. Next, let's consider the eigensections $J_3f_s(s=1,2)$. By \eqref{Jcplxcoor} and Lemma \ref{Liealgetrans}, we conclude that the leading coefficients of sections $J_3f_s$ at $p_l$ are $\frac{\sqrt{2}}{3}a_{l,s}$ and $-\frac{\sqrt{2}}{3}b_{l,s}$ for $s=1,2$ and $l=1,2,3$. Set $\phi_1=f_1$ and $\phi_2=J_3f_2$ in Lemma \ref{lemmalocal} we then have 
    \begin{align*}
        0=\frac{\sqrt{2}}{3}\sum_{l=1}^3(a_{l,1}a_{l,2}-b_{l,1}b_{l,2})=2\sqrt{2} a_{1,1}a_{1,2}.
    \end{align*}
    As a consequence, either $\vn_{p_l}(f_1)$ or $\vn_{p_l}(f_2)$ must be positive for all $1\le l\le 3$. Additionally, by the $\sigma_4$-symmetry, $\vn_{p_4}(f_s)>0$ for $l=1,2$. However, since neither $f_1$ nor $f_2$ is critical, these lead to a contradiction.  
\end{proof}

\begin{lemma}\label{noWW}
    There is at most one irreducible summand that isomorphic with $W$ within the direct sum decomposition of the $S_4$-representation $V_\lambda$.
\end{lemma}
\begin{proof}
    Suppose there are two copies of $W$ within $V_\lambda$. Referred to Lemma \ref{irreS4rep}, we can find two linearly independent eigensections $f_1$ and $f_2$ being contained in different copies of $W$, such that $\sigma_l.f_s=e^{\frac{2\pi}{3}i}f_s$ for $s=1,2$. By \eqref{sigmafw} and Proposition \ref{tetracriuni}, for each $f_s$, $\vn_{p_l}(f_s)$ are all the same, and hence must be zero since $f_s$ is non-critical. Expand $f_s$ near $p_l$ as $f(z_l,\bar{z}_l)=a_{l,s}z_l^{1/2}+b_{l,s}\bar{z}_l^{1/2}+\mathcal{O}(r^{3/2})$. Since $\sigma_l.f_s=e^{\frac{2\pi}{3}i}f_l$, combining Lemma \ref{rotacoorrel} with \eqref{sigmafw}, we have 
    $e^{\frac{2\pi}{3}i}a_{l,s}w_l+e^{\frac{2\pi}{3}i}b_{l,s}\bar{w}_l+\mathcal{O}(r^{\frac32})=e^{\frac{2\pi}{3}i}\tilde{f}_s(w_l,\bar{w}_l)=(\sigma_l.\tilde{f}_s)(w_l,\bar{w}_l)
    =a_{l,s}e^{-\frac{2\pi}{3}i}w_l+b_{l,s}e^{\frac{2\pi}{3}i}b_{l,s}\bar{w}_l+\mathcal{O}(r^{\frac32})$.
    Thus $(a_{l,s},b_{l,s})=e^{\frac{2\pi}{3}i}(e^{\frac{2\pi}{3}i}a_{l,s},e^{-\frac{2\pi}{3}i}b_{l,s})$. It follows that $a_{l,s}=0$. Moreover, by comparing the coefficients of terms of order $3/2$, we deduce that such terms are vanishing. Thus $f(z_l,\bar{z}_l)=b_{l,s}\bar{z}_l^{1/2}+\mathcal{O}(r^{5/2})$.

    Therefore, the section $f_3:=b_{4,2}f_1-b_{4,1}f_2$ satisfies $\vn_{p_4}(f_3)\ge 1$. However, $\sigma_l.f_3=b_{4,2}\sigma_l.f_1-b_{4,1}\sigma_l.f_2=e^{\frac{2\pi}{3}i}f_3$, as a result, $\vn_{p_l}(f_3)\ge 1$ for any $l=1,2,3,4$, and hence $f_3$ is critical. According to Proposition \ref{tetracriuni}, such a section spans a one-diemsional $S_4$-representation, but such a subspace intersects with $W^{\oplus 2}$ trivially. Consequently, $f_3=0$ and $f_1,f_2$ are linearly dependent, leading to a contradiction.  
\end{proof}

\begin{lemma}\label{noW}
    The $S_4$-representation $V_\lambda$ cannot have two irreducible summands such that one of which is isomorphic with $W$ and the other is isomorphic with either $V$ or $V'$. 
\end{lemma}
\begin{proof}
    Suppose there is a $W$ contained in $V_\lambda$. Recall that there is a basis $\{v,\sigma_0.v\}$ of $W$, such that $\sigma_l.v=e^{\frac{2\pi}{3}i}v$ and $\sigma_l.(\sigma_0.v)=-e^{\frac{\pi}{3}i}\sigma_0.v$ for $1\le l\le 4$. Let's denote $v$ as $h_2$ and $\sigma_0.v$ as $h_1$. From the non-critical nature of $h_1$ and $h_2$ as well as Lemma \ref{rotacoorrel} and the $S_4$-symmetries, it follows that $\vn_{p_l}(h_s)=0$ for $s=1,2$ and $1\le l\le 4$. Combining Lemma \ref{rotacoorrel} with \eqref{sigmafw}, as in the proof of the previous lemma, we deduce that $h_1(z_l,\bar{z}_l)=a_lz_l^{1/2}+\mathcal{O}(r^{5/2})$ and $h_2(z_l,\bar{z}_l)=b_l\bar{z}_l^{1/2}+\mathcal{O}(r^{5/2})$ for $1\le l\le 4$. By definition, $\sigma_0.h_1=h_2$, as a result, we conclude that $a_4=-b_4$ and $a_1=b_1$. Let's set $a_4=1$ from now on. Combining $\sigma_2. w_1=e^{-\frac{\pi}{6}i}w_4$ with $\sigma_2. h_1=e^{-\frac{2\pi}{3}i}h_1$ yields $e^{-\frac{2\pi}{3} i}\tilde{h}_1(w_4,\bar{w}_4)=(\sigma_2. \tilde{h}_1)(w_4,\bar{w}_4)=a_1(\sigma_2. w_1)+\mathcal{O}(r^5)=e^{-\frac{\pi}{6}i}a_1w_4+\mathcal{O}(r^5)$, where $\tilde{h}_1$ is the odd function associated to $h_1$. By comparing coefficients, we conclude $a_1=e^{(\frac{\pi}{6}-\frac{2\pi}{3})i}a_4=-i$.

    Since $\sigma_l.\bar{h}_1=\overline{\sigma_l.h_1}=e^{2\pi i/3}\bar{h}_1$ for $l=1,2,3,4$, we deduce from this as well as Lemma \ref{noWW} that $\bar{h}_1=ah_2$ for some $a\in\mathbb{C}^\ast$. Comparing the leading coefficients of $\bar{h}_1$ and $ah_2$ at $p_4$, yields $-a=ab_4=\bar{a}_4=1$, and hence $\bar{h}_1=-h_2$. Therefore $\zeta_1:=h_1-h_2$ and $\zeta_2:=i(h_1+h_2)$ are real eigensections, i.e. eigensections of $\lb$. Additionally, they satisfy $\sigma_0.\zeta_1=-\zeta_1$ and $\sigma_0.\zeta_2=\zeta_2$. 

    Suppose that there is one copy of $V$ contained in $V_\lambda$. We shall derive a contradiction next.
    
    The leading coefficients of $\zeta_2$ at $p_1$ are $(ia_1,ia_1)=(1,1)$. From the symmetric conditions $\sigma_4.h_1=e^{\frac{4\pi}{3}i}h_1$ and $\sigma_4.h_2=e^{\frac{2\pi}{3}i}h_2$, we calculate the leading coefficients of $\zeta_2$ at $p_2$ and $p_3$ are $(e^{\frac{4\pi}{3}i},e^{\frac{2\pi}{3} i})$ and $(e^{\frac{2\pi}{3}i},e^{\frac{4\pi}{3}i})$. Similarly, $\zeta_1$ has leading coefficients $(-i,i)$, $(-ie^{\frac{4\pi}{3}i},ie^{\frac{2\pi}{3}i})$, $(-ie^{\frac{2\pi}{3}i},ie^{\frac{4\pi}{3}i})$ and $(1,1)$ at $p_1$, $p_2$, $p_3$ and $p_4$  respectively. 

    Take $f_1$ and $f_4$ from $V$ such that $\sigma_l.f_l=f_l(l=1,4)$. The complex conjugation $f_l\mapsto \bar{f}_l(l=1,4)$ keeps the $S_4$-symmetric conditions. According to Lemma \ref{noVV} and \ref{noVV'} there is only one copy of $V$ in $V_\lambda$, thus we can assume both $f_1$ and $f_4$ are real. Moreover, since $\sigma_0. f_l=f_l$ by definition, we can assume $f_1$ and $f_4$ take the same leading coefficients with $\zeta_2$ at $p_4$ and $p_1$ respectively. From $\sigma_4. f_4=f_4$, we deduce that the leading coefficients of $f_4$ at $p_2$ and $p_3$ are both $(1,1)$. With the leading coefficients at $p_4$ being given, combining with $\sigma_1. f_1=f_1$, $\eqref{sigmafw}$ as well as the final bullet in Lemma \ref{rotacoorrel}, we conclude that $f_1$ takes leading coefficients $(e^{\frac{2\pi}{3} i},e^{\frac{4\pi}{3} i})$ at $p_2$ and takes $(-e^{\frac{\pi}{3}i},-e^{-\frac{\pi}{3} i})$ at $p_3$.

    We summarize the coefficients of these eigensections of the terms of order $\frac{1}{2}$ in the following table:
   \begin{spacing}{1.2} \[\begin{array}{|c|c|c|c|c|}
     \hline
           & p_1 & p_2 & p_3 & p_4\\
         \cline{1-5}
         \zeta_1 & (-i,i) & (-ie^{4\pi i/3},ie^{2\pi i/3}) & (-ie^{2\pi i/3},ie^{4\pi i/3}) & (1,1)\\
         \cline{1-5}
         \zeta_2 & (1,1) & (e^{4\pi i/3},e^{2\pi i/3}) & (e^{2\pi i/3},e^{4\pi i/3}) & (i,-i)\\
         \cline{1-5}
         f_1 & (0,0) & (e^{2\pi i/3},e^{4\pi i/3}) & (-e^{\pi i/3},-e^{-\pi i/3}) & (i,-i)\\
         \cline{1-5}
         f_4 & (1,1) & (1,1) & (1,1) & (0,0)\\
         \cline{1-5}
\end{array}\]\end{spacing}

    Define a new eigensection as $\psi:=f_1+f_4-\zeta_2$. We note first that $\vn_{p_l}(\psi)>0$ for $l=1,4$. The leading coefficients of $\psi$ at $p_2$ and $p_3$ are $(-2e^{\frac{4\pi}{3}i},-2e^{\frac{2\pi}{3} i})$ and $(-2e^{\frac{2\pi}{3}i},-2e^{\frac{4\pi}{3}i})$. The section $\frac{3\sqrt{2}}{2}J_3\zeta_1$ has leading coefficients $(-ie^{\frac{4\pi}{3}i},-ie^{\frac{2\pi}{3}i})$ and $(-ie^{\frac{2\pi}{3}i},-ie^{\frac{4\pi}{3}i})$ at $p_2$ and $p_3$.    
    Setting $\phi_1:=\psi$ and $\phi_2:=\frac{3\sqrt{2}}{2}J_3\zeta_1$ in Lemma \ref{lemmalocal} yields $0=4i(e^{\frac{2\pi}{3}i}+e^{\frac{4\pi}{3}i})\ne 0$, a contradiction. This shows that $W$ and $V$ cannot appear simultaneously. 
    
    Next assume $V'\subset V_\lam$, and consider eigensections $f_1',f_4'\in V'$, such that $\sigma_l.f_l'=f_l'$ for $l=1,4$, and $f_1',f_4'$ have similar leading coefficients with $\zeta_1$ at $p_4,p_1$ respectively. We obtain the following table of leading coefficients: 
    \begin{spacing}{1.2}\[\begin{array}{|c|c|c|c|c|}
     \hline
           & p_1 & p_2 & p_3 & p_4\\
         \cline{1-5}
         f'_1 & (0,0) & (e^{\pi i/6},e^{-\pi i/6}) & (e^{5\pi i/6},-e^{\pi i/6}) & (1,1)\\
         \cline{1-5}
         f'_4 & (-i,i) & (-i,i) & (-i,i) & (0,0)\\
         \cline{1-5}
\end{array}\]\end{spacing}
    Consider $\phi_1=f'_1+f_4'-\zeta_1$ and $\phi_2=\frac{3\sqrt{2}}{2}J_3\zeta_2$ in Lemma \ref{lemmalocal}, similar calculations lead to another contradiction.
\end{proof}

\begin{proposition}
    If $\lambda$ is a critical eigenvalue, then $V_\lambda$ is isomorphic with $U\oplus V'$ or $U'\oplus V$. If $\lambda$ is not critical, $V_\lambda$ is isomorphic either with $V$, $V'$ or $W$.
\end{proposition}
\begin{proof}
    Suppose that $\lambda$ is critical and the $V_\lambda^c$ is isomorphic with the $S_4$-representation $U$. Consider the subspace $JV:=\{Jf:f\in U\text{ and }J\text{ is a vector field commutes with }\Delta\}$ and the four vector fields $J_3^{p_l}(l=1,2,3,4)$, where $J_3^{p_4}=J_3$. It is easy to see that the $\{J_3^{p_1}f:l=1,2,3,4\}$ spans a space consisting of vector fields that commute with the Laplacian. Moreover, the group $A_4$ acts on this vector fields by permuting the subscripts, and $\sigma_0^\ast J_3=-J_3$. Thus $\dim_\CC\, JV=3$, and as a consequence, $JV\cong U\otimes V'\cong V'$. If $V_{\lambda}^c\cong U'$ the subspace $JV$ is isomorphic with $U'\otimes V'\cong V$.

    Therefore, there is a copy of either $U\oplus V'$ or $U'\oplus V$ within $V_\lambda$. It directly follows from Lemmas \ref{noVV}, \ref{noVV'} and \ref{noW} that $V_\lambda$ is isomorphic with either $U\oplus V'$ or $U'\oplus V$. If $\lambda$ is not critical, the decomposition of $V_{\lambda}$ consists only of summands that are isomorphic with $V$, $V'$ or $W$. Apply the lemmas above once again we deduce the second claim. 
\end{proof}

The actual eigensections we require are real sections. According to the series of lemma presented above, each irreducible $S_4$-representation can appear at most once in the direct sum decomposition of $V_\lambda$. Consequently, the map $f\mapsto \bar{f}$ induces linear automorphism on each summand. For $U$ (resp. $U'$), there must be a real section $f$ such that $U\cong\CC\{f\}$ (resp. $U'\cong\CC\{f\} $). For $V$ (resp. $V'$), there is a basis $\{f_l\}_{1\le l\le3}$ (resp. $\{f'_l\}_{1\le l\le3}$), with $\sigma_l.f_l=f_l$ and $\sigma_0.f_l=f_l$ (resp. $\sigma_0.f'_l=-f'_l$). Note that the map of conjugation keeps these symmetric conditions, thus $\bar{f}_l=a_lf_l$ for some $a_l\in\CC^\ast$. We can choose the basis $\{f_l\}_{1\le l\le 3}$ so that $a_l=1$, and this spans a real subspace consisting of real sections. For the $2$-dimensional representation $W$, we have constructed a basis of real sections $\{\zeta_1,\zeta_2\}$ in the proof of Lemma \ref{noW}. Denote these real representations of $S_4$ by $\re\,U,\re\,U', \re\,V$, $\re\,V'$ and $\re\,W$. We conclude:
\begin{proposition}
    For the space of real sections $\sob$, an eigenspace is isomorphic with one of $\re\,W$, $\re\,V$, $\re\,V'$, $\re\,U\oplus\re\,V'$ and $\re\,U'\oplus \re\,V$, as a real $S_4$-representation.  
\end{proposition}

The real dimensions of representations are given as: $\dim_{\Rr}\,\re\,U=\dim_{\Rr}\,\re\,U'=1$, $ \dim_{\Rr}\,\re\,V=\dim_{\Rr}\,\re\,V'=3$ and $\dim_{\Rr}\,\re\,W=2$. In summary, we conclude the following
\begin{corollary}
For an eigenvalue $\lambda$ of $\Delta$ on $\lb$, we observe that the multiplicity of $\lambda$ can be either $2$, $3$, or $4$. Moreover, $\lambda$ has multiplicity $4$ if and only if it is critical.
\end{corollary}

For an eigenspace $V_\lambda$, the preceeding proposition enables us to choose a basis for $V_\lambda$ and calculate the bilinear form as presented in Lemma \ref{TWformula}. Given an real eigensection $f$, we expand it near $p_l$ as $f(z_l,\bar{z}_l)=2\re\,a_lz_l^{n_l+1/2}+\mathcal{O}(r^{n_l+3/2})=a_lz_l^{n_l}+\bar{a}_l\bar{z}_l^{n_l+1/2}+\mathcal{O}(r^{n_l+3/2})$, where $n_l=\vn_{p_l}(f)$. We summarize the choices of basis and the corresponding leading coefficients $a_{l}(1\le l\le 4)$ in the following table:

\begin{spacing}{1.2}

    \[\begin{array}{|c|c|c|c|c|c|}
    \hline
        &          & a_1 & a_2 & a_3 & a_4\\ \cline{1-6}
\re W & \zeta_1 &  -i  &  -ie^{4\pi i/3}  &  -ie^{2\pi i/3} & 1 \\ 
\cline{2-6}
        & \zeta_2 &  1   &  e^{4\pi i/3}  &  e^{2\pi i/3}  &  i \\  \cline{1-6}
        &  f_1  &  0  &  e^{2\pi i/3}  &  e^{4\pi i/3}  &  i  \\ 
        \cline{2-6}
 \re V  &  f_2  &  e^{4\pi i/3}  &  0  &  -e^{-\pi i/3}  &  -e^{\pi i/6} \\ 
 \cline{2-6}
        & f_3 & -e^{-\pi i/3} & e^{4\pi i/3} & 0 &  -e^{2\pi i/3} \\ \cline{1-6}
        & f'_1 & 0 & -e^{\pi i/6} & e^{-\pi i/6} & -1 \\ 
        \cline{2-6}
\re V' & f'_2 & e^{-\pi i/6} & 0 & -e^{\pi i/6} & -e^{2\pi i/3} \\ 
\cline{2-6}
        & f'_3 & -e^{\pi i/6} & e^{-\pi i/6} & 0 & e^{\pi i/3} \\ 
        \cline{1-6}
    \end{array}\]

\end{spacing}

Given this, we can calculate the Taubes-Wu bilinear form in this tetrahedral case.

\begin{lemma}\label{bilinearform}
Consider a tangent vector $\vec{v}\in T_\vp \Cn$, with $\vec{v}=(v(p))_{p\in\vp}=(r_le^{i\theta_l})_{1\le l\le 4}$, and $r_4=0$, $\theta_1=0$, $r_1\in\RR$ and $r_2,r_3\ge 0$. The corresponding Taubes-Wu bilinear form on an eigenspace $V_\lambda$ defined in Lemma \ref{TWformula} can be listed as follows:
\begin{enumerate}[label=(\roman*)]
    \item If $V_\lambda\cong \re\,W$, then 
    \[\frac{1}{2\pi}B_{\vec{v}}=\left(\begin{array}{cc}
        -r_1-r_2\cos(\theta_2+\frac{2\pi}{3})-r_3\cos(\theta_3+\frac{4\pi}{3}) & r_2\cos(\theta_2+\frac{\pi}{6})+r_3\cos(\theta_3+\frac{5\pi}{6})\\
        r_2\cos(\theta_2+\frac{\pi}{6})+r_3\cos(\theta_3+\frac{5\pi}{6}) & 
        r_1+r_2\cos(\theta_2+\frac{2\pi}{3})+r_3\cos(\theta_3+\frac{4\pi}{3})
    \end{array}\right).\]
    \item If $V_\lambda\cong \re\, V$, then
     \[\frac{1}{2\pi}B_{\vec{v}}=\left(\begin{array}{ccc}
      r_2\cos(\theta_2+\frac{4\pi}{3})+r_3\cos(\theta_3+\frac{2\pi}{3})  & r_3\cos\theta_3  & r_2\cos\theta_2  \\
       r_3\cos\theta_3 & r_3\cos(\theta_3-\frac{2\pi}{3})-\frac{1}{2}r_1  & r_1  \\
       r_2\cos\theta_2 &  r_1 & r_2\cos(\theta_2+\frac{2\pi}{3})-\frac{1}{2}r_1
     \end{array}\right).\]
     \item If $V_\lambda\cong \re\, V'$, then
     \[\frac{1}{2\pi}B_{\vec{v}}=\left(\begin{array}{ccc}
      r_2\cos(\theta_2+\frac{\pi}{3})+r_3\cos(\theta_3-\frac{\pi}{3})  & -r_3\cos\theta_3  & -r_2\cos\theta_2  \\
       -r_3\cos\theta_3 & \frac{1}{2}r_1+r_3\cos(\theta_3+\frac{\pi}{3})  & -r_1  \\
       -r_2\cos\theta_2 &  -r_1 & \frac{1}{2}r_1+r_2\cos(\theta_2+\frac{\pi}{3})
     \end{array}\right).\]
     \item If $V_\lambda\cong \re\, U'\oplus \re\, V$ or $\re\, U\oplus \re\, V'$, the corresponding $B_{\vec{v}}$ has the matrix similar to $(2)$ or $(3)$ but with one additional row and column filled with zero. 
\end{enumerate}
\end{lemma}

\begin{proposition}\label{Misolated}
    The critical configuration $\vp$, which is the set of vertices of a regular tetrahedron, is $M$-isolated for any $0<M<\infty$.
\end{proposition}
\begin{proof}
    Recall that we should modulus the group of rotations. For this, we only consider configurations $\vq=\{q_1,q_2,q_3,q_4\}$ near $\vp$, such that $p_4=q_4$ and $q_1$ lies on the great circle passing through $p_1$ and $p_4$. 

    Choose a smooth path $\gamma$ on $\mathcal{C}_{4}$ with $\gamma(0)=\vp$ and $\gamma(1)=\vq$. Suppose additionally that $\gamma$ consists of geodesics on $\sph$. The tangent vector $\vec{v}=\gamma'(0)$ at $\vp$ satisfies $v(p_4)=0$ and $v(p_1)=r_1\in\RR$, as the assumption of Lemma \ref{bilinearform}. Let $\lambda<M$ be an eigenvalue of $\lb$, $V_\lambda$ the corresponding eigenspace and $\mul\,\lambda=N=2,3,4$. For the smooth path $\gamma(t)$, according to Lemma \ref{TWformula}, there are $C^1$-functions $\mu_i(t)(1\le i\le N=\mul\,\lambda)$ such that $\mu_i(t)$ are eigenvalues of $\mathcal{I}_{\gamma(t)}$ and $\mu_i'(0)$ are eigenvalues of the bilinear form $B_{\vec{v}}$ defined on $V_\lambda$. We want to show that none of $\mu_i(t)$ can be critical if $t>0$ is small enough. By Lemma \ref{mulge4}, it suffices to show $\mul\,\mu_i(t)<4$ for sufficiently small $t$.
    
   From Lemma \ref{bilinearform}, one can deduce that for any nonzero $\vec{v}$ satisfying the assumption there, $B_{\vec{v}}$ is non-zero. If 
 $V_\lambda\cong \re\, U'\oplus\re\, V$ or $\re\, U\oplus \re\, V'$, then $B_{\vec{v}}$ must have one zero eigenvalue and a nonzero one. Therefore, for some $i\ne j$, $\mu_i'(0)=0$ while $\mu_j'(0)\ne 0$. As a result, $\mu_i(t)\ne\mu_j(t)$ for $t\ne 0$ small enough. For the fixed $M>0$, consider $\psi_1<\cdots<\psi_k< M$ so that the set $\{\psi_s\}_{1\le s\le k}$ exhausts eigenvalues of $\lb$ less than $M$. Choose $\delta>0$ sufficiently small so that $C(\delta_0,\delta)$ given in Lemma \ref{contiofeig} satisfies $C^2<\psi_{i+1}/\psi_i$. An eigenvalue $\mu_i(t)$ of $\mathcal{I}_{\gamma(t)}$, splitting from an eigenvalue $\psi_s$ of $\lb$, cannot be one that comes from $\psi_{j}(j\ne s)$ at the same time. Therefore, for $t$ small enough, $\mul\,\mu_i(t)<\mul\,\psi_s\le 4$. 

Now consider the set of tangent vectors at $\vp$, $\mathfrak{C}_0(\vp)=\{\vec{v}\in T_\vp \mathcal{C}_4:v(p_4)=0,v(p_1)\in\RR, 1=\sum_i|v(p_i)|^2\}$. For each $\vec{v}\in \mathfrak{C}_0(\vp)$, consider the path $\vp_{v}(t)$ with each of its components being geodesic on $\sph$ and $\vp_{v}'(0)=\vec{v}$. There is $t_{v}>0$, so that each eigenvalue of $\mathcal{I}_{\vp_{v}(t)}$ that less than $M$, has multiplicity no more than $3$, provided $0<|t|<t_{v}$. By the continuity of eigenvalues there is a neighborhood of $\vec{v}$ on $\mathfrak{C}_0(\vp)$, say $U_{v}$, such that for any $\vec{w}\in U_{v}$, $\mathcal{I}_{\vp_{w}(t)}$ has no eigenvalues with multiplicity more than $3$ below $M$ if $0<|t|<t_{v}/2$. 

Since $\mathfrak{C}_0(\vp)$ is compact, it contains finitely many vectors, say $\vec{v}_1,\cdots,\vec{v}_m$, such that $U_{v_1},\cdots, U_{v_m}$ cover $\mathfrak{C}_0(\vp)$. Consequently, there exists $t_0>0$, such that for any $\vec{v}\in\mathfrak{C}_0(\vp)$, if $t<t_0$, then $\vp_v(t)$ is non-critical.

Consider a configuration $\vq\in\Cn$, such that $A.\vq\ne\vp$ for any $A\in\SO$, and $\dist(\vp,\vq)$ is small enough. There exists a vector $\vec{v}\in \mathfrak{C}_0(\vp)$, such that $\vq=\vp_v(t)$ for some $t$ dominated by $\dist(\vp,\vq)$. Therefore, $\vq$ is non-critical if $\dist(\vp,\vq)$ is sufficiently small. 
\end{proof}

As a consequence, we conclude one of the main theorems of this article:
\begin{theorem}
All Taubes-Wu tetrahedral eigensections are deformation rigid.
\end{theorem}

\subsection{The cross configuration}\label{5.4}
In this subsection, applying a $\ZT^{\oplus 3}$-action on $\sph$, we shall show that the cross configurations $\vp\in \mathcal{C}_4$, i.e. configurations that consist of two pair of antipodal points, are non-critical. As before, we shall prove that the $\mathbb{Z}_2^{\oplus 3}$-action on $\sph$ can be lifted to $\Sigma_{\vp}$. Next we discuss the relations between eigensections. Applying a shifting operator borrowed from quantum mechanics (cf. \cite[Chapter VI, Complements A]{quantum}), we can derive a contradiction provided there exists a critical eigensection.

Consider $\vp=\{p_1,\cdots,p_4\}$. Assume $p_4=(0,0,-1)$, $p_2=-p_4$, $p_1=(\sin\theta,0,\cos\theta)$,where $\theta\in (0,\pi)$, and $p_3=-p_1$. As the previous subsection, for convenient, we consider the complexification $\lbc$ of $\lb$. 

Define $\sigma_0$ as the reflection along $xOz$-plane, $\sigma_1$ as the reflection exchanging $p_1$ and $p_2$ and $\sigma_2$ as one that exchanges $p_1$ and $p_4$. With these as generators, a $\ZT^{\oplus 3}$-action on $\sph$ preserving $\vp$ can be defined. The great circle passing through $\vp$ can be lifted to the torus $\Sigma_\vp$ and hence to $\CC$ as a lattice. Let $\tilde{\sigma}_0,\tilde{\sigma_1}$ and $\tilde{\sigma}_2$ be chosen lifts of $\sigma_0,\sigma_1$ and $\sigma_2$, respectively, see Figure \ref{Lattice}. All of these lifts are reflections along some straight lines in $\CC$, for which we represent by green dashed lines in Figure \ref{Lattice}. We verify firstly that these involutions generate a $\mathbb{Z}_2^{\oplus 3}$-action on $\Sigma_\vp$. 

\begin{lemma}
$\tilde{\sigma}_i(i=0,1,2)$ generate a group isomorphic with $\ZT^{\oplus 3}$.
\end{lemma}
\begin{proof}
    It suffices to demonstrate that $\tilde{\sigma}_k$ commutes with each other.

    Since $\sigma_i \sigma_j=\sigma_j \sigma_i$, $\tilde{\sigma}_i \tilde{\sigma}_j$ is identical with either $\tilde{\sigma}_j \tilde{\sigma}_i$ or $\ei \tilde{\sigma}_j \tilde{\sigma}_i$. As a reminder, recall that by Lemma \ref{commuteswithinv}, $\ei$ commutes with any $\tilde{\sigma}_k$.

    As an illustration, consider $i=0$ and $j=1$. On $\Sigma_\vp$, as one can see from the Figure \ref{Lattice}, $\tilde{\sigma}_1$ maps $A^+$ to itself, and $\tilde{\sigma}_0$ maps $A^+$ onto $B^-$, hence $\tilde{\sigma}_0 \tilde{\sigma}_1$ maps $A^+$ onto $B^-$. However, $\tilde{\sigma}_1$ maps $B^-$ onto itself while $\ei$ maps $B^-$ to $B^+$. As a result, $\ei\tilde{\sigma}_j \tilde{\sigma}_i$ maps $A^+$ onto $B^-$, which evidently can not coincide with $\tilde{\sigma}_0 \tilde{\sigma}_1$. Consequently, $\tilde{\sigma}_0 \tilde{\sigma}_1=\tilde{\sigma}_1 \tilde{\sigma}_0$. 

    Similar arguments can be applied to the remaining cases.
\end{proof}

\begin{figure}[!h]
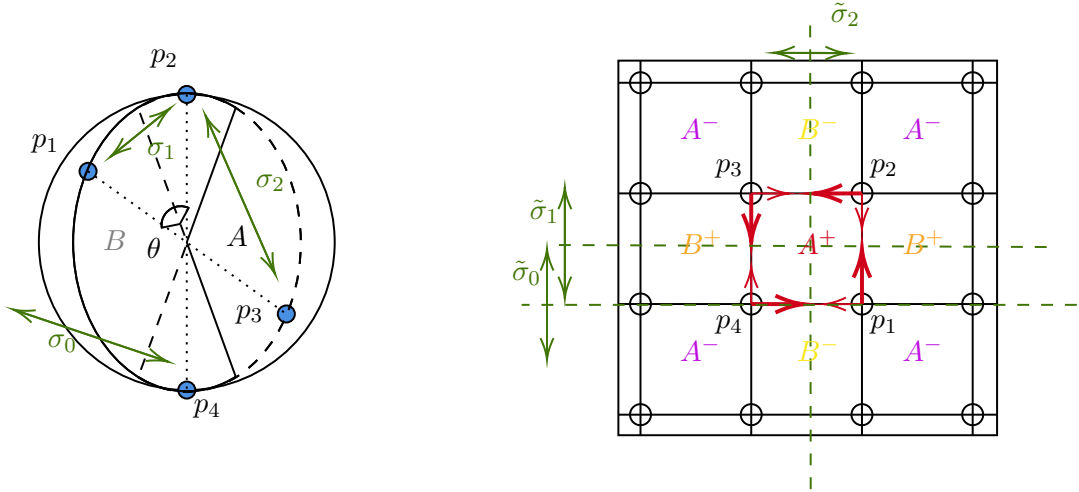

    \centering
\tikzset{every picture/.style={line width=0.75pt}} %set default line width to 0.75pt        


    \caption{The great circle being lifted to a lattice on $\CC$.}
    \label{Lattice}
\end{figure}

Now we have defined a $\mathbb{Z}_2^{\oplus 3}$-action on the surface $\Sigma_\vp$ commuting with $\ei$. Therefore, according to Section \ref{grpaction}, given an eigenvalue $\lambda$ of $\Delta$ on $\lbc$, the eigenspace $V_\lambda\subset \sobc$ can be regarded as a finite dimensional $\mathbb{Z}_2^{\oplus 3}$-representation. For simplicity, we denote $\tilde{\sigma}_k$ by $\sigma_k$ $(0\le k\le 2)$ henceforth.

It follows that $V_\lambda$ can be decomposed into direct sum of irreducible $\mathbb{Z}_2^{\oplus 3}$-representations, in where each summand is spanned by an eigensection that is eigenvector of the three generators of $\mathbb{Z}_2^{\oplus 3}$ simutaneously. An irreducible $\mathbb{Z}_2^{\oplus 3}$-representation is determined by the eigenvalues of reflections $\sigma_k$ $(0\le k\le 2)$ on it, which can only take values in $\pm 1$. 

Next we fix a choice of complex coordinates near $p_l$ in $\Sigma_\vp$. In Figure \ref{Lattice}, each bolder red arrow represents the direction of real axis of the coordinate at that point, while each lighter red arrow red arrow represents the imaginary direction. As before, we consider the action of $\sigma_k$ on these coordinates. And this can be derived from the Figure \ref{Lattice} immediately.
\begin{lemma}\label{Z23trans}
Let $w_l$ be the corresponding complex coordinate near $p_l$ for $l=1,2,3,4$. Define $\sigma_k.w_l$ as $w_l\circ\sigma_k^{-1}$, we then have:
    \begin{enumerate}[label=(\roman*)]
        \item $\sigma_0.w_l=-\bar{w}_l$ for $l=1,3$ and $\sigma_0.w_k=\bar{w}_k$ for $k=2,4$.
        \item $\sigma_1.w_l=i\bar{w}_{l+1}$ for $l=1,3$. 
        \item $\sigma_2.w_1=i\bar{w}_4$ and $\sigma_2.w_2=iw_3$.
    \end{enumerate}
\end{lemma}

\begin{lemma}
    Let $V\subset V_\lambda$ be an irreducible summand. Consider a nonzero $f\in V$, it follows that $\vn_{p}(f)$ are all the same for $p\in\vp$. 
\end{lemma}
\begin{proof}
    For any two points $p_l,p_k\in\vp$, one can always find $g\in\mathbb{Z}_2^{\oplus 3}$ that exchanges this two points. By definition, $\sigma_lf=\pm f$ for each $l$ and hence $g.f=\pm f$. Thus $\vn_{p_l}(f)=\vn_{p_l}(g.f)=\vn_{p_k}(f)$.
\end{proof}

Suppose $\lambda$ is critical, then by the previous lemma, there is an irreducible summand $V$ of $V_\lambda$, that spanned by an eigensection $f$ with $\vn_p(f)\ge 1,\forall p\in\vp$. Note that $J_2f$ satisfies $\vn_p(J_2f)=\vn_p(f)-1$ for any $p\in\vp$, while $J_2$ commutes or anti-commutes with $\sigma_l$. Thus there must be a summand $V$ within the direct sum decomposition of $V_\lambda$, with $\vn_p(f)=1$ for $f\in V\setminus\{0\}$ any $p\in\vp$, provided $\lambda$ is critical.  

\begin{definition}
    Define $V^+_{\pm,\mp}(1)=\{\alpha f:\alpha\in\CC, f\in V_{\lambda}\setminus\{0\},\sigma_0.f=+f,\sigma_1.f=\pm f, \sigma_2.f=\mp f,\vn_p(f)=1\text{ for any }p\in\vp\}$. Definition of $V^-_{\pm,\mp}(1)$ is the same but with $\sigma_0.f=-f$.
\end{definition}

\begin{lemma}
    For any eigenvalue $\lambda$ of $\Delta$ on $\lb$, there can not exists summand $V^+_{\pm,\mp}(1)$ or $V^-_{\pm,\mp}(1)$ in the decomposition of irreducible $\mathbb{Z}_2^{\oplus 3}$ representations of $V_\lambda$.
\end{lemma}
\begin{proof}
    Suppose there is a $V^+_{\pm,\mp}(1)$ within $V_{\lambda}$ and $f\in V^+_{\pm,\mp}(1)\setminus\{0\}$. Consider the lift $\tilde{f}$ of $f$, which is defined on $\Sigma_\vp$ as an odd function. Write the asymptotic expansion of $\tilde{f}$ near $p_l$ as $\tilde{f}(w_l,\bar{w}_l)=a_lw_l^3+b_l\bar{w}^3_l+\mathcal{O}(r^{5})$ for $1\le l\le 4$.

    Now applying the $\mathbb{Z}_2^{\oplus 3}$-action to $\tilde{f}$. We first deduce from $\sigma_0.\tilde{f}=\tilde{f}$ and $(1)$ of Lemma \ref{Z23trans} that
    \begin{align*}
        (-1)^lb_lw_l^3+(-1)^la_l\bar{w}_l^3+\mathcal{O}(r^{5})
        =\sigma_0.\tilde{f}(w_l,\bar{w}_l)=\tilde{f}(w_l,\bar{w}_l)=a_lw_l^3+b_l\bar{w}_l^3+\mathcal{O}(r^{5})
    \end{align*}
    and hence $a_l=(-1)^lb_l$ for $l=1,2,3,4$.

We also have
    \begin{align*}
        &i(-1)^{l}a_{l+1}w_l^3+ia_{l+1}\bar{w}_l^3+\mathcal{O}(r^{5})
        =\sigma_1.\tilde{f}(w_l,\bar{w}_l)=\pm\tilde{f}(w_l,\bar{w}_l)=\pm a_lw_1^3\pm (-1)^la_l\bar{w}_l^3+\mathcal{O}(r^{5})
    \end{align*}
 for $l=1,3$, and hence $a_{l+1}=\pm ia_l$ $(l=1,3)$. Similarly, from $\sigma.\tilde{f}=\mp\tilde{f}$, we deduce $a_4=\mp ia_1$ and $a_3=\mp(-i)a_2$. We summarize these leading coefficients of $\tilde{f}$ in the following table:
\begin{spacing}{1.2}\[
\begin{array}{|c|c|c|c|c|}
     \hline
         V^+_{\pm,\mp}(1) & p_1 & p_2 & p_3 & p_4\\
         \cline{1-5}
         \tilde{f} & (a_1,-a_1) & (\pm ia_1,\pm ia_1) & (\mp\pm a_1,-\mp\pm a_1) & (\mp ia_1,\mp ia_1)\\ \cline{1-5}
\end{array}\]\end{spacing}
Consider the vector field $J_+=J_1+iJ_2$, where $J_1$ and $J_2$ are defined before \eqref{Jcplxcoor}. Let $z_l=w_l^2$ be the complex coordinate near each $p_l$, then we can calculate $(J_+z_l)_{r=0}$ and $(J_+\bar{z}_l)_{r=0}$ from \eqref{Jcplxcoor} and Lemma \ref{Liealgetrans}. We present the result in the following table:
\begin{spacing}{1.2}
\[\begin{array}{|c|c|c|c|c|}
     \hline
          l & 1 & 2 & 3 & 4\\
         \cline{1-5}
         (J_+z_l)|_{r=0} & -\frac{1}{2}(\cos\theta+1) & -1 & \frac{1}{2}(\cos\theta-1) & 0\\ \cline{1-5}
         (J_+\bar{z}_l)|_{r=0} & \frac{1}{2}(\cos\theta-1) & 0 & -\frac{1}{2}(\cos\theta+1) & -1\\ \cline{1-5}
\end{array}\]\end{spacing}
 Set $\phi_1=J_+f$ and $\phi_2=J_+^2f$ as in Lemma \ref{lemmalocal}, then we conclude $0=\frac{3}{2}(1-\cos^2\theta)a_1^2$, which implies $a_1=0$ for $\theta\ne 0,\pi$. Consequently, $\vn_p(f)>1$, contradicting with our assumption before. Similar argument can be applied to $V^-_{\pm,\mp}(1)$.
\end{proof}
\begin{corollary}
    Suppose $\vp=\{p_1,-p_1,p_2,-p_2\}$ is the cross configuration in $\mathcal{C}_4$, then any $\ZT$ eigenvalue of $\lb$ is non-critical. 
\end{corollary}
\begin{proof}
    Consider the space of critical eigensections $V_\lambda^c=\{f\in V_\lambda: \vn_p(f)\ge 1,\,\forall p\in\vp\}$ in $\sobc$, which contains the space real critical sections by construction. Note that $V_\lambda^c$ is a $\ZT^{\oplus 3}$-subrepresentation of $V_\lambda$, and hence it is a direct sum of irreducible $\ZT^{\oplus 3}$ representations if not trivial. Thus the existence of critical eigensections leads to the existence of summands $V^+_{\pm,\mp}(1)$ or $V^-_{\pm,\mp}(1)$, contradicting with the previous lemma. 
\end{proof}

 Combining this corollary with Proposition \ref{partialanti}, we can conclude:
\begin{theorem}
    If a configuration $\vp\in\mathcal{C}_4$ contains at least one pair of antipodal points, then $\lb^{\CC}$ admits no critical eigensections. 
\end{theorem}

\section{Spectral variation and dihedral groups}\label{newfamily}
In this section, we consider a path of configurations in $\Cn$ preserved by dihedral groups for each $n>1$. Our goal is to demonstrate that within this path, there exist infinitely many critical eigensections.

In the compactified space $\ccn$, the closure of this path has two endpoints: the empty configuration and the antipodal configuration. The corresponding $\ZT$ eigenvalues of these two configurations are known explicitly. As the configuration varies along this path, the $\ZT$ spectrum also varies, resulting in an interesting spectral flow. Additionally, different types of irreducible representations intersect at some points along this path. As before, different types of irreducible representations (of the same vanishing order) are mutually exclusive, and hence critical eigensections appear at those points. The number of intersections of different representations are unbounded. Thus, for each $n>1$, there are infinitely many critical configurations in $\Cn$.
\subsection{The Dihedral Group Action}
Identify $\sph$ with $\overline{\CC}$ by stereographic projection. Take $2n=m+1$, for $0<t<\infty$, we define a path of configurations $\vp(t)\in \Cn$ as $$\vp(t):=\{p_{2n}:=0,p_1(t):=t,p_2(t):=te^{\frac{2\pi}{m}i},\cdots,p_m(t):=te^{\frac{2(m-1)\pi}{m}i}\},$$ where $p_1(t),\cdots,p_m(t)$ are vertices of a polygon. For each configuration point $p_l(t)$, the complex coordinate $z_l:=z_{p_l(t)}$ is given as in Section \ref{4.1}. We define $\vp(0)$ to be the empty configuration and $\vp(\infty)$ the antipodal configuration. When $m=3$ and $n=2$, $\vp(\sqrt{2})$ is the tetrahedral configuration. For general $0<t<+\infty$, there is a dihedral group $D_{2m}$ acts on $\sph$, preserving $\vp(t)$. 
Let $\Sigma_t:=\Sigma_{\vp(t)}$ be the double branched covering surface branched at $\vp(t)$, and we first consider the lifting of the dihedral group action.

Consider the rotation $\tau_m:\sph\to\sph:(\varphi,\theta)\mapsto (\varphi-\frac{2\pi}{m},\theta)$ and reflection $\sigma:\sph\to\sph:(\varphi,\theta)\mapsto(-\varphi,\theta)$. This two isometrics generate a subgroup of $\mathrm{O}(3)$ that isomorphic with the dihedral group $D_{2m}=\langle \tau_m,\sigma|\sigma^2=\mathrm{id},\tau_m^m=\mathrm{id},\tau_m\sigma\tau_m=\sigma\rangle$. 

On the Riemann surface $\Sigma_t$, given that a conformal structure is fixed by the branched covering $\pi:\Sigma_t\to \sph$, one can choose a unique constant curved metric both on $\rst$ and its universal covering. If $n=2$, the universal covering of $\Sigma_t$ is the flat plane $\CC$; if $n>2$, then the universal covering is the Poincar\'{e} disk $\mathbb{D}^2$. After an isometry on the universal covering, we may assume one of the preimage of $p_{2n}$ lies at the origin. Additionally, assume that the complex coordinate on this space induces one near $p_{2n}$ on $\rst$, for which we denote as $w_{2n}$, such that $w^2_{2n}$ is identical to the complex coordinate $z_{2n}$ that defined by stereographic projection near $p_{2n}$ on $\sph$.

The isometries $\tau_m$ and $\sigma$ can be lifted to $\rst$ and its universal covering $\widetilde{\Sigma}_t$, hence we can interpreted these lifts in the complex coordinate $w_{2n}$. The $\tau_m$ has two lifts $\tilde{\tau}_m:w_{2n}\mapsto e^{{2(n-1)\pi i}/{m}}w_{2n}$ and $\ei\tilde{\tau}_m$, where $\ei$ is the canonical involution of $\rst$, which can also be presented as $w_{2n}\mapsto-w_{2n}$ on $\widetilde{\Sigma}_t$. For $\sigma$, we consider the lift $\tilde{\sigma}:w_{2n}\mapsto -\bar{w}_{2n}$, another lift is $\ei\tilde{\sigma}:w_{2n}\mapsto \bar{w}_{2n}$. Note that $n-1$ and $m=2n-1$ are coprime, for $-2(n-1)+m=1$. Consequently, $\tilde{\tau}_m$ generates a cyclic group $\mathbb{Z}_{m}$, acting on $\Sigma_t$ and $\widetilde{\Sigma}_t=\CC$ or $\mathbb{D}^2$.

\begin{lemma}
    The automorphisms $\tilde{\tau}_m$ and $\tilde{\sigma}$ defined on $\rst$ generate a group that isomorphic with $D_{2m}$. 
\end{lemma}
\begin{proof}
    We only need to verify $\tilde{\tau}_m\tilde{\sigma}\tilde{\tau}_m=\tilde{\sigma}$. Over $\CC$ or $\mathbb{D}^2$, we compute
    \begin{align*}
\tilde{\tau}_m\tilde{\sigma}\tilde{\tau}_m(w_{2n})=\tilde{\tau}_m\tilde{\sigma}(e^{\frac{2(n-1)\pi}{m}i}w_{2n})=\tilde{\tau}_m(-e^{\frac{-2(n-1)\pi}{m}i}\bar{w}_{2n})=-e^{\frac{2(n-1)\pi}{m}i}e^{\frac{-2(n-1)\pi}{m}i}\bar{w}_{2n}=\tilde{\sigma}(w_{2n}).
    \end{align*}
\end{proof}

As before, for simplicity in notations, we will continue to denote the lifts $\tilde{\tau}_m$ and $\tilde{\sigma}$ as $\tau_m$ and $\sigma$. 

Next, we need to fix complex coordinates near $p_1(t),\cdots,p_m(t)$ on $\rst$. We shall do this on the universal covering $\widetilde{\Sigma}_t$.  We require: (a). $w_l^2=z_l$ for $l=1,\cdots,m$; (b). suppose $\tau_{m}(p_l)=p_k$ for some $1\le l,k\le m$, then $\tau_m.w_l=w_l\circ \tau_m^{-1}=w_{k}$. These two requirements do not conflict with each other, and once we fix a choice of complex coordinate near $p_1(t)$, requirement (b) fixes the reminders.

As an illustration, let's consider Figure \ref{m5pcrdisk}, which presents the case when $m=5$.  As before, each longer red arrow represents the positive direction of the real axis of the complex coordinate near its original point, while each shorter red arrow represents the positive direction of imaginary axis. We abbreviate $p_l(t)$ to $p_l$ in this figure.

\begin{figure}[!h]
    \centering

\tikzset{every picture/.style={line width=0.75pt}} %set default line width to 0.75pt        

\begin{tikzpicture}[x=0.75pt,y=0.75pt,yscale=-1,xscale=1]

%Shape: Arc [id:dp2312366114966291] 
\draw  [draw opacity=0] (334.78,225.78) .. controls (319.71,212.58) and (311.89,191.93) .. (316.04,170.86) .. controls (321.23,144.51) and (343.5,125.87) .. (369.07,124.17) -- (372.97,182.08) -- cycle ; \draw   (334.78,225.78) .. controls (319.71,212.58) and (311.89,191.93) .. (316.04,170.86) .. controls (321.23,144.51) and (343.5,125.87) .. (369.07,124.17) ;  
%Shape: Arc [id:dp6951509979945476] 
\draw  [draw opacity=0] (367.65,152.33) .. controls (347.68,150.75) and (329.07,138.85) .. (319.82,119.48) .. controls (308.24,95.25) and (314.96,67) .. (334.46,50.35) -- (372.17,94.47) -- cycle ; \draw   (367.65,152.33) .. controls (347.68,150.75) and (329.07,138.85) .. (319.82,119.48) .. controls (308.24,95.25) and (314.96,67) .. (334.46,50.35) ;  
%Shape: Arc [id:dp0236711348340648] 
\draw  [draw opacity=0] (264.75,266.8) .. controls (259.43,246.37) and (265.5,223.75) .. (282.36,208.77) .. controls (302.16,191.16) and (330.67,189.59) .. (351.98,203.12) -- (320.9,252.14) -- cycle ; \draw   (264.75,266.8) .. controls (259.43,246.37) and (265.5,223.75) .. (282.36,208.77) .. controls (302.16,191.16) and (330.67,189.59) .. (351.98,203.12) ;  
%Shape: Arc [id:dp16509925293068584] 
\draw  [draw opacity=0] (125.75,202.47) .. controls (142.91,192.14) and (164.96,190.98) .. (183.77,201.36) .. controls (207.27,214.33) and (218.24,241.21) .. (212.07,266.09) -- (155.73,252.16) -- cycle ; \draw   (125.75,202.47) .. controls (142.91,192.14) and (164.96,190.98) .. (183.77,201.36) .. controls (207.27,214.33) and (218.24,241.21) .. (212.07,266.09) ;  
%Shape: Arc [id:dp052277910432254604] 
\draw  [draw opacity=0] (109.73,124.34) .. controls (129.69,126.08) and (148.2,138.13) .. (157.29,157.58) .. controls (168.67,181.91) and (161.71,210.1) .. (142.09,226.58) -- (104.73,182.16) -- cycle ; \draw   (109.73,124.34) .. controls (129.69,126.08) and (148.2,138.13) .. (157.29,157.58) .. controls (168.67,181.91) and (161.71,210.1) .. (142.09,226.58) ;  
%Straight Lines [id:da6506361713854618] 
\draw [color={rgb, 255:red, 208; green, 2; blue, 27 }  ,draw opacity=1 ][line width=1.5]    (238.73,138.73) -- (238.49,110.33) ;
\draw [shift={(238.47,107.33)}, rotate = 89.51] [color={rgb, 255:red, 208; green, 2; blue, 27 }  ,draw opacity=1 ][line width=1.5]    (11.37,-3.42) .. controls (7.23,-1.45) and (3.44,-0.31) .. (0,0) .. controls (3.44,0.31) and (7.23,1.45) .. (11.37,3.42)   ;
%Straight Lines [id:da5406991970612531] 
\draw    (279.47,15.33) -- (198.47,262.33) ;
%Straight Lines [id:da2961457179822302] 
\draw    (198.47,15.33) -- (278.47,262.33) ;
%Straight Lines [id:da04416499770643956] 
\draw    (133.47,214.33) -- (343.47,62.33) ;
%Straight Lines [id:da6308753328460304] 
\draw    (368.97,138.73) -- (108.5,138.73) ;
%Straight Lines [id:da6003091098381588] 
\draw    (343.47,214.33) -- (134.47,62.33) ;
%Shape: Circle [id:dp42701095666861266] 
\draw  [fill={rgb, 255:red, 255; green, 255; blue, 255 }  ,fill opacity=1 ] (156.98,82.46) .. controls (156.98,80.09) and (158.91,78.16) .. (161.28,78.16) .. controls (163.66,78.16) and (165.58,80.09) .. (165.58,82.46) .. controls (165.58,84.84) and (163.66,86.76) .. (161.28,86.76) .. controls (158.91,86.76) and (156.98,84.84) .. (156.98,82.46) -- cycle ;
%Shape: Circle [id:dp6853094907656005] 
\draw  [fill={rgb, 255:red, 255; green, 255; blue, 255 }  ,fill opacity=1 ] (156.98,195) .. controls (156.98,192.63) and (158.91,190.7) .. (161.28,190.7) .. controls (163.66,190.7) and (165.58,192.63) .. (165.58,195) .. controls (165.58,197.38) and (163.66,199.3) .. (161.28,199.3) .. controls (158.91,199.3) and (156.98,197.38) .. (156.98,195) -- cycle ;
%Straight Lines [id:da21560082598030528] 
\draw [color={rgb, 255:red, 208; green, 2; blue, 27 }  ,draw opacity=1 ][line width=1.5]    (238.73,138.73) -- (284.47,139.3) ;
\draw [shift={(287.47,139.33)}, rotate = 180.71] [color={rgb, 255:red, 208; green, 2; blue, 27 }  ,draw opacity=1 ][line width=1.5]    (15.63,-4.7) .. controls (9.94,-1.99) and (4.73,-0.43) .. (0,0) .. controls (4.73,0.43) and (9.94,1.99) .. (15.63,4.7)   ;
%Shape: Circle [id:dp15028062652900398] 
\draw  [fill={rgb, 255:red, 255; green, 255; blue, 255 }  ,fill opacity=1 ] (234.43,138.73) .. controls (234.43,136.36) and (236.36,134.43) .. (238.73,134.43) .. controls (241.11,134.43) and (243.03,136.36) .. (243.03,138.73) .. controls (243.03,141.11) and (241.11,143.03) .. (238.73,143.03) .. controls (236.36,143.03) and (234.43,141.11) .. (234.43,138.73) -- cycle ;
%Shape: Circle [id:dp7526746678628049] 
\draw   (108.5,138.73) .. controls (108.5,66.81) and (166.81,8.5) .. (238.73,8.5) .. controls (310.66,8.5) and (368.97,66.81) .. (368.97,138.73) .. controls (368.97,210.66) and (310.66,268.97) .. (238.73,268.97) .. controls (166.81,268.97) and (108.5,210.66) .. (108.5,138.73) -- cycle ;
%Shape: Arc [id:dp29631080737611004] 
\draw  [draw opacity=0] (184.35,258.18) .. controls (191.84,239.6) and (208.76,225.41) .. (230.02,222.38) .. controls (256.6,218.6) and (281.54,233.47) .. (291.58,257.06) -- (238.19,279.83) -- cycle ; \draw   (184.35,258.18) .. controls (191.84,239.6) and (208.76,225.41) .. (230.02,222.38) .. controls (256.6,218.6) and (281.54,233.47) .. (291.58,257.06) ;  
%Shape: Circle [id:dp955771845809579] 
\draw  [fill={rgb, 255:red, 255; green, 255; blue, 255 }  ,fill opacity=1 ] (266.21,234.51) .. controls (263.96,233.75) and (262.76,231.31) .. (263.53,229.06) .. controls (264.29,226.81) and (266.74,225.61) .. (268.98,226.37) .. controls (271.23,227.14) and (272.43,229.58) .. (271.67,231.83) .. controls (270.9,234.08) and (268.46,235.28) .. (266.21,234.51) -- cycle ;
%Shape: Arc [id:dp08256982452485162] 
\draw  [draw opacity=0] (292.65,20.34) .. controls (285.07,38.89) and (268.09,53) .. (246.81,55.93) .. controls (220.21,59.59) and (195.35,44.61) .. (185.41,20.98) -- (238.9,-1.55) -- cycle ; \draw   (292.65,20.34) .. controls (285.07,38.89) and (268.09,53) .. (246.81,55.93) .. controls (220.21,59.59) and (195.35,44.61) .. (185.41,20.98) ;  
%Shape: Arc [id:dp6098254707225437] 
\draw  [draw opacity=0] (352.21,73.38) .. controls (335.17,83.92) and (313.13,85.34) .. (294.21,75.18) .. controls (270.55,62.49) and (259.26,35.74) .. (265.13,10.79) -- (321.64,24.05) -- cycle ; \draw   (352.21,73.38) .. controls (335.17,83.92) and (313.13,85.34) .. (294.21,75.18) .. controls (270.55,62.49) and (259.26,35.74) .. (265.13,10.79) ;  
%Shape: Circle [id:dp19796831038572527] 
\draw  [fill={rgb, 255:red, 255; green, 255; blue, 255 }  ,fill opacity=1 ] (311.88,82.46) .. controls (311.88,80.09) and (313.81,78.16) .. (316.18,78.16) .. controls (318.56,78.16) and (320.48,80.09) .. (320.48,82.46) .. controls (320.48,84.84) and (318.56,86.76) .. (316.18,86.76) .. controls (313.81,86.76) and (311.88,84.84) .. (311.88,82.46) -- cycle ;
%Shape: Arc [id:dp5194350725647876] 
\draw  [draw opacity=0] (143.13,50.99) .. controls (158.35,64.02) and (166.39,84.59) .. (162.47,105.7) .. controls (157.56,132.1) and (135.5,150.98) .. (109.94,152.96) -- (105.42,95.1) -- cycle ; \draw   (143.13,50.99) .. controls (158.35,64.02) and (166.39,84.59) .. (162.47,105.7) .. controls (157.56,132.1) and (135.5,150.98) .. (109.94,152.96) ;  
%Shape: Arc [id:dp21809194834130818] 
\draw  [draw opacity=0] (212.24,11.18) .. controls (216.87,30.67) and (211.27,52.03) .. (195.68,66.79) .. controls (176.17,85.25) and (147.23,87.53) .. (125.4,74.09) -- (155.79,24.65) -- cycle ; \draw   (212.24,11.18) .. controls (216.87,30.67) and (211.27,52.03) .. (195.68,66.79) .. controls (176.17,85.25) and (147.23,87.53) .. (125.4,74.09) ;  
%Shape: Circle [id:dp6297855250447408] 
\draw  [fill={rgb, 255:red, 255; green, 255; blue, 255 }  ,fill opacity=1 ] (204.85,47.69) .. controls (204.85,45.31) and (206.78,43.39) .. (209.15,43.39) .. controls (211.52,43.39) and (213.45,45.31) .. (213.45,47.69) .. controls (213.45,50.06) and (211.52,51.99) .. (209.15,51.99) .. controls (206.78,51.99) and (204.85,50.06) .. (204.85,47.69) -- cycle ;
%Straight Lines [id:da5299275286518457] 
\draw [color={rgb, 255:red, 208; green, 2; blue, 27 }  ,draw opacity=1 ][line width=1.5]    (334.47,138.73) -- (334.47,113.33) ;
\draw [shift={(334.47,110.33)}, rotate = 90] [color={rgb, 255:red, 208; green, 2; blue, 27 }  ,draw opacity=1 ][line width=1.5]    (15.63,-4.7) .. controls (9.94,-1.99) and (4.73,-0.43) .. (0,0) .. controls (4.73,0.43) and (9.94,1.99) .. (15.63,4.7)   ;
%Straight Lines [id:da05969966187951359] 
\draw [color={rgb, 255:red, 208; green, 2; blue, 27 }  ,draw opacity=1 ][line width=1.5]    (334.47,138.73) -- (318.47,138.4) ;
\draw [shift={(315.47,138.33)}, rotate = 1.21] [color={rgb, 255:red, 208; green, 2; blue, 27 }  ,draw opacity=1 ][line width=1.5]    (8.53,-2.57) .. controls (5.42,-1.09) and (2.58,-0.23) .. (0,0) .. controls (2.58,0.23) and (5.42,1.09) .. (8.53,2.57)   ;
%Shape: Circle [id:dp6453766973610437] 
\draw  [fill={rgb, 255:red, 255; green, 255; blue, 255 }  ,fill opacity=1 ] (330.17,138.73) .. controls (330.17,136.36) and (332.09,134.43) .. (334.47,134.43) .. controls (336.84,134.43) and (338.77,136.36) .. (338.77,138.73) .. controls (338.77,141.11) and (336.84,143.03) .. (334.47,143.03) .. controls (332.09,143.03) and (330.17,141.11) .. (330.17,138.73) -- cycle ;
%Straight Lines [id:da1501530380346383] 
\draw [color={rgb, 255:red, 208; green, 2; blue, 27 }  ,draw opacity=1 ][line width=1.5]    (143,138.73) -- (142.53,114.33) ;
\draw [shift={(142.47,111.33)}, rotate = 88.88] [color={rgb, 255:red, 208; green, 2; blue, 27 }  ,draw opacity=1 ][line width=1.5]    (15.63,-4.7) .. controls (9.94,-1.99) and (4.73,-0.43) .. (0,0) .. controls (4.73,0.43) and (9.94,1.99) .. (15.63,4.7)   ;
%Straight Lines [id:da9705287821295159] 
\draw [color={rgb, 255:red, 208; green, 2; blue, 27 }  ,draw opacity=1 ][line width=1.5]    (143,138.73) -- (126.47,138.39) ;
\draw [shift={(123.47,138.33)}, rotate = 1.17] [color={rgb, 255:red, 208; green, 2; blue, 27 }  ,draw opacity=1 ][line width=1.5]    (8.53,-2.57) .. controls (5.42,-1.09) and (2.58,-0.23) .. (0,0) .. controls (2.58,0.23) and (5.42,1.09) .. (8.53,2.57)   ;
%Shape: Circle [id:dp38030718622043724] 
\draw  [fill={rgb, 255:red, 255; green, 255; blue, 255 }  ,fill opacity=1 ] (138.7,138.73) .. controls (138.7,136.36) and (140.63,134.43) .. (143,134.43) .. controls (145.37,134.43) and (147.3,136.36) .. (147.3,138.73) .. controls (147.3,141.11) and (145.37,143.03) .. (143,143.03) .. controls (140.63,143.03) and (138.7,141.11) .. (138.7,138.73) -- cycle ;
%Straight Lines [id:da18820903958674262] 
\draw [color={rgb, 255:red, 208; green, 2; blue, 27 }  ,draw opacity=1 ][line width=1.5]    (161.41,82.37) -- (144.38,103.02) ;
\draw [shift={(142.47,105.33)}, rotate = 309.52] [color={rgb, 255:red, 208; green, 2; blue, 27 }  ,draw opacity=1 ][line width=1.5]    (15.63,-4.7) .. controls (9.94,-1.99) and (4.73,-0.43) .. (0,0) .. controls (4.73,0.43) and (9.94,1.99) .. (15.63,4.7)   ;
%Straight Lines [id:da11053923269639054] 
\draw [color={rgb, 255:red, 208; green, 2; blue, 27 }  ,draw opacity=1 ][line width=1.5]    (161.41,82.37) -- (174.99,91.64) ;
\draw [shift={(177.47,93.33)}, rotate = 214.34] [color={rgb, 255:red, 208; green, 2; blue, 27 }  ,draw opacity=1 ][line width=1.5]    (8.53,-2.57) .. controls (5.42,-1.09) and (2.58,-0.23) .. (0,0) .. controls (2.58,0.23) and (5.42,1.09) .. (8.53,2.57)   ;
%Shape: Circle [id:dp22976347242238915] 
\draw  [fill={rgb, 255:red, 255; green, 255; blue, 255 }  ,fill opacity=1 ] (272.08,45.01) .. controls (273.5,46.91) and (273.11,49.6) .. (271.2,51.02) .. controls (269.3,52.44) and (266.61,52.05) .. (265.19,50.15) .. controls (263.76,48.25) and (264.16,45.55) .. (266.06,44.13) .. controls (267.96,42.71) and (270.66,43.1) .. (272.08,45.01) -- cycle ;
%Shape: Circle [id:dp2818247354161436] 
\draw  [fill={rgb, 255:red, 255; green, 255; blue, 255 }  ,fill opacity=1 ] (204.85,229.78) .. controls (204.85,227.41) and (206.78,225.48) .. (209.15,225.48) .. controls (211.52,225.48) and (213.45,227.41) .. (213.45,229.78) .. controls (213.45,232.16) and (211.52,234.08) .. (209.15,234.08) .. controls (206.78,234.08) and (204.85,232.16) .. (204.85,229.78) -- cycle ;
%Shape: Circle [id:dp13368064960708215] 
\draw  [fill={rgb, 255:red, 255; green, 255; blue, 255 }  ,fill opacity=1 ] (160.06,78.28) .. controls (162.31,77.54) and (164.75,78.76) .. (165.5,81.01) .. controls (166.24,83.27) and (165.02,85.7) .. (162.77,86.45) .. controls (160.52,87.19) and (158.08,85.97) .. (157.33,83.72) .. controls (156.59,81.47) and (157.81,79.03) .. (160.06,78.28) -- cycle ;
%Straight Lines [id:da9835220098857353] 
\draw [color={rgb, 255:red, 208; green, 2; blue, 27 }  ,draw opacity=1 ][line width=1.5]    (316.18,195) -- (299.14,215.66) ;
\draw [shift={(297.24,217.97)}, rotate = 309.52] [color={rgb, 255:red, 208; green, 2; blue, 27 }  ,draw opacity=1 ][line width=1.5]    (15.63,-4.7) .. controls (9.94,-1.99) and (4.73,-0.43) .. (0,0) .. controls (4.73,0.43) and (9.94,1.99) .. (15.63,4.7)   ;
%Straight Lines [id:da29236271109538015] 
\draw [color={rgb, 255:red, 208; green, 2; blue, 27 }  ,draw opacity=1 ][line width=1.5]    (316.18,195) -- (329.76,204.28) ;
\draw [shift={(332.24,205.97)}, rotate = 214.34] [color={rgb, 255:red, 208; green, 2; blue, 27 }  ,draw opacity=1 ][line width=1.5]    (8.53,-2.57) .. controls (5.42,-1.09) and (2.58,-0.23) .. (0,0) .. controls (2.58,0.23) and (5.42,1.09) .. (8.53,2.57)   ;
%Shape: Circle [id:dp28347830634982096] 
\draw  [fill={rgb, 255:red, 255; green, 255; blue, 255 }  ,fill opacity=1 ] (311.88,195) .. controls (311.88,192.63) and (313.81,190.7) .. (316.18,190.7) .. controls (318.56,190.7) and (320.48,192.63) .. (320.48,195) .. controls (320.48,197.38) and (318.56,199.3) .. (316.18,199.3) .. controls (313.81,199.3) and (311.88,197.38) .. (311.88,195) -- cycle ;
%Shape: Arc [id:dp24386310464063388] 
\draw  [draw opacity=0] (193.62,100.73) .. controls (204.49,87.8) and (220.79,79.59) .. (239.01,79.6) .. controls (269.98,79.62) and (295.39,103.41) .. (297.98,133.71) -- (238.97,138.83) -- cycle ; \draw  [color={rgb, 255:red, 74; green, 144; blue, 226 }  ,draw opacity=1 ] (193.62,100.73) .. controls (204.49,87.8) and (220.79,79.59) .. (239.01,79.6) .. controls (269.98,79.62) and (295.39,103.41) .. (297.98,133.71) ;  
%Straight Lines [id:da5906663467597062] 
\draw [color={rgb, 255:red, 74; green, 144; blue, 226 }  ,draw opacity=1 ]   (193.62,100.73) -- (199.47,87.33) ;
%Straight Lines [id:da42373076322397396] 
\draw [color={rgb, 255:red, 74; green, 144; blue, 226 }  ,draw opacity=1 ]   (193.62,100.73) -- (205.47,96.33) ;

% Text Node
\draw (205,127) node [anchor=north west][inner sep=0.75pt]    {$p_{2n}$};
% Text Node
\draw (336,146) node [anchor=north west][inner sep=0.75pt]    {$p_{1}$};
% Text Node
\draw (133.5,146) node [anchor=north west][inner sep=0.75pt]    {$p_{1}$};
% Text Node
\draw (320,175) node [anchor=north west][inner sep=0.75pt]    {$p_{m}$};
% Text Node
\draw (158,59) node [anchor=north west][inner sep=0.75pt]    {$p_{m}$};
% Text Node
\draw (318.18,82) node [anchor=north west][inner sep=0.75pt]    {$p_{2}$};
% Text Node
\draw (140.28,174) node [anchor=north west][inner sep=0.75pt]    {$p_{2}$};
% Text Node
\draw (280.18,37.86) node [anchor=north west][inner sep=0.75pt]    {$p_{3}$};
% Text Node
\draw (216.18,226.86) node [anchor=north west][inner sep=0.75pt]    {$p_{3}$};
% Text Node
\draw (217.18,33) node [anchor=north west][inner sep=0.75pt]    {$p_{4}$};
% Text Node
\draw (247.18,225.86) node [anchor=north west][inner sep=0.75pt]    {$p_{4}$};
% Text Node
\draw (259,87.4) node [anchor=north west][inner sep=0.75pt]  [color={rgb, 255:red, 74; green, 144; blue, 226 }  ,opacity=1 ]  {$\tau _{m}$};
\end{tikzpicture}

    \caption{$\vp(t)$ and complex coordinates on the Poincar\'{e} disk $\mathbb{D}^2\to\rst$.}
    \label{m5pcrdisk}
\end{figure}

With complex coordinate $w_1$ and $w_{2n}$ being given as in the Figure \ref{m5pcrdisk}, we can define $w_l(l=2,\cdots,m)$ as above and this yields:
\begin{lemma}
    \begin{enumerate}[label=(\roman*)]
        \item Suppose $\tau_{m}(p_l(t))=p_k(t)$ for some $l,k=1,\cdots,m$, we have $\tau_m. w_l=w_k$. 
        \item $(\tau^{-2l}_m\sigma\tau^{2l}_m). w_{m-l+1}=\bar{w}_{m-l+1}$ for $l=1,\cdots,m$, with $w_{m+1}=w_1$.
    \end{enumerate}
\end{lemma}

Given the lifting $D_{2m}$-action on $\rst$, by Lemma \ref{commuteswithinv}, we can define an infinite dimensional $D_{2m}$-representation on the space $\osb(\rst)$ and hence $\mathcal{L}_{1}^2(\lbct)$. Consider a $\ZT$ eigenvalue $\lambda$ of $\lbct$, the corresponding eigenspace $V_\lambda\subset \mathcal{L}_{1}^2(\lbct)$ can be interpreted as a finite dimensional $D_{2m}$-representation.

Recall that irreducible finite dimensional $D_{2m}$-representations can be listed as follows:
\begin{enumerate}[label=(\roman*)]
    \item $U$, the trivial representation, on where $\tau_m.v=\sigma.v=v$.
    \item $U'$, the alternating representation, on where $\tau_m.v=-\sigma.v=v$.
    \item $W$, the standard representation, which is two dimensional, with a basis $\{v,\sigma.v\}$, such that $\tau_m.v=e^{{2\pi i}/{m}}v$ while $\tau_m.(\sigma.v)=e^{-{2\pi i}/{m}}\sigma.v$.
\end{enumerate}

Let $V_{\lam}$ be the eigenspaces, recall that we could define a filtration of $V_{\lam}$ by orders:
\begin{definition}
   Let $V_\lambda^j=\{f\in V_\lambda:\ord(f)\ge j\}$. By Lemma \ref{upperorder}, these spaces define a finite filtration $V_\lambda=V^0_\lambda\supset V^1_\lambda\supset\cdots\supset V_\lambda^{n-1}\supset 0$. Define $W_\lambda^{j-1}$ as the orthogonal (with respect to the $L^2$ inner product) complement of $V^j_\lambda$ in $V^{j-1}_\lambda$, for $j\ge 1$.
\end{definition}

In particular, we have $V_{\lam}^{j-1}=V_{\lam}^{j}\oplus W_{\lam}^{j-1}$.

\iffalse
\begin{lemma}\label{twofilthesame}
    $V_\lambda^j=\{f\in V_\lambda: \min_{1\le l\le m}\vn_{p_{l}(t)}(f)\ge j\}$.
\end{lemma}
\begin{proof}
    By definition, each $V^j_\lambda$ is a $D_{2m}$-subrepresentation of $V_\lambda$. We may assume $V^j_\lambda\cong U^{\oplus c}\oplus (U')^{\oplus d}\oplus W^{\oplus e}$ for some nonnegative integers $c,d$ and $e$. Suppose $f\in V^j_\lambda$, then we can decompose it into $f=f_1+f_2+f_3$, for $f_1\in U^{\oplus c}$, $f_2\in (U')^{\oplus d}$ and $f_3\in W^{\oplus e}$. By definition, $\vn_{p_1(t)}(f_s)=\cdots=\vn_{p_m(t)}(f_s)$ for $s=1,2,3$. It follows that $\vn_{p_1(t)}(f)=\cdots=\vn_{p_m(t)}(f)$. Suppose $V_\lambda^j\ne \{f\in V_\lambda: \min_{1\le l\le m}\vn_{p_{l}(t)}(f)\ge j\}$, then we can choose $f\in V_\lambda$ such that $\vn_{p_l(t)}(f)>\vn_{p_{2n}}(f)$ for $1\le l\le m$. However, this contradicts with Lemma \ref{morethanonepointleastorder}.
\end{proof}
\fi

\begin{lemma}\label{ldcefofonedim}
    Suppose $f$ is an eigensection of $\mathcal{I}_{\vp(t)}^\CC$, such that $\CC\{f\}$ is isomorphic with $U$ (resp. $U'$). Then all of the leading coefficients of $f$ at $p_l(t)~(1\le l\le m)$ are $(a,a)$ (resp. $(a,-a)$) for some $a\ne 0$. 
\end{lemma}
\begin{proof}
    One can deduce this from $\sigma.f=f$ (resp. $\sigma.f=-f$), $\tau_m.w_j=w_k$ for $j,k$ such that $\tau_m(p_j(t))=p_k(t)$, and $\sigma.w_1=\bar{w}_1$ directly.
\end{proof}
As mentioned before \eqref{mentionsign}, given that we have fixed the choice of complex coordinate at $p_l(t)$ $(l=1,\cdots,m)$, there is no ambiguity in the signs of the leading coefficients $(a,\pm a)$. Next, we discuss the irreducible representation decomposition of $V_\lam$.

\begin{lemma}\label{symmemulres}
    There is an orthogonal decomposition of $V_\lambda=\oplus_{j=0}^{n-1}W_\lambda^j$, where each summand $W_\lambda^j$ is a $D_{2m}$-subrepresentation of $V_\lambda$. For each $W_\lambda^j$, there exists a further orthogonal decomposition of irreducible $D_{2m}$-representations, denoted as $W_\lambda^j\cong U^{\oplus c_j}\oplus (U')^{\oplus d_j}\oplus W_\lambda^{\oplus e_j}$, where $c_j=c_j(t)$, $d_j=d_j(t)$, and $e_j=e_j(t)$ are integers dependent on $t$. Moreover, it holds that $c_j\ge d_{j+1}$, $d_j\ge c_{j+1}$, $c_j+d_j\le 1$ and $c_j+c_{j+1},d_j+d_{j+1}\le 1$. Consequently, $\sum_{j=0}^{n-1}c_j+d_j\le n$ and $\sum_{j=0}^{n-1}c_j,\sum_{j=0}^{n-1}d_j\le \frac{n}{2}+1$. 
\end{lemma}
\begin{proof}
      Given $f\in W^j$, then by definition, $\lan f,h\ran_{L^2}=0$ for any $h\in V^{j+1}_\lambda$. From the fact that $D_{2m}$ acts on $\rst$ isometrically, we conclude $\lan g.f,g.h\ran_{L^2}=0$ for any $g\in D_{2m}$. Consequently, $D_{2m}$ preserves $W^j_{\lambda}$, i.e. each $W^j_\lambda$ is a $D_{2m}$-subrepresentation of $V_\lambda$. Therefore, there is an orthogonal decomposition of the irreducible $D_{2m}$-representation of $W^j_\lambda$.

      We claim that for an eigensection $f\in W^j_\lambda$ such that $\CC\{f\}\cong U$ (resp. $\CC\{f\}\cong U'$), we have $\CC\{J_3^{p_{2n}}f\}\cong U'$ (resp. $\CC\{f\}\cong U$). To see this, note that $\tau_{m}.(J_3^{p_{2n}}f)=((\tau_m^{-1})_{\ast} J_3^{p_{2n}})(\tau_m.f)=J_3^{p_{2n}}f$ while $\sigma.(J_3^{p_{2n}}f)=((\sigma^{-1})_{\ast} J_3^{p_{2n}})(\sigma.f)=-J_3^{p_{2n}}(\sigma.f)$.
      
      It follows that $c_{j}\ge d_{j+1}$ and $d_{j}\ge c_{j+1}$. As a result, $c_{j}+d_{j}\ge c_{j+1}+d_{j+1}$. To show that $c_j+d_j\le 1$, it suffices to show that $c_0+d_0\le 1$.

      Consider two eigensections $f_1,f_2\in W^0_{\lambda}$, spanning two copies of $U$ (or $U'$), with leading coefficients $(a_1,a_1)$ and $(a_2,a_2)$ (resp.  $(a_1,-a_1)$ and $(a_2,-a_2)$) at $p_1(t)$. Then consider a section $f_3:=a_2f_1-a_1f_2$. The vanishing order $\vn_{p_l}(f_3)\ge 1$ for $l=1,\cdots,m$, and hence $f_3\in V_\lambda^1$. However, $f_1,f_2\in W^0_{\lambda}$, as a consequence, $a_2f_1-a_1f_2=f_3=0$. Therefore, $c_0,d_0\le 1$.

      Suppose $c_0=d_0=1$. Then we can select $f_+$ and $f_-$ from $W^0_\lambda$, such that $\tau_m. f_\pm=f_\pm$ while $\sigma.f_\pm=\pm f_{\pm}$. Now, consider two sections $\phi_1=f_+$ and $\phi_2=J_3^{p_{2n}}f_-$. We can assume that the leading coefficients of $f_\pm$ at $p_1(t),\cdots, p_m(t)$ are all $(a,\pm a)$ for some $a\in\CC^\ast$. Suppose $p_l(t)$ has spherical coordinate $(\varphi_l,\theta(t))$ on $\sph$ for each $1\le l\le m$, where $\varphi_l=2(l-1)\pi/m$. The leading coefficients of $\phi_2$ at $p_l(t)$ are $(\frac12\sin\theta(t) a,\frac12\sin\theta(t)a)$. By Lemma \ref{lemmalocal}, we have $0=\sum_{l=1}^m \sin\theta(t)a^2=m\sin\theta(t)a^2$. Thus $a=0$ and $f_\pm\notin W^0_\lambda$, leading to a contradiction.

      Suppose $c_j+c_{j+1}\ge 2$, then $c_j,c_{j+1}\ge 1$. By applying the operator $J_3^{p_{2n}}$, we deduce that $d_j\ge 1$ and hence $c_j+d_j\ge 2$, making a contradiction. Similarly, $d_j+d_{j+1}\le 1$.

      Finally, if $\sum_{j=0}^{n-1}c_j>\frac{n}{2}+1$, then for some $0\le j<n-1$, $c_j,c_{j+1}\ge 1$, which cannot be the case. A similar argument can be applied to $\sum_{j=0}^{n-1}d_j$.
\end{proof}

\begin{lemma}\label{UU'ordb}
Consider an irreducible $D_{2m}$-representation $U\,(\text{resp. }U')\subset W^j_\lam$ for some $0\le j\le n-1$, then we can find a real eigensection $f\in U$, such that $U\,(\text{resp. }U')=\CC\{f\}$, and $\vn_{p_l(t)}(f)=j$ for any $1\le l\le m$. Moreover, the leading coefficients of $f$ at $p_l(t)$ are $(a,a)$ (resp. $(ia,-ia)$ ) for some $a\in\RR\setminus\{0\}$ for any $1\le l\le m$.
\end{lemma}
\begin{proof}
    Consider the involution $W^j_\lam\to W^j_\lam:h\mapsto \bar{h}$. According to Lemma \ref{symmemulres}, $c_j\,(\text{resp. }d_j)\,\le 1$, hence this map restricts to an involution on $U\,(\text{resp. }U')$. Thus we may find a real section $f\in U\,(\text{resp. }U')$, such that $U\,(\text{resp. }U')=\CC\{f\}$. And since $\tau_m.f=f$, all $\vn_{p_l(t)}(f)$ are the same. Suppose $\vn_{p_l(t)}(f)>j$, then we deduce a contradiction from Lemma \ref{morethanonepointleastorder}.

    Since $f$ is real, it holds that the leading coefficients of $f$ at $p_l(t)$, say $(a,b)$, satisfies $a=\bar{b}$. Then the last assertion follows from Lemma \ref{ldcefofonedim}.
\end{proof}

Lemma \ref{symmemulres} yields the following key observation:
\begin{corollary}\label{crossingcri}
    Suppose there are two summands of irreducible $D_{2m}$-representation of $V_\lambda$, each of which is isomorphic with either $U$ or $U'$, then $\lambda$ is critical. 
\end{corollary}
\begin{proof}
    By assumption, either $c_0+d_1=2$ or $c_1+d_0=2$ in Lemma \ref{symmemulres}. It follows that $V^1_\lam\subset V^c_\lam$ has positive dimension, and hence $\lam$ is critical.
\end{proof}

\subsection{The deformation of $D_{2m}$-representations}
By Theorem \ref{crossingcri}, the existence of a critical configuration is equivalent to the existence of two summands of specific irreducible representations. In this subsection, we will use a spectral flow to prove this existence.

\begin{definition}
    Consider the orthogonal decomposition $L^2\big(\lbct\big) = \widehat{\oplus}_{k=1}^{\infty} V^t_k$, where $V^t_k$ is the eigenspace of $\Delta$ with eigenvalue $\lambda_k(\vp(t))$. Regarding $V^t_k$ as a $D_{2m}$-representation, let's suppose $V^t_k = U^t_k \oplus (U')^t_k \oplus W^t_k$, where $U^t_k := U^{\oplus c_k}$, $(U')^t_k := (U')^{\oplus d_k}$, and $W^t_k := W^{\oplus e_k}$ are subrepresentations of $V^t_k$. We also define $H^t_+:=\widehat{\oplus}_{k=1}^{\infty}U^t_k,\;H^t_-:=\widehat{\oplus}_{k=1}^{\infty}(U')^t_k$, and $H^t_2:=\widehat{\oplus}_{k=1}^{\infty}W^t_k$.
\end{definition}
We should note that the constants $c_k$, $d_k$ and $e_k$ depends on $t$ and might jump as $t$ varies. 

Since $D_{2m}$ commutes with $\Delta$, $\Delta: H^t_\pm\to H^t_\pm$ and $\Delta: H^t_2\to H^t_2$, consequently, we can decompose the spectrum of $\Delta$ on $\mathcal{L}_{1}^2(\lbct)$ into three parts, namely $\mathrm{Spec}(\Delta)=\mathfrak{S}^t_+\cup\mathfrak{S}^t_-\cup\mathfrak{S}^t_2$, where $\mathfrak{S}^t_\pm=\mathrm{Spec}(\Delta|_{H^t_\pm})$ and $\mathfrak{S}^t_2=\mathrm{Spec}(\Delta|_{H^t_2})$. On each of the spaces $H^t_\pm$ and $H^t_2$, eigenvalues of the restriction of $\Delta$ can still be defined by min-max characterization.

Fix $t_0$, and consider another $t\in(0,+\infty)$, the diffeomorphism $\Phi_t:\overline{\CC}\to\overline{\CC}:z\mapsto \frac{t}{t_0}z$ commutes with the $D_{2m}$-action on $\sph$, satisfying $\Phi_t(\vp(t_0))=\vp(t)$. Under the pullback metric $\Phi_t^\ast\mathrm{d}s^2$, the $D_{2m}$-action remains isometric. Such a diffeomorphism identifies $L^2\big(\lbct\big)$ and $\mathcal{L}_{1}^2\big(\lbct\big)$ with 
$L^2\big(\mathcal{I}_{\vp(t_0)}^{\CC}\big)$ and $\mathcal{L}_{1}^2\big(\mathcal{I}_{\vp(t_0)}^{\CC}\big)$ respectively, by sending $f$ to $\Phi_t^\ast f:=f\circ\Phi_t$. And the Laplacian defined on $\lbct$ induces an operator $\Delta_t$ on $\mathcal{I}_{\vp(t_0)}^{\CC}$, with which the previously lifted $D_{2m}$-action still commutes. 

\begin{lemma}\label{fixedspace}
    $\Phi_t^\ast H^t_\pm\equiv H^{t_0}_\pm$ and $\Phi_t^\ast H^t_2\equiv H^{t_0}_2$.
\end{lemma}
\begin{proof}
    The operator $\Phi_t^\ast: L^2\big(\lbct\big)\to L^2\big(\mathcal{I}_{\vp(t_0)}^{\CC}\big)$ is bounded. Suppose $f\in H_{+}^t$, then $f=\sum_{k=1}^{\infty}f_k$, for $f_k\in U^t_{k}$. It follows that $\tau_m.(\Phi_t^\ast f)=\tau_m.(\sum_{k=1}^\infty \Phi_t^\ast f_k)=\sum_{k=1}^\infty \tau_m.\Phi_t^\ast f_k=\sum_{k=1}^\infty \Phi_t^\ast (\tau_m.f_k)=\Phi_t^\ast f$. Similarly, $\sigma.\Phi_t^\ast f=\Phi_t^\ast f$. Therefore, $\Phi_t^\ast H^t_+=H^{t_0}_+$. By the same token, $\Phi_t^\ast H^t_-=H^{t_0}_-$. Consider an orthonormal basis of $H^t_2$, namely $\{h_{j},\sigma.h_{j}\}_{j=1}^\infty$, such that for each $j$, $\CC\{h_{j},\sigma.h_{j}\}\cong W_2$, and $\tau_m.h_j=e^{\frac{2\pi}{m}i}h_j$. Then for any $f\in H^{t_0}_{\pm}$, $\lan f,\Phi^\ast_t h_j\ran_{L^2}=\lan\tau_m.f,\tau_m.(\Phi^\ast_t h_j)\ran_{L^2}=e^{-\frac{2\pi}{m}i}\lan f,\Phi^\ast_t h_j\ran_{L^2}$. It follows that $\Phi^\ast_t H^t_2\perp (H^{t_0}_+\oplus H^{t_0}_-)$ and hence $\Phi^\ast_t H^t_2=H^{t_0}_2$.
\end{proof}

%For notation simplification, we denote $H^{t_0}_\pm$ as $H_\pm$ and $H^{t_0}_2$ as $H_2$ in the following discussion, for the moment. 

The perturbation theory of operators $\Delta_t$ on $L^2\big(\mathcal{I}_{\vp(t_0)}^{\CC}\big)$ can be reduced into the three spaces $H^{t_0}_\pm$ and $H^{t_0}_2$, which are fixed according to Lemma \ref{fixedspace}. The following lemma is a direct consequence of Proposition \ref{continuousordinary}.

\begin{lemma}
Consider $\lambda^{\ast}_k(t)$ as the $k$-th eigenvalue of $\Delta_t|_{H^{t_0}_\ast}$, where the superscript and subscript $\ast$ valued in $+,-$ or $2$. The function $\lambda_k^\ast(t)$ is continuous in $t$.
\end{lemma}

\begin{lemma}\label{monotonicityofpmlambda}
    Suppose $\lambda^{\pm}_k(t_0)$ is the $k$-th eigenvalue of $\Delta_{t_0}|_{H^{t_0}_\pm}$. Then for any $k\ge 1$, $\lambda^+_k(t)$ strictly increases while $\lambda_k^-(t)$ strictly decreases as $t$ varies.
\end{lemma}
\begin{proof}
   Let $U^{\pm}_{k,t}$ be the corresponding eigenspace of $\lambda^\pm_k(t)$ in $H^{t}_\pm$. By Lemma \ref{symmemulres} and \ref{UU'ordb}, we can decompose $U^{\pm}_{k,t}$ into $\oplus_{j=0}^{n-1} U_{k,t}^{\pm,j}$, where $U_{k,t}^{\pm,j}=\CC\{f_{k,t}^{\pm,j}\}$, with real sections $f_{k,t}^{\pm,j}$ satisfying $||f_{k,t}^{\pm,j}||_{L^2}=1$ and $\vn_{p_l}(f_{k,t}^{\pm,j})=j$ for $1\le l\le m$. 
   
   Now we can apply Lemma \ref{TWformula} to $\re\, U_{k,t}^{\pm}=\oplus_{j=0}^{n-1}\RR\{f_{k,t}^{\pm,j}\}$, and conclude that there are differentiable functions of $\ZT$ eigenvalues of $\lbt$, namely $\mu_i^\pm(t)$ for $1\le i\le N$ ($N=\dim_{\RR}\,\re\, U_{k,t}^{\pm}=\dim_{\CC}\, U_{k,t}^{\pm}$), such that $\frac{d}{dt}\mu_i^{\pm}(t)$ are eigenvalues of the Taubes-Wu bilinear form $B_{\vec{v}}$ on $U_{k,t}^{\pm}$, where $\vec{v}=\frac{d}{dt}\vp(t)$. 

   In the chosen complex coordinates $z_l$ near $p_l$ for $1\le l\le m$, $\vec{v}(p_l(t))=\frac{2}{1+t^2}\frac{\partial}{\partial z_l}+\frac{2}{1+t^2}\frac{\partial}{\partial \bar{z}_l}$, thus $v(p_l(t))=\frac{2}{1+t^2}$ in \eqref{TWblf}. And according to the last assertion of Lemma \ref{UU'ordb}, if $f_{k,t}^{\pm,0}$ is nontrivial, we have $\pm\frac{d}{dt}\mu_i^{\pm}(t)>0$ for some $i$. 

   Consequently, eigenvalues of $B_{\vec{v}}$ on $U_{k,t}^+$ are nonnegative, and one of them is positive if and only if $U_{k,t}^{+,0}$ is nontrivial. Similarly, those of $B_{\vec{v}}$ on $U_{k,t}^-$ are nonpositive, and one of them is negative if and only if $U_{k,t}^{-,0}$ is nontrivial.
   Therefore, $\lam_k^+(t)$ is nondecreasing while $\lam_k^-(t)$ is nonincreasing.

   Next we shall show that $\lam_k^\pm$ are strictly monotonic.
   
   Suppose for some $0<t_1<t_2<+\infty$, $\lam_k^+(t_1)=\lam_k^-(t_2)=:\lam$. Then $\lam_k^+(t)\equiv \lam$ on the interval $[t_1,t_2]$. Let $V_\lam^t$ be the eigenspace of eigenvalue $\lam$ in $\mathcal{L}_1^2\big(\lbct\big)$. By Lemma \ref{symmemulres}, we have 
   \[V_\lam^t=\oplus_{j=0}^{n-1}\big(U^{\oplus c_j(t)}\oplus (U')^{\oplus d_j(t)}\oplus W^{\oplus e_j(t)}\big).\]

   Suppose $c_0(t_3)=1$ for some $t_3\in(t_1,t_2)$. It follows from Lemma \ref{TWformula} and the discussion above that $\dim\, V_\lam^{t_3+\epsilon}\cap H_+^{t_3+\epsilon}\le \dim\,V_\lam^{t_3}\cap H_+^{t_3}-1$ for $\epsilon>0$ sufficiently small. Let $T_+=\{t\in(t_1,t_2): c_0(t)=1\}$. Then $T_+$ must be finite and $|T_+|<\dim V_\lam^{t_1}\cap H_+^{t_1}$, since $\dim\, V_\lam^{t_2}\cap H_+^{t_2}>0$.  

   Consider $V_\lam^t\cap H_-^t$, and assume for some $l$, $\lam_l^-(t)\equiv\lam$ for $t\in [t_1,t_2]$. A similar argument and application of Lemma \ref{TWformula} show that $T_-:=\{t\in (t_1,t_2):d_0(t)=1\}$ is a finite set.

   Now we can choose an interval $(t_4,t_5)\subset [t_1,t_2]$, such that $(t_4,t_5)\cap T_+=\varnothing$. According to Lemma \ref{symmemulres}, there must be $d_0(t)=1$ for any $t\in (t_4,t_5)$. However, this implies that $T_-$ contains an interval, yielding a contradiction.
\end{proof}

\subsubsection{Limit behavior of the spectrum}

We have introduced the eigenvalue functions $\lambda^\pm_k(t)$, where $\lambda^\pm_k(t)\in\mathfrak{S}^t_{\pm}$. According to Theorem \ref{continuousztdegenerate} and Lemma \ref{monotonicityofpmlambda}, the limits $\lim_{t\to 0}\lambda^\pm_k(t)$ and $\lim_{t\to +\infty}\lambda^\pm_k(t)$ exist. We denote these limits as $\lambda_k^\pm(0)$ and $\lambda_k^\pm(\infty)$ respectively.
Moreover, $\lambda^\pm_k(0)$ is a Laplacian eigenvalue of the round sphere $\sph$ while $\lambda^\pm_k(\infty)$ is a $\ZT$ eigenvalue of $\mathcal{I}_{\vq}$, where $\vq=\{0,\infty\}$ is the antipodal configuration. However, it's important to note that these eigenvalues may not correspond to the $k$-th eigenvalue of $\Delta$. We shall explore the locations of these limit eigenvalues in the Laplacian spectra of $\sph$ and $\mathcal{I}_{\vq}$.

Firstly, recall the following results about spherical harmonics (cf. \cite{quantum}):
\begin{example}\label{sphericalspectrum}
    The Laplacian spectrum of $\sph$ is $\{l(l+1):l=0,1,\cdots\}$. The multiplicity of $l(l+1)$ is $2l+1$, and a basis of the eigenspace $V_{l(l+1)}$ is given by $\{\cos (j\varphi) K_l^j(\theta),\;\sin (j\varphi) K_l^j(\theta):j=0,1\cdots,l\},$
where $K_l^j(\theta)=\sin^{-j}\theta \mathrm{d}^{l-j}(\sin^{2l}\theta)/\mathrm{d}(\cos\theta)^{l-j}$. 
\end{example}

The $D_{2m}$-action on $\sph$ induces actions on eigenspaces of $\Delta$ on the round sphere $\sph$ and line bundle $\mathcal{I}_{\vq}$. Indeed, Proposition \ref{antipodalspectrum} and Example \ref{sphericalspectrum} offer an orthogonal decomposition of irreducible $D_{2m}$-representations for each $V_{(l-\frac12)(l+\frac12)}$ or $V_{l(l+1)}$, with $D_{2m}$ acting on the coordinate $\varphi$. It is straightforward to check that $\tau_m. \cos (j\varphi) K_l^j(\theta)=\cos (j(\varphi+\frac{2\pi}{m})) K_l^j(\theta)$.

\begin{lemma}\label{D2mdecomoflimitspectra}
The complexified eigenspace, as a $D_{2m}$-representation have the following decomposition:
\begin{enumerate}[label=(\roman*)]
    \item Let $\lambda=(l-\frac12)(l+\frac12)$ be a $\ZT$ eigenvalue of $\mathcal{I}_{\vq}$, we write $l+1=mk+i$ for $i,k\in\mathbb{N}$ and $0\le i<m$, then $V_\lambda\otimes \CC\cong W^{\oplus l-k}\oplus (U\oplus U')^{\oplus k}$, where $U\cong \CC\{\cos( (j-\frac12)\varphi) G^j_l(\theta)\}$ and $U'\cong\CC\{\sin((j-\frac12)\varphi) G_l^j(\theta)\}$ for each $j\equiv 0\mod m$, $1\le j\le l$. Also, $W\cong \CC\{\cos( (j-\frac12)\varphi) G^j_l(\theta),\sin((j-\frac12)\varphi) G_l^j(\theta)\}$ for other $1\le j\le l$.
    \item  For $\lambda=l(l+1)$ be a Laplacian eigenvalue of the round sphere $\sph$, write $l=mk+i$, where $k,i\in\mathbb{N}$ and $0\le i<m$, then $V_\lambda\otimes \CC\cong U\oplus (U\oplus U')^{\oplus k}\oplus W^{\oplus 2l-k}$, where $U\cong\CC\{\cos (j\varphi) K^j_l(\theta)\}$ and $ U'\cong\CC\{\sin (j\varphi) K_l^j(\theta)\}$ for each $j\equiv 0\mod m$, $0\le j\le l$, and $W\cong\CC\{\cos (j\varphi) K^j_l(\theta),\sin (j\varphi) K_l^j(\theta)\}$ for other $0\le j\le l$.
\end{enumerate}
\end{lemma}

Consequently, on the spaces $L^2(\sph)\otimes \CC$ and $L^2(\mathcal{I}_{\vq}^{\CC})$, we have the decompositions $L^2(\sph)\otimes\CC=H_+^0\oplus H_-^0\oplus H_2^0$ and $L^2(\mathcal{I}_{\vq}^\CC)=H_+^\infty\oplus H_-^\infty\oplus H_2^\infty$, where each summand is defined in a way similar to those within the decomposition of $L^2\big(\lbct\big)$. And the spectra are decomposed into three part. For example, $\lambda_1^+(0)=0$ and $\lambda_1^-(0)=m(m+1)$, while $\lambda_1^\pm(\infty)=(m-\frac{3}{2})(m-\frac{1}{2})$. 

\begin{lemma}
    Consider the eigenvalue function $\lambda_k^\pm(t)\;(0<t<+\infty)$. The limits $\lambda^\pm(0)$ and $\lambda^\pm(\infty)$ are the $k$-th eigenvalue of $\Delta$ on $H_\pm^0$ and $H_\pm^\infty$ respectively.
\end{lemma}
\begin{proof}
    The proof is essentially the same with that of Theorem \ref{continuousztdegenerate}. The path $\vp(t)$ we consider here is invariant under the $D_{2m}$ action. And the cut-off function introduced in \eqref{cutoff} is also $D_{2m}$-invariant. Thus, for a family of eigensections $f_t$ of eigenvalue $\lam_k^\pm(t)$, sections $\chi_tf_t$ lie in $H_\pm^0$ and $H_\pm^\infty$ for $t\to 0$ or $t\to +\infty$ respectively. Here, $\chi_t=\chi_{\epsilon(t),\{p_{2n}\}}$ for $t\to 0$ and $\epsilon(t)\to 0$ such that $\vp(t)\subset B_{\epsilon(t)}(\{p_{2n}\})$; $\chi_t=\chi_{\epsilon(t),\vq}$ for $\epsilon(t)\to 0$ as $t\to+\infty$, such that $\vp(t)\subset B_{\epsilon(t)}(\vq)$.
    Thus the limits of $\chi_tf_t$ with $t\to0$ or $t\to +\infty$, are $k$-th eigensections of $\sph$ or $\mathcal{I}_{vq}$, respectively. 
\end{proof}

\begin{proposition}\label{howtheycross}
    Consider the $k$-th eigenvalue $\lambda_k^+(0)$ of $\Delta$ in $H^0_+$. Suppose $\lambda_k^+(0)=l(l+1)$, then $\lambda_k^+(\infty)= (l+m-\frac32)(l+m-\frac12)$. On the other hand, for $\lambda_k^-(0)=l'(l'+1)$, we have $\lambda_k^-(\infty)= (l'-\frac32)(l'-\frac12)$.
\end{proposition}
\begin{proof}
We prove this by induction. Firstly, it is known that $\lambda_1^+(0)=0$ and $\lambda_1^+(\infty)=(m-\frac{3}{2})(m-\frac12)$, as well as $\lambda_1^-(0)=m(m+1)$ and $\lambda_1^-(\infty)=(m-\frac{3}{2})(m-\frac12)$. Assume we have proved the proposition for $\lam_j^\pm$, $1\le j\le k-1$.

According to Lemma \ref{monotonicityofpmlambda}, the eigenvalue function $\lambda^+_k(t)$ strictly increases while $\lambda_k^-(t)$  strictly decreases as $t$ varies. Therefore, $\lambda^+_k(\infty)> l(l+1)$ while $\lambda^-_k(\infty)< l'(l'+1)$. This implies that $\lambda^+_k(\infty)\ge (l+m-\frac32)(l+m-\frac12)$ while $\lambda_k^-(\infty)\le (l'-\frac12)(l'+\frac12)$. 
    
If we write $l$ as $mk+i$, for $0\le i<m$, then $(l+m-1)+1=m(k+1)+i$. Thus $\dim_{\CC}\,H^{\infty}_+\cap V_{\lambda}=\dim_{\CC}\,H^{0}_+\cap V_{l(l+1)}$, where $\lambda=(l+m-\frac32)(l+m-\frac12)$. Since the proposition holds for any $\lam_j^+$ $(1\le j\le k-1)$, if $\lam_k^+(\infty)>(l+m-\frac32)(l+m-\frac12)$, there must be $\dim_{\CC}\,H^{\infty}_+\cap V_{\lambda}<\dim_{\CC}\,H^{0}_+\cap V_{l(l+1)}$, leading to a contradiction. 

Similarly, if we write $l'=mk'+i'$, then $(l'-1)+1=mk+i$, and hence $\dim_{\CC}\,H^{\infty}_+\cap V_{\lambda}=\dim_{\CC}\,H^{0}_+\cap V_{l'(l'+1)}$, where $\lambda=(l'-\frac32)(l'-\frac12)$. And $\lam_k^-(\infty)<(l'-\frac12)(l'+\frac12)$ leads to another contradiction.
\end{proof}

Now, we will discuss the possible intersection points of the spectral flow in a region $((l-\frac32)(l-\frac12),l(l+1))$.

Consider an eigenvalue $l(l+1)$ of $\sph$. Suppose $l+1=mk+i$, where $m=|\vp(t)|-1=2n-1$ and $k,i$ are nonnegative integers such that $0\le i<m$. Then by Lemma \ref{D2mdecomoflimitspectra}, there are $k$ copies of $U'$ within the irreducible $D_{2m}$-representation decomposition of the complexified eigenspace $V_{l(l+1)}\otimes\CC$. Consequently, for some $k_0\ge 1$, $\lambda_{k_0}^-(0)=\lambda_{k_0+1}^-(0)=\cdots=\lambda_{k_0+k-1}^-(0)=l(l+1)$. It follows from the Proposition \ref{howtheycross} that $\lambda_{k_0}^-(\infty)=\lambda_{k_0+1}^-(\infty)=\cdots=\lambda_{k_0+k-1}^-(\infty)= (l-\frac32)(l-\frac12)$.

On the other hand, consider $1\le j\le m-1$. For eigenvalues $(l-j)(l-j+1)$ of $\sph$, there are $k_j+1$ copies of $U$ in the decomposition of complexified eigenspaces $V_{(l-j)(l-j+1)}\otimes \CC$, where $k_j$ are positive integers satisfying $l-j=mk_j+i_j$ for $0\le i_j< m$. Let $n_{m-1}$ be the least integer that $\lam^+_{n_{m-1}}(0)=(l-m+1)(l-m+2)$. Define $n_j$ inductively by $n_j=n_{j+1}+k_{j+1}+1$, for $1\le j\le m-1$.

Now, we have $\lambda^+_{n_j}(0)=\cdots=\lambda^+_{n_j+k_j}(0)=(l-j)(l-j+1)$. It follows from Proposition \ref{howtheycross} that $(l-j+m-\frac32)(l-j+m-\frac12)=\lambda^+_{n_j}(\infty)=\cdots=\lambda^+_{n_j+k_j}(\infty)$ for $j=1,\cdots,m-1$.

Since $(l-j+m-\frac32)(l-j+m-\frac12)>(l-\frac32)(l-\frac12)$, for each pair of integers $(n_1,n_2)$ such that $k_0\le n_1\le k_0+k-1$ and $n_{m-1}\le n_2\le n_1+k_1$, there exists $t_{n_1,n_2}\in(0,+\infty)$ such that $\lambda^-_{n_1}(t_{n_1,n_2})=\lambda^+_{n_2}(t_{n_1,n_2})=:\lambda_{n_1,n_2}$. Define the set of cross times $T_l$ to be $\{t_{n_1,n_2}:(l-\frac32)(l-\frac12)<\lambda_{n_1,n_2}<l(l+1)\}$. According to the strict monotonicity of $\lam_k^\pm$, as proved in Lemma \ref{monotonicityofpmlambda}, $T_l$ must be discrete.

For general $n>1$, we have:

\begin{theorem}
    For each $t_{n_1,n_2}$ defined above, $\lambda_{n_1,n_2}$ is a critical eigenvalue of the line bundle $\mathcal{I}_{\vp(t_{n_1,n_2})}$. Moreover, we have a lower bound of the number of cross times of eigenvalues: $|T_l|\ge\sum_{j=1}^{m-1}2\frac{k_j}{n+2}$, 
     where $k_j~(j=1,\cdots,m)$ are the integers previously defined by $l-j=mk_j+i_j$. Consequently, in $\Cn$, there are more than $[\sum_{j=1}^{m-1}2k_j/(n+2)]$ critical configurations that with critical eigenvalues lying in the interval $\big((l-\frac{3}{2})(l-\frac{1}{2}),l(l+1)\big)$, where $m=2n-1$, $l$ is a positive integer.
\end{theorem}
\begin{proof}
    The fact that eigenvalue $\lambda_{n_1,n_2}$ is critical follows from Corollary \ref{crossingcri}. 
    
    Recall from Lemma \ref{symmemulres} that $\dim_\CC\,V_{\lambda}\cap H_+^t\le\frac{n}{2}+1$. As a consequence, we have $\dim_{\CC}V_{\lam_{n_1,n_2}}\cap H_+^{t_{n_1,n_2}}\le \frac{n}{2}+1.$
    
    It follows that \begin{align*}\sum_{j=1}^{m-1}k_j=\sum_{j=1}^{m-1}\dim_{\CC}\, V_{(l-j)(l-j+1)}\cap H_+^0
    =\sum_{t_{n_1,n_2}\in T_l}\dim_{\CC}V_{\lam_{n_1,n_2}}\cap H_+^{t_{n_1,n_2}}\le |T_l|\times (\frac{n}{2}+1).\end{align*}

    Here, the first identity holds due to the strict monotonicity of $\lam^\pm_k$. Therefore, $|T_l|\ge \sum_{j=1}^{m-1}2\frac{k_j}{n+2}$.
\end{proof}

The following figure illustrates the behavior of the spectrum for $\vp(t)\subset\mathcal{C}_4$, focusing only on several eigenvalues. The red circles highlight the cross times of eigenvalues corresponding to $U$ and $U'$, which signify critical configurations.

\begin{figure}[!h]
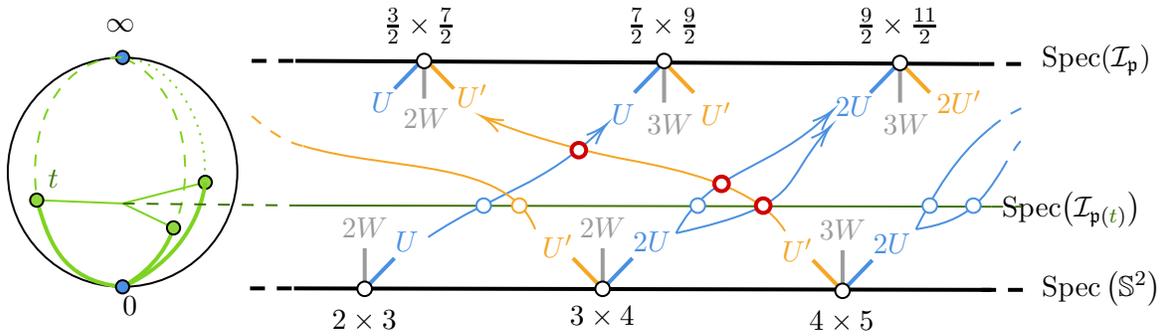

    \centering

\tikzset{every picture/.style={line width=0.75pt}} %set default line width to 0.75pt        



    \caption{Part of the spectral flow on $\mathcal{C}_4$.}
    \label{spectralflowonc4}
\end{figure}

As $\lim_{l\to\infty}|T_l|=\infty$, we conclude:

\begin{theorem}
There are infinitely many critical configurations in $\Cn$ for each $n\geq 1$.
\end{theorem}

Since $\{\vp(t):0<t<+\infty\}$ is relatively compact in $\ccn$, there exists a convergent sequence $\{t_i\}$ such that each $\vp(t_i)$ is critical, with critical eigenvalue $\lambda_i$, and $\lim_{i\to\infty}\vp(t_i)\in \ccn$. However, it's important to note that, in general, $\{\lambda_i\}_i$ is divergent.

Finally, we make some further remarks:
\begin{enumerate}[label=(\roman*)]
    \item When $n>2$, as the rotations in dihedral group always preserve $p_{2n}(t)$, it holds that $\vn_{p_{2n}}>1$. Consequently, critical eigensections constructed above will not be non-degenerate.
    \item When $n=2$, by Lemma \ref{npf==1}, for any critical configuration and eigenvalue we construct above, the corresponding critical eigensection must be non-degenerate. 
    
    \item When $n=2$, $\vp(t)$ is critical if and only if for some eigenvalue $\lambda$, $\dim_\CC\,V_\lambda\cap H^t_{+}=\dim_\CC\,V_\lambda\cap H^t_{-}=1$.
    \item As shown in Proposition \ref{Misolated}, the tetrahedral configuration, lying on the path $\vp(t)\subset \mathcal{C}_{4}$, is $M$-isolated for any $0<M<+\infty$. However, since number critical configurations lying on $\mathcal{C}_4$ are infinite, it's possible that the tetrahedral configuration is not isolated in $\mathcal{C}_4$.

    \item When $n=2$, combining Proposition \ref{rk1} with Lemma \ref{symmemulres}, we conclude that for any $t\in(0,+\infty)$, when considering the eigenvalue $\lambda$, the eigenspace $V_\lambda\subset \mathcal{L}_1^2\big(\lbct\big)$ satisfies $\dim_\CC\,V_\lambda\cap H^t_{\pm}=1$.  Equivalently, for any $\lambda_k^\pm(t)$, $\mul\,\lambda_k^\pm(t)\equiv1$. This serves as an intriguing example illustrating the phenomenon where eigenvalues within the same symmetry subspaces do not cross each other. For further details, we refer to \cite{Dihedral} and \cite[Appendix 10]{Arnold}. 
\end{enumerate}

\appendix

	\bibliographystyle{alpha}
	\bibliography{references}
\end{document}